\documentclass[11pt,reqno]{amsart}
\usepackage[margin = 3cm]{geometry}

\usepackage{color}
\usepackage{amsmath, amsthm, amssymb, mathrsfs}
\usepackage{amsfonts}
\usepackage[ansinew]{inputenc}
\usepackage[dvips]{epsfig}
\usepackage{graphicx}
\usepackage[english]{babel}
\usepackage{enumerate}
\usepackage{hyperref}
\theoremstyle{plain}
\newtheorem{thm}{Theorem}[section]

\newtheorem{lem}[thm]{Lemma}
\newtheorem{prop}[thm]{Proposition}

\newtheorem{defi}[thm]{Definition}

\theoremstyle{remark}
\newtheorem{rem}[thm]{Remark}

\numberwithin{equation}{section}

\newcommand{\Y}{\mathcal{Y}}
\newcommand{\de}{\partial}

\newcommand{\R}{\mathbb{R}}
\newcommand{\eps}{\varepsilon}

\newcommand{\N}{\mathbb{N}}

\newcommand{\mres}{\mathop{\hbox{\vrule height 7pt width .5pt depth 0pt
\vrule height .5pt width 6pt depth 0pt}}\nolimits}

\newcommand{\average}{{\mathchoice {\kern1ex\vcenter{\hrule height.4pt
width 6pt depth0pt} \kern-9.7pt} {\kern1ex\vcenter{\hrule
height.4pt width 4.3pt depth0pt} \kern-7pt} {} {} }}

\def\R{\mathbb{R}}

\begin{document}

\title[Lagrangian structure of the Vlasov--Poisson system in domains]{The Lagrangian structure of the Vlasov--Poisson system in domains with specular reflection}

\author{Xavier Fernández-Real}

\address{ETH Z\"{u}rich, Department of Mathematics, Raemistrasse 101, 8092 Z\"{u}rich, Switzerland}

\email{xavierfe@math.ethz.ch}


\begin{abstract}
In this work, we deal with the Vlasov--Poisson system in smooth physical domains with specular boundary condition, under mild integrability assumptions, and $d \ge 3$. We show that the Lagrangian and Eulerian descriptions of the system are also equivalent in this context by extending the recent developments by Ambrosio, Colombo, and Figalli to our setting. In particular, assuming that the total energy is bounded, we prove the existence of renormalized solutions, and we also show that they are transported by a weak notion of flow that allows velocity jumps at the boundary. Finally, we show that flows can be globally defined for $d = 3, 4$.
\end{abstract}

\maketitle
\tableofcontents 
\section{Introduction}
In the last three decades, there has been a growing interest in the existence of solutions to transport and continuity equations under weak regularity assumptions, motivated by physical models where, for example, one can only assume that the total energy of a system is bounded. 

The interest in the relation between continuity and transport equations and their Lagrangian structure arises in this setting. In the late 1980s, DiPerna and Lions in \cite{DL89, DL89b} introduced the notion of renormalized solutions and the idea of strong convergence of commutators. There, they proved existence, uniqueness, and stability of regular Lagrangian flows for Sobolev vector fields with bounded divergence. More than a decade later, Ambrosio extended this result to BV vector fields with bounded divergence, \cite{Amb04}, and in a more recent paper, Bouchut and Crippa were able to consider vector fields whose gradient is given by the singular integral of an $L^1$ function, \cite{BC13}. Finally, we also mention the work of Ambrosio, Colombo, and Figalli, \cite{ACF15}, where they prove the existence of a unique maximal regular flow under very mild assumptions on the vector field, including a local integrability assumption. (We refer the reader to \cite{AC14} for an extensive summary on the theory of continuity equations and Lagrangian flows with weak regularity assumptions on the vector fields.) 

In this work, we deal with the Vlasov--Poisson system, which when posed globally corresponds to a transport equation with vector field given by a singular integral. The Vlasov--Poisson system describes the evolution of a nonnegative charge density under the action of a self-induced electric field under no magnetic field. Given a charge density $f = f_t(x, v):(0, T)\times\R^d\times\R^d\to [0, \infty)$, that is, the density of particles at position $x$ with velocity $v$ at time $t$, the evolution of $f_t$ is given by the transport equation
\begin{equation}
\label{eq.vlasov}
\de_t f_t(x, v) +v\cdot \nabla_x f_t(x, v) + E_t(x) \cdot\nabla_v f_t(x, v) = 0,\quad\textrm{ in }\quad (0, T)\times\R^d\times\R^d.
\end{equation}
Here, $E_t(x)$ denotes the electric field, which is generated by the physical density of particles at the position $x$ at time $t$, $\rho_t(x) = \int f_t(x, v) \, dv$, through the Poisson equation, $-\Delta_x V_t(x) = \rho_t(x)$ and $E_t(x) = \nabla_x V_t(x)$, so that 
\begin{equation}
\label{eq.poisson}
E_t(x) = c_d\int_{\R^d}\rho_t(z) \frac{x-z}{|x-z|^d}\,dz,\quad\textrm{ in }[0, T)\times\R^d.
\end{equation}
Both equations \eqref{eq.vlasov} and \eqref{eq.poisson} form the Vlasov--Poisson system. In our work, we are interested in the Vlasov--Poisson system restricted to a physical domain $\Omega$ with perfect conductor boundary conditions, so that $x\in \Omega\subset\R^d$, $f:(0, T)\times\Omega\times\R^d\to[0,\infty)$, and \eqref{eq.vlasov} is fulfilled in the interior of the domain.  The perfect conductor boundary conditions translate into an electric field with no tangential components at the boundary, or equivalently, the boundary of the domain is grounded (that is, $V_t|_{\de\Omega} = 0$); see also \cite{Guo93}. Then, the potential is computed as the solution to the Poisson equation with zero Dirichlet conditions: $-\Delta_x V_t(x) = \rho_t(x)$ for $x\in \Omega$ and $V_t(x) = 0$ on $\de\Omega$. The electric field is now given by 
\begin{equation}
\label{eq.poisson2}
E_t(x) = c_d\int_{\R^d}\rho_t(z) \nabla_x G(x, z) \, dz,\quad\textrm{ in }[0, T)\times\Omega,
\end{equation}
where $G(x, z)$ denotes the Green function of the domain. To complete the description of our system, we still need to set a boundary condition. Since we will be interested in the Lagrangian structure of the solution, we set a boundary condition given by pure specular reflection, so that we have 
\begin{equation*}
\label{eq.vlasov2}
f_t(x, v) = f_t(x, v - 2 (v\cdot n(x) ) n(x)),\quad  \textrm{ on } (0, \infty)\times\de\Omega\times\R^d,
\end{equation*}
where $n(x)$ denotes the inward unit normal vector at $x\in \de\Omega$.

The Vlasov--Poisson system has been well-studied. Existence and regularity of classical solutions for $\Omega = \R^d$ were obtained under different assumptions in \cite{Hor82, Hor82b, BD85, Sch87, Wol93}. In dimension 3, the existence and uniqueness of classical solutions under a smooth initial datum is a classical result, \cite{Pfa92, LP91}. In general, though, these first results required integrability hypothesis on the initial datum. These integrability hypothesis are needed even to define a weak solution, as one a priori needs $E_t f_t$ to be in $L^1$. In \cite{DL88}, DiPerna and Lions, with the introduction of new concepts, were able to define a new notion of solution and prove its global existence under some mild initial integrability condition and boundedness of the total energy.  

More recently, Ambrosio, Colombo, and Figalli, in \cite{ACF17}, were remarkably able to reduce the initial condition integrability to merely $L^1$ to prove existence, and they also establish that the Lagrangian structure associated to the transport equation in the Vlasov--Poisson system is still valid (in some weaker sense, developed in \cite{ACF15}) for renormalized solutions. That is, they are able to prove the equivalence between the Lagrangian and Eulerian structure in the setting of renormalized solutions. 

Yet, if $\Omega\neq \R^d$ much less is known, even for classical solutions. In \cite{Guo95}, Guo shows that derivatives of solutions cannot be uniformly bounded near the boundary. Classical solutions exist under suitable conditions on the initial datum near the boundary and the domain for both specular and absorption boundary conditions, as shown in \cite{Guo94, Hwa04, HV09, HV10}. On the other hand, in \cite{Guo93}, Guo shows the existence of weak solutions for the more general setting of the Vlasov--Maxwell system in 3 dimensions, with smooth domains (see also \cite{Ars79, Ale93, Abd94, Wec95} for existence and stability of weak solutions). Closer to the weak notion of solution that is considered in this paper, we refer the reader to the work by Mischler \cite{Mis99}, where the author is able to establish existence and weak regularity results for the Vlasov equation in domains with specular reflection, under weak assumptions on a fixed electric field and source term. We also refer to \cite{Sku13},\cite{Guo94}, and \cite{GvP87}, and references therein, for a better understanding of the state-of-the-art of the problem and applications.

The results by Ambrosio, Colombo, and Figalli in \cite{ACF17}  provide a general Lagrangian structure for weak/renormalized solutions. In the present work, we extend their results to the context of smooth physical domains with specular reflection. In particular, we introduce the notion of renormalized solution in a domain and an analogous definition of Maximal Regular Flow that allows jumps in the velocity component, to account for the specular reflection condition. We show that this notion of flow transports renormalized solutions, and, in particular, we show that renormalized and Lagrangian solutions are also equivalent in this new situation. 

Thanks to the generality of the techniques introduced in \cite{ACF15, ACF17}, we can deal with the problem at a local level, and establish that solutions evolve following certain flows far from the boundary. In order to deal with the boundary region, we note that the existence of flows coming from vector fields given by singular integrals is quite rigid, in the sense that, a priori, the theory is not valid for vector fields of the form \eqref{eq.poisson2}, where the Calder\'on--Zygmund theory does not directly apply. Therefore, one cannot immediately use the results from \cite{BBC16}. In order to establish the existence of flows around the boundary, we first show that in small balls centered at the boundary the domain and the electric field are close to the half-space situation. Hence, we reduce the problem, after some fine estimates, to the half-space case. 

For the half-space problem we note that, after a reflection, the electric field has an expression similar to a convolution against an $L^1$ function, and we prove that this is enough to get our result. 
\vspace{0.3cm}

We divide the work as follows. In Section \ref{sec.intro}, we state our results and introduce the problem and main definitions. In Section \ref{sec.vphs}, we study the Vlasov--Poisson system in the half-space, constructing a new problem in the whole space based on the image charge method. In Section~\ref{sec.mrf}, we recall the notion of Maximal Regular Flow, we introduce the notion of Maximal Specular Flow, and prove some equivalences with the half-space problem. In Sections~\ref{sec.PbB} and \ref{sec.mainres}, we prove our main results in the half-space situation, in particular showing the existence of a renormalized solution. Finally, in Sections~\ref{sec.gendom}, \ref{sec.uniq}, and \ref{sec.Main}, we show our main result for the general domains situation. We emphasize that we treat the half-space and the general domain situation separately, as we believe that the half-space approach in a non-local way has its own interest. 

We would like to mention that we deal with the repulsive case; that is, the interaction force between particles is repulsive, and can be used to model, for example, the evolution of positively (or negatively) electrically charged particles under their self-consistent electric field. In the attractive case, the sign of the electric field changes, and most of our results are still valid up to minor modifications of the proofs. As in \cite{ACF17}, however, one would need to consider an effective density instead of the physical one, as we can no longer control the initial energy and some particles might have infinite velocity and disappear from the physical space. Similarly, minor modifications of the current work would also allow to obtain analogous results for the relativistic Vlasov--Poisson system in domains. 
\vspace{0.3cm}

{\bf Acknowledgements:} The author would like to thank Alessio Figalli for his guidance, patience, and useful discussions on the topics of this paper. 

The author acknowledges support of the ERC grant ``Regularity and Stability in Partial Differential Equations (RSPDE)"

\section{Main results}
\label{sec.intro}
Let $\Omega\subset\R^d$ be a $C^{2, 1}$ domain, with $d\ge 3$. We want to study the evolution of a distribution function $f_t(x, v) = f:(0, \infty)\times \Omega\times\R^d$ under a self-consistent electric field generated by the Vlasov--Poisson system, assuming a grounded boundary, (i.e., a zero Dirichlet condition). We will also impose a specular reflection condition on the boundary. 

For any point $x\in \de\Omega$, we will denote by $n(x) = n:\de\Omega\to S^{d-1}$ the \emph{inward-pointing} unit normal vector to $\de\Omega$ at the point $x$ (we remark that for future convenience we have used the non-standard convention of taking the \emph{inward} unit normal). We denote, for each $x\in\de\Omega$, the reflection operator at $x$ as $R_x:\R^d\to\R^d$. With this convention, given a point $x\in\de\Omega$, the reflection of a velocity $v\in\R^d$ is given by 
\begin{equation}
\label{eq.reflv}
R_x v = v - 2 (v\cdot n(x) ) n(x)
\end{equation}

The problem, then, is the following,
\begin{equation}
\label{eq.original_Om}
  \left\{ \begin{array}{ll}
  \de_t f + v\cdot \nabla_x f_t + E_t\cdot\nabla_v f_t=0 & \textrm{ in } (0,\infty)\times\Omega\times\R^d\\
  f_t(x, v) = f_t(x, R_x v) & \textrm{ on } (0, \infty)\times\de\Omega\times\R^d.
  \end{array}\right.
\end{equation}

We define the physical density for $t \ge0 $, $x\in\Omega$, by $\rho_t(x) = \int_{\R^d} f_t(x, v) dv$. The electric field is then given by $E_t(x) = -\nabla_x V_t$, where the potential $V_t$ is the solution to the following Laplace equation with zero Dirichlet conditions,
\begin{equation}
\label{eq.elfield}
  \left\{ \begin{array}{rcll}
  - \Delta_x V_t(x) &= &\rho_t(x) & \textrm{ in } (0,\infty)\times\Omega,\\
  V_t(x) &= &0 & \textrm{ on } (0, \infty)\times\de\Omega,\\
  \lim_{|x|\to \infty} V_t(x) &=& 0 & \textrm{ for } t\in(0, \infty).  
  \end{array}\right.
\end{equation}

The solution to the previous problem can be obtained by means of the Green function of the domain. Thus, if we denote by $G_\Omega(x_1, x_2):\overline{\Omega}\times\overline{\Omega}\to \R_{\ge 0}$ the Green function of $\Omega$, then
\begin{equation}
\label{eq.greenv}
V_t(x) = \int_{\Omega} G_\Omega(x, z) \rho_t(z)\,dz.
\end{equation} 

Thus, we can rewrite our problem as
\begin{equation}
\label{eq.original_Om_2}
  \left\{ \begin{array}{ll}
  \de_t f + v\cdot \nabla_x f_t + E_t\cdot\nabla_v f_t=0 & \textrm{ in } (0,\infty)\times\Omega\times\R^d\\
  \rho_t(x)=\int_{\R^d} f_t(x, v) dv & \textrm{ in } (0,\infty)\times\Omega\\
  E_t(x) = -\int_{\Omega} \nabla_x G_\Omega(x, z)\rho_t(z)\,dz & \textrm{ in } (0,\infty)\times\Omega \\
  f_t(x, v) = f_t(x, R_x v) & \textrm{ on } (0, \infty)\times\de\Omega\times\R^d.
  \end{array}\right.
\end{equation}

The transport structure is clear, as the problem can be written in the simpler way 
\begin{equation}
\label{eq.VPOm}
\left\{ \begin{array}{ll}
  \de_t f_t + b_t\cdot \nabla_{x,v} f_t=0 & \textrm{ in } (0,\infty)\times\Omega\times\R^d\\
  f_t(x, v) = f_t(x, R_x v) & \textrm{ on } (0,\infty)\times\de\Omega\times\R^d,
  \end{array}\right.
\end{equation}
where the vector field $b_t(x, v) = (v, E(x)):\Omega\times\R^d\to \R^{2d}$ is divergence-free and coupled to $f_t$ via \eqref{eq.original_Om_2}.  

This equation can be reinterpreted in a distributional sense when $b_t f_t$ is in $L^1_{\rm{loc}}$. However, this is not true in general, and one needs to introduce the concept of renormalized solution, \cite{DL88}. This is based on the fact that, in the smooth case, one can consider $C^1\cap L^\infty$ functions $\beta:\R\to \R$ and $\beta(f_t)$ still solves the equation. That is, using that the vector field is divergence-free,
\begin{equation}
\label{eq.VPHS_rn_D}
\left\{ \begin{array}{ll}
  \de_t \beta(f_t) + {\rm div}_{x,v} (b_t \beta(f_t))=0 & \textrm{ in } (0,T)\times\Omega\times\R^d\\
  \beta(f_t)(x, v) = \beta(f_t)(x, R_x v) & \textrm{ on } (0,T)\times\de\Omega\times\R^d,
  \end{array}\right.
\end{equation}

Notice also that, in general, one still needs in \eqref{eq.VPHS_rn_D} to make sense of the trace of $\beta(f_t)$ on $(0,T)\times\de\Omega\times\R^d$. Let us define some useful sets that will appear recurrently throughout the work. 

\begin{defi}
\label{defi.spaces_Om}
We define the following sets
\begin{align*}
\gamma^{\pm}_{\Omega, T} & = \{(t, x, v)\in (0, T)\times\de\Omega\times\R^d : \pm v\cdot n(x) >  0\},\\
\gamma^{0}_{\Omega, T} & = \{(t, x, v)\in (0, T)\times\de\Omega\times\R^d : v\cdot n(x)  = 0\},\\
\gamma_{\Omega, T} & = (0,T)\times\de\Omega\times\R^d.
\end{align*}
We will also denote $d\sigma_x^\Omega$ the standard surface measure of $\de\Omega$. 

Also, given $T > 0$, we consider the following function space,
\begin{equation}
\label{eq.testf_Om} 
\begin{split}
\mathcal{T}_{\Omega, T} = \{\phi & \in C^1_c([0, \infty) \times \R^{2d}) :  {\rm supp}~ \phi \Subset \{[0, T) \times \overline{\Omega}\times\R^d \}\setminus\{( 0 \times \de\Omega\times \R^d )\cup \gamma^0_{\Omega,T} \} \}.
\end{split}
\end{equation}
Finally, we will drop the subscript $T$ to denote $T= \infty$, and drop the subscript (or superscript) $\Omega$ to denote $\Omega = \R^d_+ = \{x_1 > 0\}$ (e.g., $\mathcal{T}_T := \mathcal{T}_{\R^d_+, T}$, $\gamma^0_\Omega := \gamma^0_{\Omega, \infty}$, $\gamma^\pm := \gamma^\pm_{\R^d_+, \infty}$). 
\end{defi}

Throughout the work we will usually avoid the set $\gamma_{\Omega, T}^0$, which corresponds to velocities tangent to the boundary, as in \eqref{eq.testf_Om}. We will do that for mathematical convenience (same as in \cite{Guo93}). Notice that, in this set, the reflection property does not have any effect, and also notice that the set is lower dimensional in the $v$ space for each fixed $x\in \de\Omega$. 

We can now define what it means to be a renormalized solution to \eqref{eq.VPOm}.  A similar definition appears in \cite{Mis99} and in \cite{Guo93}.

\begin{defi}[Renormalized solution in a domain with specular reflection]
\label{defi.HS_Om}
Let $T > 0$, let $\Omega$ be a $C^{1, 1}$ domain and let $b\in L^1_{\rm{loc}} ([0, T]\times \Omega\times \R^{d}; \R^{2d})$, be a divergence-free vector field. A Borel function $f_t\in L^1([0,T]\times\Omega \times\R^{d})$ is a \emph{renormalized solution} of \eqref{eq.VPOm} starting from $f_0$ if for every $\beta\in C^1\cap L^\infty(\R)$ there exists $f_{\beta,t}^+\in L^\infty_{\rm loc}(\gamma^+_{\Omega,T})$ such that for every $\varphi \in \mathcal{T}_{\Omega,T}$,
\begin{equation}
\label{eq.defi_HS_Om}
\begin{split}
 \int_{ \Omega\times \R^{d}} \varphi_0(x, v) \beta(f_0(x, v)) \,dx\, dv\,  + &\int_0^T \int_{\Omega\times\R^{d}} [\de_t \varphi_t(x, v) + \nabla_{x, v} \varphi_t(x, v)\cdot  b_t(x, v) ] \beta(f_t(x, v))\,dt \,dx\,dv 
\\ & ~~~~ + \int_{\gamma^+_{\Omega,T}} v\cdot n(x) \big(\varphi_t( x, v) - \varphi_t(x, R_x v)\big) f^+_{\beta,t}(x, v)\, dt\,d\sigma^\Omega_x\,dv = 0.
\end{split}
\end{equation}

A Borel function $f_t \in L^\infty([0, T]\times\Omega\times\R^d)$ is a \emph{distributional solution} to \eqref{eq.VPOm} if it is a renormalized solution with $\beta$ fixed as any smooth bounded continuation of $\beta(x) = x$ for $|x|\leq \|f_t\|_{L^\infty}$.
\end{defi}

\begin{rem}
\label{rem.defi.HS}
The previous definition is equivalent to asking $\beta(f_t)$ to solve the Vlasov--Poisson equation in the Green formula sense, and the reflection property in the distributional sense. That is, we equivalently say that $f_t\in L^1([0,T]\times\Omega \times\R^{d})$ is a \emph{renormalized solution} of \eqref{eq.VPOm} starting from $f_0$ if for every $\beta\in C^1\cap L^\infty(\R)$ there exists $f_{\beta,t}^+\in L^\infty_{\rm loc}(\gamma^+_{\Omega,T}\cup\gamma^-_{\Omega,T})$ such that for every $\varphi \in \mathcal{T}_{\Omega,T}$
\[
\begin{split}
 \int_{ \Omega\times \R^{d}} \varphi_0(x, v) \beta(f_0(x, v))\,dx\,dv  + \int_0^T \int_{\Omega\times\R^{d}}& [\de_t \varphi_t(x, v) + \nabla_{x, v} \varphi_t(x, v)\cdot  b_t(x, v) ] \beta(f_t(x, v))dt \,dx\,dv 
\\  &~~~ + \int_{\gamma^+_{\Omega,T}\cup\gamma^-_{\Omega,T}} v\cdot n(x)\, \varphi_t( x, v) f^+_{\beta,t}(x, v) dt\,d\sigma_x^\Omega \,dv = 0,
\end{split}
\]
and 
\[
\int_{\gamma^-_{\Omega,T}} v\cdot n(x)\,\varphi_t( x, v) \left(f^+_{\beta,t}(x, v) -f^+_{\beta,t}(x, R_x v) \right) \,dt\,d\sigma_x^\Omega\,dv = 0.
 \]
In the following proofs we will be using both definitions.
\end{rem}

\begin{rem}
\label{rem.nondivfree}
We are dealing with divergence free vector fields, $b_t(x, v) = (v, E(x))$. In the more general case where $f_t$ solves \eqref{eq.VPOm} but with a non-divergence free vector field $b_t(x, v)$, the previous definition should be modified in the following way:
A Borel function $f_t\in L^1([0,T]\times\Omega \times\R^{d})$ is a \emph{renormalized solution} of \eqref{eq.VPOm} with a general vector field $b\in L^1_{\rm{loc}} ([0, T]\times \Omega\times \R^{d}; \R^{2d})$ starting from $f_0$ if for every $\beta\in C^1\cap L^\infty(\R)$ there exists $f_{\beta,t}^+\in L^\infty_{\rm loc}(\gamma^+_{\Omega,T})$ such that for every $\varphi \in \mathcal{T}_{\Omega,T}$, \eqref{eq.defi_HS_Om} holds with the following right-hand side,
\[
-\int_0^T \int_{\Omega\times\R^{d}} \varphi_t(x, v) \beta(f_t(x, v))\,{\rm div}_{x, v} b_t(x, v)  \,dt\,dx\,dv.
\]
Notice that, in general, for this definition to make sense one needs ${\rm div}_{x, v} \,b_t \in L^1_{\rm loc}((0, T)\times\Omega\times \R^{d})$.
\end{rem}

\begin{defi}[Commutativity property]
\label{defi.comm2}We say that a renormalized solution in a domain with specular reflection (Definition~\ref{defi.HS_Om}) fulfils the \emph{commutativity property} if there exists a measurable function $f^+_t: \gamma_{\Omega,T}^+\to [0, \infty)$ such that
\[
f^+_{\beta, t} = \beta(f^+_t),
\]
for all $\beta\in C^1\cap L^\infty(\R)$.
\end{defi}

\begin{rem}
It is not a priori true that the trace of renormalized solutions can be taken in the strong sense, i.e., that the commutativity property holds for renormalized solutions. A counterexample can be found, for example, in \cite[Remark 25]{AC14} for traces taken in the temporal domain. Nevertheless, this will become true in this case once we introduce the theory of Lagrangian solutions.
\end{rem}

Our first main result establishes that renormalized solutions are, in fact, transported by a suitable notion of flow. That is, Eulerian solutions are Lagrangian. 

\begin{thm}
\label{thm.main1_Om}
Let $T > 0$, $\Omega\subset\R^d$ a $C^{2, 1}$ domain, and $f_t \in L^\infty((0, T); L^1_+(\Omega\times\R^d))$, a weakly continuous function. Suppose that 
\begin{enumerate}[(i)]
\item either $f_t \in L^\infty((0, T); L^\infty(\Omega\times\R^d))$ is a distributional solution to Vlasov--Poisson system in $\Omega$ with specular reflection, \eqref{eq.original_Om_2};
\item or $f_t$ is a renormalized solution of the Vlasov--Poisson system in $\Omega$ with specular reflection, \eqref{eq.original_Om_2} (see Definition~\ref{defi.HS_Om}). 
\end{enumerate}
Then $f_t$ is a Lagrangian solution transported by the Maximal Specular Flow in $\Omega\times\R^d$ associated to the vector field $b_t$ (see Definition~\ref{defi.msf_dom}); in particular, $f_t$ is renormalized (according to Definition~\ref{defi.HS_Om}), and fulfils the commutativity property (Definition~\ref{defi.comm2}).
\end{thm}

Our second main result establishes the existence of a renormalized (and thus, by the previous theorem, Lagrangian) solution to \eqref{eq.original_Om_2} for a weak initial value $f_0 \in L_+^1(\Omega\times\R^d)$, where $L_+^1$ denotes the space of nonnegative functions in $L^1$.

\begin{thm}
\label{thm.main2_d}
Let $\Omega\subset\R^d$ be a $C^{2, 1}$ domain. Consider $f_0 \in L_+^1(\Omega\times\R^d)$, $\rho_0 (x) = \int_{\R^d} f_0(x, v) dv$, satisfying the finite initial energy condition,
\begin{equation}
\label{eq.hyp_main_d}
\int_{\Omega\times\R^d} |v|^2 f_0(x, v)\, dx\, dv + \int_{\Omega\times\Omega}G_\Omega(x, z)\,\rho_0 (x)\,\rho_0(z)\,dx\,dz < \infty.
\end{equation}
Then, there exists a global Lagrangian solution of the Vlasov--Poisson system in the domain $\Omega$ with specular reflection \eqref{eq.original_Om} with initial datum $f_0$, $f_t \in C([0, \infty) ; L^1_{\rm loc}(\Omega\times\R^d)) $, transported by the Maximal Specular Flow in the domain $\Omega\times\R^d$ (see Definition~\ref{defi.msf_dom}); which is renormalized (Definition~\ref{defi.HS_Om}) and fulfils the commutativity property (Definition~\ref{defi.comm2}). 

Moreover, the physical density $\rho_t = \int f_t dv$ and the electric field $E_t$ are strongly continuous in $L^1_{\rm loc}(\Omega)$; $\rho_t, E_t \in C([0, \infty) ; L^1_{\rm loc}(\Omega))$.
\end{thm}

Finally, we show that the initial energy bounds the total energy at all times, and thanks to that the flow is actually globally defined in dimensions $d = 3, 4$.

\begin{thm}
\label{thm.main3_D}
Let $\Omega\subset\R^d$ be a $C^{2, 1}$ domain. Consider $f_0 \in L_+^1(\Omega_+\times\R^d)$, $\rho_0 (x) = \int_{\R^d} f_0(x, v) dv$, satisfying the finite initial energy condition, \eqref{eq.hyp_main_d}. Let $f_t \in C([0, \infty) ; L^1_{\rm loc}(\Omega\times\R^d)) $ be the solution to the Vlasov--Poisson with specular reflection \eqref{eq.original_Om} with initial datum $f_0$ from Theorem~\ref{thm.main2_d}. Let $\rho_t = \int f_t dv$ be the physical density. Then, the following properties hold:
\begin{enumerate}[(i)]
\item for every $t \geq 0$, the energy is bounded by the initial energy,
\begin{equation}
\label{eq.boundenergy__D}
\begin{split}
&\int_{\Omega\times\R^d} |v|^2 f_t\, dx\, dv + \int_{\Omega\times\Omega}G_\Omega(x, z)\rho_t(x) \rho_t(z)\,dx\,dz \leq \\
&~~~~~~~~~~~~~~~~~~~~~~~~~~~~~~~~\le \int_{\Omega\times\R^d} |v|^2 f_0\, dx\, dv + \int_{\Omega\times\Omega}G_\Omega(x, z)\rho_0(x) \rho_0(z)\,dx\,dz ;
\end{split}
\end{equation}
\item if $d = 3, 4$, then the flow is globally defined on $[0, \infty)$ for $f_0$-a.e. $(x, v) \in \Omega\times \R^{d}$; that is, trajectories do not blow up in finite time, and $f_t$ is the image of $f_0$ through an incompressible flow. In particular, the map
\[
t \mapsto \int_{\Omega\times\R^d} \varphi(f_t(x, v)) \, dx\, dv,
\]
is constant for any Borel function $\varphi:[0,\infty)\to [0, \infty)$. 
\end{enumerate} 
\end{thm}

\subsection{Notation} Throughout the whole work we will denote $\R^d_+ := \{x\in\R^d : x_1 > 0\}$, $\overline{\R^d_+} := \{x\in\R^d : x_1 \geq 0\}$ and $\de\R^d_+ := \{x\in\R^d : x_1 = 0\}$ (and analogously with $\R^d_-$). We also denote the application 
\begin{equation}
':\R^d \to \R^d,\quad (x_1,x_2,\dots,x_d)' \mapsto (-x_1, x_2, \dots,x_d),
\end{equation}
corresponding to switching the first coordinate (notice that it is an involution, $x'' = (x')' =x$). 

We will say that a function $f = f(x,v )$ is even (resp. odd) with respect to $(x_1, v_1)$ if $f(x, v) = f(x', v')$ (resp. $f(x, v) = -f(x', v')$). We will similarly say that $f = f(x)$ is even (resp. odd) with respect to $x_1$ or $\{x_1 = 0\}$ if $f(x) = f(x')$ (resp. $f(x) = -f(x')$ ). 

\section{Vlasov--Poisson in the half-space}
\label{sec.vphs}
In the following sections we consider a simpler version of the problem in hand. In particular, we will be considering the half-space case, that is, $\Omega = \R^d_+$.

By doing so, we explore the main ideas of the final result and deduce some intermediate lemmas that will be very useful in the final proof. We also follow a slightly different approach by directly dealing with a problem in the whole domain, that has interest in its own. 

Our problem considers the case of a grounded boundary, and in the half-space the field $E$ is found through the method of image charges imposing a zero potential condition on $x_1 = 0$ (given that we assume a perfect boundary conductor; see also \cite{Guo93}). The image charges method is what inspired the choice of the extended problem we consider in this section. When dealing with general domains, we will proceed with a change of variables to reduce the problem to a half-space situation locally. 

We consider the $d$-dimensional Vlasov--Poisson system in the half-space: the evolution of a distribution function $f_t(x, v) = f:(0,\infty)\times\overline{\R_+^d}\times\R^d \to [0,\infty)$ under the self-consistent field generated by the Poisson's equation in the half-space with grounded boundary. We also impose a boundary condition given by specular reflection. The problem looks as follows, where now the Green function of the domain is explicit,
\begin{equation}
\label{eq.original}
  \left\{ \begin{array}{ll}
  \de_t f + v\cdot \nabla_x f_t + E_t\cdot\nabla_v f_t=0 & \textrm{ in } (0,\infty)\times\R_+^d\times\R^d\\
  \rho_t(x)=\int_{\R^d} f_t(x, v) dv & \textrm{ in } (0,\infty)\times\overline{\R_+^d}\\
  E_t(x) = -\nabla_x V_t & \textrm{ in } (0,\infty)\times\R_+^d \\
  V_t = c_d (d-2)^{-1}\int_{\R^d_+} \rho_t(y) \left(|x-y|^{2-d}-|x'-y|^{2-d}\right)dy & \textrm{ in } (0,\infty)\times\R_+^d\\
  f_t(x, v) = f_t(x, v') & \textrm{ in } (0, \infty)\times\de\R^d_+\times\R^d,\\
  \end{array}\right.
\end{equation}
where $c_d > 0$ is a dimensional constant such that $c_d {\rm div}\left(x|x|^{-d}\right) = \delta_0$. 

We will start by proving the statement of our main results, Theorems~\ref{thm.main1_Om} and \ref{thm.main2_d}, for the half-space. 

\begin{thm}
\label{thm.main1}
Theorem~\ref{thm.main1_Om} holds for $\Omega = \R^d_+$.
\end{thm}

\begin{thm}
\label{thm.main2}
Theorem~\ref{thm.main2_d} holds for $\Omega = \R^d_+$.
\end{thm}

As already mentioned, we want to reduce the problem to a whole space situation. This will be accomplished by a symmetrisation. The aim of this section is to prove that \eqref{eq.original} (or alternatively, Problem A) can be equivalently stated as Problem B (see below). 
\\[0.1cm]
{\bf Problem A:} This problem corresponds to \eqref{eq.original}.
\begin{equation}
\label{eq.A}
  \left\{ \begin{array}{ll}
  \de_t f_t + v\cdot \nabla_x f_t + E_t\cdot\nabla_v f_t=0 & \textrm{ in } (0,\infty)\times\R_+^d\times\R^d\\
  \rho_t^o(x)=\int_{\R^d} f_t(x, v) dv & \textrm{ in } (0,\infty)\times\overline{\R_+^d}\\
  \rho_t^o(x)=-\rho_t^o(x')& \textrm{ in } (0,\infty)\times\R_-^d\\
  E_t(x) = c_d \int_{\R^d} \rho_t^o(y) \frac{x-y}{|x-y|^d}dy = \rho_t^o * K& \textrm{ in } (0,\infty)\times\R_+^d \\
  f_t(x, v) = f_t(x, v') & \textrm{ in } (0, \infty)\times\de\R^d_+\times\R^d,\\
  \end{array}\right.
\end{equation}
where we have expressed the electric field as a single convolution of a new density $\rho^o_t$ against the kernel $K = \frac{x}{|x|^d}$. As defined, $\rho_t^o$ is the odd reflection of $\rho_t$ with respect to $\{x_1 = 0\}$.
\\[0.1cm]
{\bf Problem B:} In the spirit of the previous definition, we define a new problem in the whole $\R^d$. We consider the evolution of a function $g_t(x, v) = g: (0,\infty)\times\R^d\times\R^d \to [0,\infty)$, even with respect to $(x_1, v_1)$, via the following transport problem,
\begin{equation}
\label{eq.B}
  \left\{ \begin{array}{ll}
  \de_t g_t + v\cdot \nabla_x g_t + \tilde E_t\cdot\nabla_v g_t=0 & \textrm{ in } (0,\infty)\times\R^d\times\R^d,\\
  \tilde \rho_t(x)=\int_{\R^d} g_t(x, v) dv & \textrm{ in } (0,\infty)\times\overline{\R_+^d},\\
  \tilde \rho_t(x)=-\int_{\R^d} g_t(x, v) dv & \textrm{ in } (0,\infty)\times\R_-^d,\\
  \tilde E_t(x) = c_d \int_{\R^d} \tilde \rho_t(y) \frac{x-y}{|x-y|^d}dy & \textrm{ in } (0,\infty)\times\overline{\R_+^d}, \\
  \tilde E_t(x) = - c_d \int_{\R^d} \tilde \rho_t(y) \frac{x-y}{|x-y|^d}dy & \textrm{ in } (0,\infty)\times\R_-^d, \\ 
    g_t(x, v) = g_t(x', v') & \textrm{ in } (0,\infty)\times\R^d\times\R^d.\\
  \end{array}\right.
\end{equation}

This problem is posed so that any solution induces an electric field $\tilde E$ such that $\tilde E_1$ is odd with respect to $x_1$, and $\tilde E_i$ for $i \ge 2$ is even with respect to $x_1$; that is, 
\begin{equation}
(\tilde E_t(x))' = \tilde E_t(x'), \quad \textrm{in} \quad (0, \infty)  \times \R^d.
\end{equation}
Similarly, given a solution with even initial datum such that the electric field is odd in the first component and even in the others, then the solution must be even. This should reflect the specular reflection property from the previous problem.

We want to see that solutions to Problem~A in \eqref{eq.A} correspond to solutions of Problem~B in \eqref{eq.B}, and vice-versa. Notice, firstly, that Problem B also presents a transport structure (similarly to Problem A via \eqref{eq.VPOm}) given by 
\begin{equation}
\label{eq.PbB}
\de_t g_t + \tilde b_t\cdot \nabla_{x,v} g_t=0.
\end{equation}
where vector field $\tilde b_t(x, v) = (v, \tilde E(x)): \R^{2d}\to \R^{2d}$ is divergence-free and coupled to $g_t$ via \eqref{eq.B}. 

For Problem B we need to deal with the notion of renormalized solution in the whole space (see the analogy with Definition~\ref{defi.HS_Om}). We recall it here for the reader convenience.

\begin{defi}[Renormalized solution]
\label{defi.WS}
Let $T > 0$, and let $\tilde b\in L^1_{\rm{loc}} ([0, T]\times \R^{2d}; \R^{2d})$ be a divergence-free vector field. A Borel function $g_t\in L^1([0,T]\times\R^{2d})$ is a \emph{renormalized solution} of \eqref{eq.PbB} starting from $g_0$ if for every $\beta\in C^1\cap L^\infty(\R)$ and for every $\phi \in C^1_c ([0, T)\times \R^{2d})$, 
\begin{equation}
\label{eq.defi_WS}
 \int_{\R^{2d}} \phi_0(x, v) \beta(g_0(x, v))\,dx\,dv 
 + \int_0^T \int_{\R^{2d}} [\de_t \phi_t(x, v) + \nabla_{x, v} \phi_t(x, v) \cdot \tilde b_t(x, v) ] \beta(g_t(x, v))\,dx\,dv\,dt = 0.
\end{equation}

In the case of Problem B, a function $g_t\in L^\infty((0, T); L^1(\R^{2d}))$ even with respect to $(x_1, v_1)$ is a renormalized solution of \eqref{eq.B} starting from $g_0$ if setting $\tilde b_t(x, v) = (v, \tilde E_t)$ defined as in \eqref{eq.B} then $g_t$ solves \eqref{eq.defi_WS}. 
\end{defi}
\begin{rem}
\label{rem.distr.gen}
For more general vector fields, $\tilde b\in L^1_{\rm{loc}} ([0, T]\times \R^{2d}; \R^{2d})$, not necessarily divergence free, we say that a function $g_t\in L^1([0, T]\times \R^{2d})$ is a renormalized solution to the transport equation in the whole space $(0,T)\times\R^d\times\R^d$, $\de_t g_t + \tilde b_t\cdot \nabla_{x,v} g_t=0$, starting at $g_0$, if for every $\beta\in C^1\cap L^\infty(\R)$ and for every $\phi \in C^1_c([0, T)\times\R^{2d})$, \eqref{eq.defi_WS} holds with the following right-hand side,
\[
-\int_0^T \int_{\R^{2d}}  {\rm div}_{x, v}( \tilde b_t(x, v) ) \varphi_t(x, v)\beta(g_t(x, v)) \,dx\,dv\,dt.
\]
Notice that, in general, for this definition to make sense one needs ${\rm div}_{x, v} \,b_t \in L^1_{\rm loc}((0, T)\times\R^{2d})$.
\end{rem}

\subsection{Equivalence between problems} Let us now prove that there exists a relation between renormalized solutions in the half-space and in the whole space under symmetry conditions. 

\label{sec.equiv}
\begin{lem}
\label{lem.testfunctions}
In the statement of Definition~\ref{defi.WS}, in the setting of Problem~B, \eqref{eq.B}, for a solution $g_t \in L^\infty([0, T); L^1(\R^{2d}))$ it is enough to consider test functions $\phi \in C^1_c([0, T)\times \R^{2d})$ such that 
\begin{equation}
\label{eq.suppphi}
{\rm supp}~\phi \Subset \{[0, T) \times {\R^d}\times\R^d \}\setminus\{( 0 \times \de\R^d_+\times \R^d )\cup \gamma^0_T \} \}.
\end{equation}
Equivalently, $g_t \in L^\infty([0, T); L^1(\R^{2d}))$ is a renormalized solution of \eqref{eq.B} starting from $g_0$ if \eqref{eq.defi_WS} holds for $\tilde b_t(x, v) = (v, \tilde E_t)$ and for any $\beta\in C^1\cap L^\infty(\R)$ and for any $\phi$ fulfilling \eqref{eq.suppphi}.
\end{lem}
\begin{proof}
We start by claiming that, if $\tilde E_t$ is defined as in \eqref{eq.B}, then 
\begin{equation}
\label{eq.claimE}
\tilde E \in L^\infty([0, T); L^p_{\rm loc}(\R^d))~~\textrm{for some } p > 1.  
\end{equation} 
Indeed, notice that $|\tilde E_t(x)| \leq C(|\tilde \rho_t|*|K|)(x)$, where $K = x/|x|^d$. Take any $p \in \left(1, \frac{d}{d-1}\right)$ so that $K \in L^p_{\rm loc}(\R^d)$. Use a local version of Young's inequality for convolutions in any ball $B_R(0)\subset \R^d$. That is, by means of Hölder inequality,
\begin{align*}
\|\tilde E_t\|_{L^p(B_R)}^p & \leq C\int_{B_R}\left(\int_{\R^d}|\tilde \rho_t(y)| |K(x-y)| dy \right)^pdx \\
& = C\int_{B_R}\left(\int_{\R^d} \big(|\tilde\rho_t(y)||K(x-y)|^p\big)^{\frac{1}{p}}|\tilde\rho_t(y)|^{\frac{p-1}{p}}dy\right)^p dx\\
& \leq C\int_{B_R}\left(\int_{\R^d} |\tilde\rho_t(y)||K(x-y)|^pdy\right) \left(\int_{\R^d} |\tilde\rho_t(y)|dy\right)^{p-1} dx\\ 
& = C \left(|\tilde \rho_t|(\R^d)\right)^{p-1} \int_{\R^d}|\tilde{\rho_t}(y)| \left(\int_{B_R} |K(x-y)|^p\,dx \right) dy\\
& \leq  C \left(|\tilde \rho_t|(\R^d)\right)^{p}\|K\|_{L^p(B_R)}^p  <  \infty,
\end{align*}
as we wanted. We are using also that $|\tilde \rho_t| \in L^\infty([0, T); L^1(\R^d))$.

Let us define, for $h \in L^\infty([0, T); L^\infty(\R^{2d}))$,  $\psi \in C^1_c([0,T)\times\R^{2d})$, and in the setting of Problem~B, \eqref{eq.B},
\begin{equation}
\label{eq.defopa2} \tilde A(\psi, h ) := \int_{ \R^{2d}} \psi_0(x, v) h(0, x, v)\,dx\,dv  + \int_0^T \int_{\R^{2d}} [\de_t \psi_t(x, v) + \nabla_{x, v} \psi_t(x, v) \cdot\tilde b_t(x, v) ] h(t, x, v)\, dt \,dx\,dv.
\end{equation}

Let $\xi \in C^\infty([0,\infty))$ be a monotone increasing function, with $\xi \equiv 1$ in $[1, \infty)$, $\xi \equiv 0$ in $[0, 1/2]$ and $\xi' \leq 3$ in $[1/2, 1]$. Let us define $\xi^\eps\in C^1(\R^2)$ as
\[
\xi^\eps(u, v) = \xi\left(\frac{|u|+|v|}{\eps}\right).
\]

Notice that, given any test function $\psi\in C^1_c([0, T)\times\R^{2d})$, then $\psi\, \xi^\eps(x_1, v_1)\xi^\eps(x_1, t)$ fulfils the condition \eqref{eq.suppphi} for any $\eps > 0$. Let $\xi^\eps_{xv} :=\xi^\eps(x_1, v_1)$ and $\xi^\eps_{xt} :=\xi^\eps(x_1, t)$. Our lemma is then equivalent to seeing that 
\begin{equation}
\label{eq.lemequiv}
\tilde A \Big(\psi\, \big(1-\xi^\eps_{xv}\xi^\eps_{xt}\big) , h\Big) = \tilde A(\psi, h) - \tilde A(\psi\, \xi^\eps_{xv}\xi^\eps_{xt}, h) \to 0 \textrm{ as } \eps \downarrow 0,
\end{equation}
for any $h \in L^\infty([0, T); L^\infty(\R^{2d}))$. Let us check it by taking limits term by term in the definition of $\tilde A$, \eqref{eq.defopa2}, and assuming ${\rm supp}\,\psi \subset [0, T)\times B_R \times B_R$ for $B_R \subset \R^d$. 

For the first term, notice that $\big(1-\xi^\eps_{xv}\xi^\eps_{xt}\big) = 0$ whenever $|x_1| \geq \eps$, and therefore
\begin{align*}
\int_{ \R^{2d}} \psi_0(x, v) \big(1-\xi^\eps_{xv}\xi^\eps_{x0}\big) h(0, x, v)\,dx\,dv \leq C \int_{B_R\cap \{|x_1|\leq \eps\}} dx = C\eps \to 0 \textrm{ as }\eps \downarrow 0,
\end{align*}
for some constant $C$ depending on the $L^\infty$ norm of $\psi_0$ and $h$, and the volume of $B_R$. 

For the term corresponding to $\de_t \psi_t$ we have 
\begin{align*}
\int_0^T & \int_{\R^{2d}} \de_t \big(\psi_t (1-\xi^\eps_{xv}\xi^\eps_{xt})\big)  h\,dt \, dx\, dv \leq \\
& \leq \int_0^T \int_{B_R}\int_{B_R\cap \{|x_1|\leq \eps\}} |\de_t (\psi_t) |h\,dt \, dx\, dv+\int_0^T \int_{\R^{2d}} \psi_t \xi^\eps_{xv}|\de_t \xi^\eps_{xt}| h\,dt \, dx\, dv = I_1 +{II}_1.
\end{align*}
The term $I_1$ can be treated as before, and $I_1 \to 0$ as $\eps \downarrow 0$. For the second term, notice that 
\[
|\de_t \xi^\eps_{xt}| = \frac{1}{\eps}\left|\xi'\left(\frac{|x_1|+t}{\eps}\right)\right| = 0 \textrm{ if } |x_1| > \eps \textrm{ or } t > \eps,
\]
and is bounded by $3/\eps$ otherwise. Therefore,
\[
II_1 \leq \frac{C}{\eps} \int_0^\eps  \int_{B_R\cap \{|x_1|\leq \eps\}} dt dx \leq C\eps \to 0 \textrm{ as }\eps \downarrow 0.
\]

We now check the terms corresponding to $v_1\de_{x_1}\psi_t$ and $\tilde E_{t, 1} \de_{v_1} \psi_t$. The rest of the terms follow in a similar manner. 
\begin{align*}
\int_0^T & \int_{\R^{2d}} v_1 \de_{x_1} \big(\psi_t (1-\xi^\eps_{xv}\xi^\eps_{xt})\big)  h\,dt \, dx\, dv \leq \\
& \leq \int_0^T \int_{B_R}\int_{B_R\cap \{|x_1|\leq \eps\}} v_1 |\de_{x_1} (\psi_t) |h\,dt \, dx\, dv\\
&~~~~~~~~+\int_0^T \int_{\R^{2d}} v_1 \psi_t \xi^\eps_{xv}|\de_{x_1} \xi^\eps_{xt}| h\,dt \, dx\, dv+\int_0^T \int_{\R^{2d}} v_1 \psi_t |\de_{x_1} \xi^\eps_{xv}|\xi^\eps_{xt} h\,dt \, dx\, dv.
\end{align*}
Proceeding as before, the three terms go to 0 as $\eps \downarrow 0$. 

Finally, let 
\begin{align*}
\int_0^T & \int_{\R^{2d}} \tilde E_{t, 1} \de_{v_1} \big(\psi_t (1-\xi^\eps_{xv}\xi^\eps_{xt})\big)  h\,dt \, dx\, dv \leq \\
& \leq \int_0^T \int_{B_R}\int_{B_R\cap \{|x_1|\leq \eps\}} |\tilde E_{t, 1}|\cdot |\de_{v_1} (\psi_t) |h\,dt \, dx\, dv +\int_0^T \int_{\R^{2d}} |\tilde E_{t, 1}|\cdot \psi_t\cdot|\de_{v_1}\xi^\eps_{xv}| \xi^\eps_{xt} h\,dt \, dx\, dv\\
& = I_2 +II_2.
\end{align*}

We have that, by Hölder's inequality and thanks to the initial claim \eqref{eq.claimE},
\[
I_2  \leq C \int_0^T \int_{B_R\cap \{|x_1|\leq \eps\}} |\tilde E_{t, 1}|(x)\,dt\,dx  \leq  C\|\tilde E_{t, 1}\|_{L^p((0, T)\times B_R)} \cdot \left(T\mathscr{L}^d \left( B_R\cap \{|x_1|\leq \eps\} \right) \right)^{\frac{p-1}{p}} = C\eps^{\frac{p-1}{p}}.
\]
Therefore, $I_2\to 0 $ as $\eps \downarrow 0 $. For the second term, 
\begin{align*}
II_2&  \leq \frac{C}{\eps} \int_0^T \int_{B_R\cap \{|x_1|\leq \eps\}} \int_{B_R\cap \{|v_1|\leq \eps\}}  |\tilde E_{t, 1}|(x) dt dv dx \\
& \leq \mathscr{L}^d \left( B_R\cap \{|v_1|\leq \eps\} \right) \frac{C}{\eps} \int_0^T \int_{B_R\cap \{|x_1|\leq \eps\}} |\tilde E_{t, 1}|(x) dt dv dx \\
& \leq  C\|\tilde E_{t, 1}\|_{L^p((0, T)\times B_R)} \cdot \left(T\mathscr{L}^d \left( B_R\cap \{|x_1|\leq \eps\} \right) \right)^{\frac{p-1}{p}} = C\eps^{\frac{p-1}{p}} \to 0 \textrm{ as }\eps \downarrow 0,
\end{align*}
so that \eqref{eq.lemequiv} follows and the result is proved. 
\end{proof}
\begin{rem}
\label{rem.moregen}
Notice that the previous lemma holds true for more general vector fields, not necessary built as an electric field. That is, we only require $E\in L^\infty((0, T); L^p_x(\R^d))$ for some $p> 1$. Moreover, the condition can actually be reduced to asking 
\[
\eps^{-1} \int_0^T \int_{x\in D^1_\eps}\int_{v\in D^2_\eps} |E_t(x, v)| \, dt\,dx\,dv \to 0,\quad \textrm{ as } \eps \to 0,
\]
where $D^1_\eps$ and $D^2_\eps$ are domains of volume $\eps$. Other fields, like $E\in L^\infty((0, T); L^q_{x, v}(\R^{2d}))$ for $q > 2$ or linear combinations of the previous ones, also work.
\end{rem}

\begin{prop}
\label{prop.AimpB}
Let $f_t\in L^\infty((0, T); L^1(\R^d_+\times \R^{d}))$ be a renormalized solution of Problem A according to Definition~\ref{defi.HS_Om} starting from $f_0$. Let $g_t$ be the even extension with respect to $(x_1, v_1)$ of $f_t$ to the whole space. That is, 
\begin{equation}
\label{eq.evenext}
\left\{ \begin{array}{ll}
 g_t(x, v) = f_t(x, v) & \textrm{ in } (0,T)\times\overline{\R^d_+}\times \R^d \\
g_t(x, v) = f_t(x', v') & \textrm{ in } (0,T)\times\R^d_-\times \R^d.
  \end{array}\right.
\end{equation}
Then $g_t$ is a renormalized solution of Problem B according to Definition~\ref{defi.WS}.
\end{prop}
\begin{proof}
First of all, for the sake of readability we will assume that $f_t \in L^\infty((0, T); L^\infty(\R^d_+\times \R^{d}))$. Thus, instead of considering $\beta(f_t)$ we can consider directly $f_t$ (analogously with $g_t$). 

Let us define $\tilde b_t(x, v) = (v, \tilde E(x))$ from $g_t$ as in Problem B, \eqref{eq.B}. We similarly define $b_t(x, v) = (v, E(x))$ from $f_t$ as in Problem~A, \eqref{eq.A}. Notice that by definition of $g_t$, we have 
\begin{equation}
\label{eq.Eproperty}
\left. \begin{array}{ll}
 \tilde E_t(x) = E_t(x) & \textrm{ in } (0,T)\times\overline{\R^d_+} \\
(\tilde E_t(x))' =  E_t(x') & \textrm{ in } (0,T)\times\R^d_-.
  \end{array}\right.
\end{equation}
That is, $\tilde E_t$ is just an extension to the whole $\R^d$ of $E_t$ fulfilling the symmetry property $(\tilde E_t(x))' =  \tilde E_t(x')$.

Let us also define, for $h_1 \in L^\infty([0, T); L^\infty(\R^d_+\times \R^{d}))$,  $\varphi \in \mathcal{T}_T$, and in the setting of Problem~A,
\begin{align}
\nonumber A(\varphi, h_1 ) :=& \int_{ \R^d_+\times \R^{d}} \varphi_0(x, v) h_1(0, x, v)\,dx\,dv\\
\label{eq.defopa}&  + \int_0^T \int_{\R^{d}_+\times\R^{d}} [\de_t \varphi_t(x, v) + \nabla_{x, v} \varphi_t(x, v) \cdot b_t(x, v) ] h_1(t, x, v)\, dt \,dx\,dv.
\end{align}
Similarly, for $h_2 \in L^\infty([0, T); L^\infty(\R^{2d}))$,  $\psi \in C^1_c([0,T)\times\R^{2d})$, we define $\tilde A(\psi, h_2)$ as in \eqref{eq.defopa2}. 

Now fix any $\phi \in C^1_c([0,T)\times\R^{2d})$. Thanks to Lemma~\ref{lem.testfunctions} we can assume that ${\rm supp}~ \phi \Subset \{[0, \infty) \times \R^d\times\R^d \}\setminus\{( 0 \times \de\R^d_+\times \R^d )\cup \gamma^0 \}$.  We want to check $\tilde A(\phi, g_t) = 0$.

Let $\varphi^+ = \phi 1_{\{x_1 \geq 0\}}$, and $\varphi^- = \phi 1_{\{x_1 \leq 0\}}$. Consider $\overline{ \varphi}^-$ defined as $\overline{ \varphi}^-(t, x, v) = \varphi^-(t, x', v')$ in $[0, T]\times\R^{2d}$. Thus, $\varphi^+$ and $\overline{\varphi}^-$ can be thought as elements of $\mathcal{T}_T$. Now, since $f_t$ is a renormalized solution of Problem~A, we know from Definition~\ref{defi.HS_Om} that there exists $f_t^+\in L^\infty(\gamma_T^+)$ such that
\[
A(\varphi^+, f_t) + \int_{\gamma^+_T} \left( \varphi_t^+( x, v)-  \varphi_t^+( x, v')\right) f^+_{t}(x, v) v_1 dt\,d\sigma_x\,dv = 0.
\]
Similarly for $\overline{\varphi}^-$, 
\[
A(\overline{\varphi}^-, f_t) + \int_{\gamma^+_T} \left( \overline{\varphi}^-_t( x, v)-  \overline{\varphi}^-_t( x, v')\right) f^+_{t}(x, v) v_1 dt\,d\sigma_x\,dv = 0.
\]

Now notice that on $\gamma_T^+$, $\overline{\varphi}^-_t( x, v) = \varphi^-_t( x, v') = \varphi^+_t( x, v')$, so that adding the previous two equalities we obtain 
\[
A(\varphi^+, f_t) +A(\overline{\varphi}^-, f_t) = 0.
\]

On the other hand, 
\begin{align*}
\int_0^T & \int_{\R^{d}_-\times\R^{d}} [\de_t \phi_t(x, v) + v\cdot \nabla_{x} \phi_t(x, v) + \tilde E_t(x)\cdot \nabla_v \phi_t(x, v)  ] g_t( x, v)\, dt \,dx\,dv = \\
& = \int_0^T  \int_{\R^{d}_+\times\R^{d}} [\de_t \phi_t(y', w') + w' \cdot (\nabla_{x} \phi_t)(y', w') + \tilde E_t(y')\cdot (\nabla_v \phi_t)(y', w')  ] g_t( y', w')\, dt \,dy\,dw\\
& = \int_0^T  \int_{\R^{d}_+\times\R^{d}} [\de_t \overline{\varphi}_t^-(x, v) + v \cdot \nabla_{x} \overline{\varphi}_t^-(x, v) + E_t(x)\cdot \nabla_v \overline{\varphi}_t^-(x, v)  ] g_t( x, v)\, dt \,dx\,dv \\ & = A(\overline{\varphi}_t^-, f_t).
\end{align*}
We have used here the symmetry property on $\tilde E$,  \eqref{eq.Eproperty}, and \eqref{eq.evenext}. Thus, recalling the definition of $\tilde A$, \eqref{eq.defopa2}, we get
\[
\tilde A(\phi, g_t) =  A(\varphi^+, f_t)+ A(\overline{\varphi}_t^-, f_t) = 0, 
\]
as we wanted.
\end{proof}

We also want the converse result. Before proving it, let us state the following lemma, regarding the existence of a trace for the solution to Problem~B. 

\begin{lem}
\label{lem.trace}
Let $g_t\in L^\infty((0, T); L^\infty(\R^{2d}))$ be a distributional solution of Problem B starting from $g_0$, (in particular, even with respect to $(x_1, v_1)$). Then $g_t$ admits a unique trace function $\Gamma(g_t) \in  L^\infty(\gamma_T)$ in the sense of the Green formula,
\begin{equation}
\label{eq.check0} 
\begin{split}
& \int_{ \R^d_+\times \R^{d}} \varphi_0(x, v) g_0(x, v)\,dx\,dv + \int_0^T \int_{\R^{d}_+\times\R^{d}} [\de_t \varphi_t(x, v) + \nabla_{x, v} \varphi_t(x, v) \cdot \tilde b_t(x, v) ] g_t( x, v)\, dt \,dx\,dv\\
 & ~~~~~~~~~~~~~~~~~~~~~~~~~~~~~~~~~~~~~~~~~~~~~~~~~~~~~~~~~~~~~~~~~~~+ \int_{\gamma_T} v_1 \varphi_t( x, v) \Gamma(g_t)(x, v) dt\,d\sigma_x\,dv = 0,
 \end{split}
\end{equation}
for every test function $\varphi\in C^1_c([0, T)\times \overline{\R^d_+}\times\R^d)$.
\end{lem}
\begin{proof}
Notice that we can rewrite the expression \eqref{eq.check0} using the operator $A$ defined in \eqref{eq.defopa} with $\tilde b$ restricted to $(0, T)\times\R^d_+\times\R^d$. Thus, we have to prove that there is some $\Gamma(g_t) \in L^\infty(\gamma_T)$ such that
\begin{equation}
\label{eq.true}
A(\varphi_t, g_t) + \int_{\gamma_T} v_1 \varphi_t( x, v) \Gamma(g_t)(x, v) dt\,d\sigma_x\,dv = 0,
\end{equation}
for any test function $\varphi\in C^1_c([0, T)\times \overline{\R^d_+}\times\R^d)$. One can see, as in Lemma~\ref{lem.testfunctions}, that it is enough to consider $\varphi\in \mathcal{T}_T$ (defined in Definition~\ref{defi.spaces_Om}).

Thus, given any test function $\varphi\in \mathcal{T}_T$, let us define now for any $\varepsilon > 0$, a test function in the whole space as
\begin{equation}
\label{eq.testeps}
\varphi_t^\varepsilon(x, v) = \left\{ \begin{array}{ll}
 \varphi_t(x, v) & \textrm{ in } (0,T)\times\overline{\R^d_+}\times \R^d \\
\left(1+\frac{x_1}{\varepsilon}\right)\,\varphi_t((0, x_2, \dots, x_d), v) & \textrm{ in } (0,T)\times[-\varepsilon, 0)\times\R^{d-1}\times \R^d\\
0 & \textrm{ in } (0,T)\times(-\infty, -\varepsilon)\times\R^{d-1}\times \R^d.
  \end{array}\right.
\end{equation}
Notice that $\varphi^\varepsilon_t$ being Lipschitz and compactly supported can be used as a test function for Problem B. That is,
\begin{align}
\nonumber 0 & =  \tilde A(\varphi^\varepsilon, g_t) = A(\varphi, g_t) + 
\int_{ D_\varepsilon} \varphi^\varepsilon_0(x, v) g_0( x, v)\,dx\,dv\\
&\label{eq.bimpa1}~~~~ + \int_0^T \int_{D_\varepsilon} [\de_t \varphi^\varepsilon_t(x, v) + \nabla_{x, v} \varphi^\varepsilon_t(x, v) \cdot\tilde b_t(x, v) ] g_t( x, v)\, dt \,dx\,dv,
\end{align}
where $D_\varepsilon := \{(x, v) \in \R^{2d} : -\varepsilon \leq x_1 \leq 0 \}$. We are using here again the notation introduced in the proofs of Lemma~\ref{lem.testfunctions} and Proposition~\ref{prop.AimpB}, \eqref{eq.defopa2}-\eqref{eq.defopa}.

Now letting $\varepsilon \downarrow 0$ and using that $g_t\in L^\infty$, and $\tilde E \in L^p_{\rm loc}$ for $p = \frac{d}{d-1/2}$ (see \eqref{eq.claimE} in Lemma~\ref{lem.testfunctions}), we get that, for \eqref{eq.true} to be true we must have 
\begin{equation}
\label{eq.bimpa2}
\lim_{\varepsilon\downarrow 0} \frac{1}{\varepsilon}\int_0^T\int_{D_\varepsilon} v_1  \varphi_t((0, x_2, \dots, x_d), v) g_t(x, v) \, dt \,dx\,dv   = \int_{\gamma^+_T\cup\gamma^-_T} v_1\varphi_t(x, v) \Gamma(g_t)(x, v) \, dt \,d\sigma_x\,dv,
\end{equation}
for some $\Gamma(g_t) \in L^\infty(\gamma_T^+\cup\gamma_T^-)$.

In order to do so, define, for any test function $\rho = \rho(t, x_2, \dots, x_d, v) \in C^1_c((0, T)\times \R^{d-1}\times \R^d)$ compactly supported in $\{v_1 \neq 0\}$, 
\[
G_\rho(x_1) = \int_0^T \int_{\R^{d}}\int_{\R^{d-1}} \rho(t, x_2,\dots,x_d, v) g_t(x, v) dx_2\dots dx_d \, dv\, dt,\quad\quad  \textrm{a.e. in }\R.
\]
We claim that $G_\rho$ has a continuous representative and, in particular, there exists a function $\Gamma(g_t) \in L^\infty(\gamma_T^+\cup\gamma_T^-)$ satisfying \eqref{eq.bimpa2}. 

Indeed, it is enough to check for $\rho$ of the form $\rho = \rho_t(t) \rho_{\bar x} (x_2,\dots,x_d) \rho_{v_1}(v_1) \rho_{\bar v}(v_2,\dots,v_d)$ with $\rho_t(0) = \rho_{v_1}(0) = 0 $, where we denote $\bar x = (x_2,\dots,x_d)$ and $\bar v = (v_2,\dots,v_d)$. Let $\rho_{x_1}(x_1)\in C^1_c(\R)$ be a test function for $G_\rho$, and compute 
\[
\int_\R \de_{x_1}\rho_{x_1} G_\rho dx_1.
\]
By using that $g_t$ is a solution for Problem~B with $\tilde\rho = \rho_t\rho_{x_1}\rho_{\bar x} \tilde\rho_{v_1}\rho_{\bar v}$ as test function, where $\tilde\rho_{v_1} = \rho_{v_1}/v_1$, and we get
\begin{align*}
\int_\R \de_{x_1}\rho_{x_1} G_\rho dx_1 & = \int_{0}^T\int_{\R^{2d}} v_1g_t(x, v) \de_{x_1}\tilde \rho\,dt\,dx\,dv\\
& = - \int_{\R} \rho_{x_1} \left\{\int_0^T\int_{\R^{2d-1}} \frac{g_t}{v_1} \left(\de_t \rho + \bar v\cdot  \nabla_{\bar x} \rho + \tilde E_t \cdot \nabla_v \rho\right)dt\,d\bar x\,dv\right\}\,dx_1,
\end{align*}
where we are using that $\tilde \rho = \rho_{x_1}\rho /v_1$.
Notice that 
\[
\left\{\int_0^T\int_{\R^{2d-1}} \frac{g_t}{v_1} \left(\de_t \rho + \bar v\cdot  \nabla_{\bar x} \rho + \tilde E_t \cdot \nabla_v \rho\right)dt\,d\bar x\,dv\right\}(x_1) \in L^1_{\rm loc}(\R),
\]
because $g_t / v_1  \in L^\infty({\rm supp}\,\rho)$ and $\tilde E\in L^p_{\rm loc}$. Thus, $G_\rho\in W^{1,1}_{\rm loc}$, and in particular, it is continuous almost everywhere. Let $\tilde G_\rho$ be the continuous representative.  One can then define the linear operator $\mathcal{L}_\Gamma(g_t) : L^1(\gamma_T^+\cup\gamma_T^-) \to \R$ via $\mathcal{L}_\Gamma(g_t)(\rho) = \tilde G_\rho(0)$, and noticing that $|\tilde G_\rho(0)| \leq \|g_t\|_{L^\infty}\|\rho\|_{L^1}$ we get that $\mathcal{L}_\Gamma(g_t)\in \big(L^1(\gamma_T^+\cup\gamma_T^-)\big)^*\cong L^\infty(\gamma_T^+\cup\gamma_T^-)$. This implies that, for some $\Gamma(g_t) \in L^\infty(\gamma_T^+\cup\gamma_T^-)$, \eqref{eq.bimpa2} holds by taking $\rho = v_1 \varphi$ as the test function.
\end{proof}

\begin{prop}
\label{prop.BimpA}
Let $g_t\in L^\infty((0, T); L^1(\R^{2d}))$ be a renormalized solution of Problem B according to Definition~\ref{defi.WS} starting from $g_0\in L^1(\R^{2d})$. Let $f_t$ be the restriction of $g_t$ to $[0, T)\times\overline{\R^d_+}\times\R^{d}$.  

Then, for every $\beta \in C^1\cap L^\infty(\R)$ there exists $f^+_{\beta,t}\in L^\infty(\gamma_T^+)$, such that $\beta(f_t)$ and $f_{\beta, t}^+$ are a distributional solution to Problem A; i.e., $f_t$ is a renormalized solution according to Definition~\ref{defi.HS_Om}.
\end{prop}
\begin{proof}
As in the proof of Proposition~\ref{prop.AimpB}, we assume that $g_t \in L^\infty((0, T); L^\infty(\R^{2d}))$ to avoid the use of $\beta$ throughout the proof. We also recall that, by definition, $g_t$ is even with respect to $(x_1, v_1)$, i.e., $g_t(x, v) = g_t(x', v')$ in $[0, T)\times\R^{2d}$.

We will see that $f_t$ is a renormalized solution according to Remark~\ref{rem.defi.HS}. That is, we want to check that there exists some $f_t^+\in L^\infty(\gamma^+_T\cup\gamma^-_T)$ such that for any $\varphi \in \mathcal{T}_T$,
\begin{equation}
\label{eq.check1}
A(\varphi, f_t) + \int_{\gamma^+_T\cup\gamma^-_T} v_1 \varphi_t( x, v) f^+_{t}(x, v) dt\,d\sigma_x\,dv = 0,
\end{equation}
and
\begin{equation}
\label{eq.check2}
\int_{\gamma^-_T} v_1 \varphi_t(x, v) \left(f^+_{t}(x, v) -f^+_{t}(x, v') \right)  \,dt\,d\sigma_x\,dv = 0.
\end{equation}
We keep using the notation introduced in the proofs of Lemma~\ref{lem.testfunctions} and Proposition~\ref{prop.AimpB}, \eqref{eq.defopa2}-\eqref{eq.defopa}. The first part, \eqref{eq.check1}, corresponds to the result in Lemma~\ref{lem.trace}. 

To see \eqref{eq.check2}, let us consider a smooth and compactly supported extension of $\varphi$ to the whole $[0, T) \times \R^{2d}$. That is, we consider $\phi \in C^1([0, T) \times \R^{2d})$ such that $\varphi = \phi$ in $[0, T) \times \overline{\R^{d}_+}\times \R^d$. Let $\varphi^- = \phi 1_{\{x_1\leq 0\}}$, so that defining $\overline{\varphi}^-_t(x, v) = \varphi^-_t(x', v')$ in $[0, T) \times \overline{\R^{d}_+}\times \R^d$, then we can treat $\overline{\varphi}^-_t$ as a test function for Problem~A, since $\overline{\varphi}^-_t\in \mathcal{T}_T$. By Lemma~\ref{lem.trace} again,
\begin{equation}
\label{eq.bimpa3}
0 = A(\overline{\varphi}^-, f_t) + \int_{\gamma^+_T\cup\gamma^-_T} v_1\overline{\varphi}_t^-(x, v) \Gamma(g_t)(x, v) \, dt \,d\sigma_x\,dv.
\end{equation}

Combining this with \eqref{eq.check1},
\begin{equation}
\label{eq.bimpa4}
\int_{\gamma^+_T\cup\gamma^-_T} v_1\overline{\varphi}_t^-(x, v) \Gamma(g_t)(x, v) \, dt \,d\sigma_x\,dv + \int_{\gamma^+_T\cup\gamma^-_T} v_1\varphi_t(x, v) \Gamma(g_t)(x, v) \, dt \,d\sigma_x\,dv = 0.
\end{equation}

We have used here that, as in the proof of Proposition~\ref{prop.AimpB} and since $g_t$ is an even solution to Problem~B, 
\[
\tilde A(\phi, g_t) =  A(\varphi^+, f_t)+ A(\overline{\varphi}_t^-, f_t) = 0.
\]

From \eqref{eq.bimpa4}, using $\overline{\varphi}_t^-(x, v) = \varphi_t(x, v')$ on $\gamma_T^+\cup\gamma_T^-$,
\begin{equation}
\label{eq.bimpa5}
\int_{\gamma^+_T\cup\gamma^-_T} v_1\varphi_t(x, v) \left(\Gamma(g_t)(x, v)- \Gamma(g_t)(x, v')\right) \, dt \,d\sigma_x\,dv = 0,
\end{equation}
and assuming $\varphi_t$ is supported on $\gamma^-_T$ we have shown that \eqref{eq.check2} holds.
\end{proof}

\section{Maximal regular flows with specular reflection:  the Maximal Specular Flow}
\label{sec.mrf}

We would like to establish a suitable notion of flow associated to \eqref{eq.A}, and more generally to the problem in a general domain with specular reflection, \eqref{eq.original_Om_2}. Let us start by recalling the definition of a Maximal Regular Flow in $\R^{d}$ associated to a Borel vector field $b = b_t(x):(0, T)\times \R^{d}\to\R^{d}$ (see \cite{ACF15}), which contrary to Regular Flows, allows blow-ups at any time after the initial condition. This definition will be useful in the context of Problem~B, \eqref{eq.B}, but we will require another notion of flow to deal with the specular reflection.

We recall here the definition of \emph{local Maximal Regular Flow} from \cite[Definition 4.4]{ACF15}. That is, we consider the trajectories of a given Maximal Regular Flow (see \cite[Definition 4.2]{ACF17}) inside an open set $A\subset \R^d$. There is no specular behaviour at this point.

As in \cite{ACF15}, in order to identify the boundary of $A$ we define a potential function $P_A : A \to [0, \infty)$ by 
\[
P_A(x) = \max\left\{ [\textrm{dist}(x, \R^d\setminus A)]^{-1}, |x| \right\},
\]
which satisfies
\[
\lim_{x\to \de A} P_A(x) = \infty,
\]
in the sense that, for any $M >0 $, there exists $K \Subset A$ such that $P_A(A\setminus K) \ge M$.

We also have to define the notion of hitting time and entering time with respect to $A$.

\begin{defi}[Hitting and entering time in $A$ at time $s$]
\label{defi.hitenttime}
Let $\tau > 0$, $s > 0$, $A\subset \R^d$ open, and $\eta: (s-\tau, s+\tau)\to \R^d$ continuous. The \emph{hitting time} of $\eta$ in $A$ at time $s$ is given by 
\[
h_{A, s}^+(\eta):= \sup\{t\in[s, s+\tau) : \max_{[s, t]} P_A(\eta) < \infty\},
\]
while the \emph{entering time} of $\eta$ in $A$ at time $s$ is 
\[
h_{A, s}^-(\eta):= \inf\{t\in(s-\tau, s] : \max_{[t, s]} P_A(\eta) < \infty\}.
\]
We put $h_A^+(\eta)= h_A^-(\eta) = s$ if $\eta(s)\notin A$.
\end{defi}

We note that we will refer as functions of bounded variation (or BV functions) to the functions whose distributional derivative is a finite Radon measure. We will refer as absolutely continuous functions (AC) to the functions of bounded variation whose distributional derivative does not have a singular part (with respect to the Lebesgue measure). Finally, given two metric space $X$, $Y$, and a Borel map $f: X\to Y$, we denote the push-forward of a measure $\mu \in \mathcal{M}(X)$ as $f_\#\mu\in\mathcal{M}(Y)$, which is given by $f_\#\mu(B) := \mu(f^{-1}(B))$ for any $B\subset Y$ Borel, and fulfils the change of variables formula 
\begin{equation}
\label{eq.pf_cv}
\int_Y \phi\,d(f_\#\mu) = \int_X \phi\circ f\,d\mu,
\end{equation}
for any $\phi: Y \to [0, \infty]$ Borel (see \cite{AC14}).

Now we are able to provide the definition of local Maximal Regular Flow in $A$, given by 

\begin{defi}[Local Maximal Regular Flow in $A$, \cite{ACF15}]
\label{defi.mrf_A}
Let $b = b_t: (0, T)\times A\to \R^d$ be a Borel vector field. Let $s\in (0,T)$, and let $X = X(t, s, x):(0,T)\times(0,T)\times A\to\R^{d}$ be a Borel map. We call $X(\cdot,s,\cdot)$ a \emph{local Maximal Regular Flow} (associated to $b = b_t$) in $A$ starting at time $s$ if there exist two Borel maps $t_{s,X}^+:\R^d\to (s,T]$ and $t_{s,X}^-:\R^d\to [0,s)$ such that $X(\cdot,s,x):(t^-_{s,X}(x), t^+_{s,X}(x))\to \R^{d}$ satisfies 
\begin{enumerate}[(i)]
 \item $X(\cdot,s,x)\in AC_{\rm loc}((t^-_{s,X}(x), t^+_{s,X}(x));\R^d)$ for a.e. $x\in\R^d$. Moreover, $X(\cdot,s,x)$ solves the ODE $  \frac{d}{dt}X(t,s,x) = b_t(X(t,s,x))$ for a.e. $ t\in (t^-_{s,X}(x), t^+_{s,X}(x))$, with $X(s, s, x) = x$,
 
 \item for all $t\in[0,T]$, and for any $A'\Subset A$, 
 \[
  X(t, s,\cdot)_\#\left(\mathscr{L}^d \mres \{h^-_{A', s}(X(\cdot, s, x)) <t< h^+_{A', s}(X(\cdot, s, x))\}\right)\leq C\mathscr{L}^d\mres A',
 \]
 for some constant $C$ depending only on $X$, $A'$, and $s$, and where $h^-_{A', s}$ and $h^+_{A', s}$ are defined in Definition~\ref{defi.hitenttime}.
\item for a.e. $x\in \R^d$ the following dichotomy follows:
\begin{itemize}
 \item either $t_{s,X}^+ = T$ (resp. $t_{s,X}^- = 0$) and $X(\cdot, s, x)$ can be continuously extended to $t = T$ (resp. $t= 0$) and therefore $X(\cdot, s, x) \in C([s,T];\R^d)$ (resp. $X(\cdot, s, x) \in C([0,s];\R^d)$),
 \item or
 \[
  \lim_{t\uparrow t_{s,X}^+} P_A\left(|X(t, s, x)|\right) = \infty,\quad\quad ({\rm resp.}\lim_{t\downarrow t_{s,X}^-	} P_A\left(|X(t, s, x)|\right) = \infty).
 \]
\end{itemize}
\end{enumerate}
We will simply refer to a \emph{Maximal Regular Flow} whenever $A = \R^d$. Moreover, when the set $A$ is clear from the context we will simply write $t^\pm_{s, X}$. 
\end{defi}

The previous definition is consistent with what we intuitively understand as a flow that might blow up, requiring local in time boundedness with respect to the Lebesgue measure, local absolute continuity, and the fulfilment of an ODE at those points. Condition (ii) ensures that there is not \emph{too much} production of mass, and that sets with finite Lebesgue measure go to sets with controlled Lebesgue measure. In our case, the divergence-free condition will ensure that the constant $C = 1$, and the inequality in (ii) is actually an equality (mass is transported). 

In a similar manner we can define what we call a \emph{Maximal Specular Flow}. In this case, we would like a definition analogous for flows in $\Omega\times\R^d$ that may have \emph{jumps} in a velocity coordinate, corresponding to specular reflections on the boundary. Thanks to the restrictions imposed on these jumps we can still talk about a solution being transported by the Maximal Specular Flow. 

However, since we are now talking about \emph{velocities} and \emph{positions} at which reflections occur, we need a particular structure for the vector field considered. In this case, we will talk about a \emph{Maximal Specular Flow} in a domain $\Omega\times\R^d$ associated to a Borel vector field $b = b_t(x,v) : (0, T) \times\Omega\times\R^d \to \R^d\times\R^d $ of the form $b_t(x, v) = (v, E_t(x, v))$, for some Borel vector field $E = E_t(x, v) : (0, T) \times\Omega\times\R^d \to \R^d$. Thus, the velocity coordinate will correspond to the temporal derivative of the position, while the \emph{force field} might still depend on all the variables. 

Still, the following definition is analogous to Definition~\ref{defi.mrf_A} of Maximal Regular Flow introduced in \cite{ACF15}. Namely, the only difference is that in the velocity component the absolute continuity in time holds everywhere except at the boundary. This is encoded in (ii) from the following definition: the singular part of the temporal derivative of the velocity exists only at the boundary, and has a very particular structure due to the specular reflection condition. Away from the boundary, $(X, V)$ is absolutely continuous, and the following definition coincides exactly with the definition of Maximal Regular Flow, and one could still have trajectories blowing up in finite time.

\begin{defi}[Maximal Specular Flow in a domain]
\label{defi.msf_dom}
Let $s\in (0,T)$, and let us consider a Borel map
\[
(X, V) = (X, V)(t, s, x, v):(0,T)\times(0,T)\times\Omega\times\R^d\to\Omega\times\R^d.\]
Let $b = b_t(x, v) = (v, E_t(x, v)): (0, T) \times\Omega\times\R^d \to \R^d\times\R^d $ be a Borel vector field with $E = E_t(x, v) : (0, T) \times\Omega\times\R^d \to \R^d$. Let $\Omega\subset \R^d$ be a $C^{1, 1}$ domain. We call $(X, V)(\cdot,s,\cdot, \cdot)$ a \emph{Maximal Specular Flow in $\Omega\times\R^d$} (associated to $b = b_t$) starting at time $s$ if there exist two Borel maps $t_{s,X, V}^+:\Omega\times\R^d\to (s,T]$ and $t_{s,X, V}^-:\Omega\times\R^d\to [0,s)$ such that $(X, V)(\cdot,s,x, v):(t^-_{s,X, V}(x,v), t^+_{s,X, V}(x, v))\to \Omega\times\R^d$ satisfies 
\begin{enumerate}[(i)] 
 \item for a.e. $(x, v)\in \Omega\times\R^d$
 \begin{align*} X(\cdot,s,x, v)& \in AC_{\rm loc}((t^-_{s,X, V}(x, v), t^+_{s,X, V}(x, v));\Omega)\\
 V(\cdot, s, x, v)& \in BV_{\rm loc}((t^-_{s,X, V}(x, v), t^+_{s,X, V}(x, v));\R^d). 
 \end{align*}
Moreover, $(X, V)(\cdot,s,x, v)$ solves the ODE 
 \begin{equation}
 \label{eq.ODE_MSF}
  \frac{d}{dt}(X, V)(t,s,x, v) = b_t(X(t,s,x, v), V(t, s, x, v)),\quad \textrm{for a.e. } t\in (t^-_{s,X, V}(x, v), t^+_{s,X, V}(x, v)),
 \end{equation}
 with $(X, V)(s, s, x, v) = (x, v)$,
 \item for a.e. $(x, v) \in \Omega\times\R^d$,
 \[
 D_t V(\cdot, s, x, v) = \dot V(\cdot, s, x, v)dt + D_t^s V(\cdot, s, x, v),
 \]
 where $\dot Vdt$ denotes the absolutely continuous part with respect to the (temporal) Lebesgue measure, and $D_t^s V$ is the singular part. Moreover, 
 \[
 \textrm{supp}~D_t^s V(\cdot, s, x, v) \subset \{\tau \in (t^-_{s,X, V}(x, v), t^+_{s,X, V}(x, v)): X(\tau, s, x, v) \in \de\Omega\}
 \]
 that is, $|D^s_t V|$ is concentrated on $X(\cdot, s, x, v) \in\de\Omega$ (jumps in velocity only occur at the boundary), and is of the form 
 \[
 D_t^s V (\cdot, s, x, v)= 2\sum_{j\in\N} A_j \delta_{t_j},
 \]
 for some $(t_j)_{j\in \N} \subset (t^-_{s,X, V}(x, v), t^+_{s,X, V}(x, v))$, where $\delta_{t_j}$ denotes the Dirac delta at $t_j$, and 
 \begin{equation}
 \label{eq.specular_condition}
 A_j = - \lim_{\tau \downarrow 0}  n(X(t_j, s, x, v)) \left[ V(t_j - \tau, s, x, v) \cdot n(X(t_j, s, x, v))\right].
\end{equation}
 \item for all $t\in[0,T]$,
 \[
  (X, V)(t, s,\cdot, \cdot)_\#\left(\mathscr{L}^{2d} \mres \{t_{s,X, V}^- <t< t_{s,X, V}^+\}\right)\leq C\mathscr{L}^{2d},
 \]
 for some constant $C$ depending only on $(X, V)$ and $s$.
\item for a.e. $(x, v)\in \Omega\times\R^d$ the following dichotomy follows:
\begin{itemize}
 \item either $t_{s,X, V}^+ = T$ (resp. $t_{s,X, V}^+ = 0$) and $(X, V)(\cdot, s, x, v)$ can be continuously extended to $t = T$ (resp. $t= 0$) and therefore $(X, V)(\cdot, s, x, v) \in C([s,T];\R^d\times \R^d)$ (resp. $(X, V)(\cdot, s, x, v) \in C([0,s];\R^d\times\R^d)$),
 \item or
 \[
  \lim_{t\uparrow t_{s,X, V}^+} |(X, V)(t, s, x, v)| = \infty,\quad\quad ({\rm resp.}\lim_{t\downarrow t_{s,X, V}^-	} |(X, V)(t, s, x, v)| = \infty).
 \]
\end{itemize}
\end{enumerate}
\end{defi}
\begin{rem}
\label{rem.msf}
Note that for the half-space situation, $\Omega = \R^d_+$, in condition (i) above we also have that 
\[
 V_i(\cdot, s, x, v) \in AC_{\rm loc}((t^-_{s,X, V}(x, v), t^+_{s,X, V}(x, v));\R) \quad\textrm{ for } i \in \{2,\dots,d\}.
\]
Moreover, a quick check shows that in condition (ii) above the singular part of $D_t V_1$ is concentrated on $X_1 = 0$, that is $X_1(\cdot, s, x, v) |D_t^s V_1| \equiv 0$, and 
\[
 D_t^s V_1 (\cdot, s, x, v)= 2\sum_{j\in\N} \alpha_j \delta_{t_j},
 \]
 for some $(t_j)_{j\in \N} \subset (t^-_{s,X, V}(x, v), t^+_{s,X, V}(x, v))$, where $\alpha_j = \lim_{\tau \downarrow t_j} V_1(\tau, s, x, v)$.
\end{rem}

The previous definition is constructed according to what we intuitively perceive as a flow with a specular reflection condition on a boundary in the $x$-coordinates: it coincides with the definition of Maximal Regular Flow at all points away from the boundary. At boundary points, moreover, we impose specular reflection of the vector field (namely, \eqref{eq.specular_condition}), so that the resulting vector fields are absolutely continuous everywhere, except at reflection points where the vector field is at most BV. Notice, moreover, that the jump induced only occurs in the velocity components, and it occurs in the component normal to the boundary (i.e., if the boundary is a half-space with normal vector $\boldsymbol{e}_1$, then the velocity will be AC in all components except in $V_1$, which will be BV).  

Notice that, if $b_t$ is Lipschitz, then the previous notion coincides with building the standard flow away from the boundary, and imposing specular reflection boundary conditions. Indeed, away from the boundary \eqref{eq.ODE_MSF} holds pointwise and when reaching a boundary point, \eqref{eq.specular_condition} ensures that the velocity component instantly changes in the right specular way.

It is a priori not clear whether the restriction of even (with respect to $(x_1, v_1)$) Maximal Regular Flows to $\{x_1 > 0\}$ induces a Maximal Specular Flow in $\R^d_+\times\R^d$: one needs to check that the $V_1$ component of the velocity flow behaves in the expected way. That is why we need the following lemma.

\begin{lem}
\label{lem.mrfmsf}
Let $s\in (0,T)$, and let 
\[
(X, V) = (X, V)(t, s, x, v):(0,T)\times(0,T)\times\R^{d}\times\R^d\to\R^{d}\times\R^d
\]
be a Borel map such that $(X, V)(\cdot,s,\cdot, \cdot)$ is a Maximal Regular Flow starting at time $s$ associated to $b = b_t(x, v) = (v, E_t(x, v)): (0, T) \times\R^d\times\R^d \to \R^d\times\R^d $ for a Borel vector field $E = E_t(x, v) : (0, T) \times\R^d\times\R^d \to \R^d$.

Suppose that for a.e. $(x, v)\in \R^d\times\R^d$ we have that
\begin{enumerate}[(i)]
\item $t^-_{s,X, V}(x,v)= t^-_{s,X, V}(x',v')$ and $t^+_{s,X, V}(x,v)= t^+_{s,X, V}(x',v')$ ,
\item for $t \in (t^-_{s,X, V}(x,v), t^+_{s,X, V}(x, v))$, and for a.e. $(x, v) \in \R^d\times\R^d$,
\begin{equation}
\label{eq.evencondition}
(X', V')(t, s, x, v) = (X, V)(t, s, x', v').
\end{equation}
\end{enumerate}
Then, the map
\[
(\tilde X, \tilde V) = (\tilde X, \tilde V)(t, s, x, v):(0,T)\times(0,T)\times\R^{d}_+\times\R^d\to\R^{d}_+\times\R^d
\]
defined as 
\begin{align*}
\tilde X_1(t, s, x, v) &= |X_1(t, s, x, v)|\\
\tilde X_i(t, s, x, v) &= X_i(t, s, x, v),\quad\textrm{for }\quad i\in \{1, \dots,d\},\\
\tilde V(t, s, x, v)& = \frac{d}{dt} \tilde X(t, s, x, v),
\end{align*}
is well defined, and $(\tilde X, \tilde V)(\cdot, s, \cdot, \cdot)$ is a Maximal Specular Flow in $\R^d_+\times\R^d$ (see Definition~\ref{defi.msf_dom}) starting at time $s$ associated to $b = b_t(x, v) = (v, E_t(x, v))$. 
\end{lem}
\begin{proof}
In order to check that $(\tilde X, \tilde V)(\cdot, s, \cdot, \cdot)$ is a Maximal Specular Flow, it is enough to check that 
\begin{equation}
\label{eq.tildeV1}
\tilde V_1(\cdot, s, x, v) \in BV_{\rm loc}((t^-_{s,X, V}(x, v), t^+_{s,X, V}(x, v));\R),
\end{equation}
and that Definition~\ref{defi.msf_dom}(ii) (more precisely, Remark~\ref{rem.msf}) holds for $\tilde V_1$; the rest of assumptions follow trivially from the definition of Maximal Regular Flow and the symmetry of $(X, V)$, \eqref{eq.evencondition}.

Let us fix $s \in (0, T)$, $(x, v)\in \R^d\times\R^d$, and let $t^\pm := t^\pm_{s, X, V}(x, v)$. Let
\[
h(t) := X_1(t, s, x, v),
\]
so that $h, \de_t h \in AC_{\rm loc}(t^-,t^+)$ for a.e. $(x, v)\in \R^d\times\R^d$ by assumption (note that $\de_t h = V_1(t, s, x, v)$). We first want to show \eqref{eq.tildeV1}, that is, we want to prove
\begin{equation}
\label{eq.provemsf}
\de_t |h| = {\rm sgn} (h)\de_t h \in BV_{\rm loc}(t^-,t^+).
\end{equation}

Let $\varphi_\eps\in C^\infty(\R)$ be an approximation of the sign function, 
\begin{align*}
& \varphi_\eps(t) \equiv -1 \textrm{ for } t \leq -\eps,\quad\varphi_\eps(t) \equiv 1 \textrm{ for } t \geq \eps,\\
& \varphi_\eps(t) \textrm{ is non-decreasing, }0\leq \de_t \varphi_\eps  \leq \frac{2}{\eps}.
\end{align*}

In order to check \eqref{eq.provemsf} it is enough to prove the following bound,
\begin{equation}
\label{eq.provemsf2}
\int_a^b \left|\de_t \left\{\de_t h (s) \varphi_\eps(h (s))\right\} \right|ds \leq C(h, a, b),\textrm{ for all } (a, b)\Subset (t^-, t^+),
\end{equation}
for some constant $C(h, a, b)$ depending on $h$ and the endpoints $a$ and $b$, but independent of $\eps$. Notice that, since $\de_{tt} h\in L^1_{\rm loc}$ and $|\varphi_\eps|\leq 1$,
\begin{equation}
\label{eq.abscontpartb}
\int_a^b |\de_{tt} h|\,|\varphi_\eps(h)| \leq \|\de_{tt} h \|_{L^1(a, b)},
\end{equation}
where in the limit $\eps \downarrow 0$ the equality is attained, and it would correspond to the absolutely continuous part with respect to the Lebesgue measure. 

Let us now bound, independently of $\eps$, the other term in \eqref{eq.provemsf2}, $ \int_a^b \varphi_\eps' (h) (\de_t h)^2$. Since $\de_t h $ is continuous, $\{|\de_t h | > 0\}\cap (a, b)$ is an open set in $\R$, and in particular, it is a countable union of open disjoint intervals $I_i = (a_i, b_i)$; that is, 
\[
\{|\de_t h | > 0\}\cap (a, b) = \bigcup_{i \in \N} I_i,\quad\textrm{ such that } \de_t h > 0 \textrm{ or }\de_t h < 0 \textrm{ in } I_i, \textrm{ for each } i \in \N. 
\]

Now consider a fixed interval $I_i = (a_i, b_i)$ for some $i\in \N$, and suppose that $\de_t h > 0$ in $I_i$. Let $I_i^\eps = I_i\cap \{|h|< \eps\}$, which is still a single open interval because $h$ is continuous and $\de_t h$ has constant sign in $I_i$. Then compute, by changing variables $s \mapsto r = h(s)$,
\begin{align*}
\int_{I_i} \varphi_\eps' (h(s)) (\de_t h(s))^2 ds & = \int_{I_i^\eps}\varphi_\eps' (h(s)) (\de_t h(s))^2 ds\\
& = \int_{h(I_i^\eps)} \varphi_\eps'(r) \de_t h (h^{-1}(r)) dr\\
& \leq \sup_{s \in I_i^\eps}\,\de_t h(s) \int_{-\eps}^\eps \varphi_\eps'(r) dr = 2 \de_t h(\bar s_i),
\end{align*}
for some $\bar s_i \in \overline{I_i^\eps}$. We are using here that $\varphi_\eps'\circ h \equiv 0$ outside $\{|h|<\eps\}$. If $\de_t h < 0 $ in $I_i$ we would have obtained $2|\de_t h(\bar s_i)|$ instead, so that in general,
\[
\int_{I_i} \varphi_\eps' (h(s)) (\de_t h(s))^2 ds \leq 2|\de_t h(\bar s_i) - \de_t h(b_i)|,
\]
where we are also using that by continuity $\de_t h(b_i) = 0$. Thus, using that the integrand is non-negative
\[
\int_a^b\varphi_\eps' (h) (\de_t h)^2  = \sum_{i \in \N} \int_{I_i}\varphi_\eps' (h) (\de_t h)^2  \leq 2\sum_{i \in \N} |\de_t h(\bar s_i) - \de_t h(b_i)| = 2\sum_{i\in \N} \left|\int_{\bar s_i}^{b_i} \de_{tt}h (s) ds\right|,
\]
and thanks to the absolute continuity of $\de_t h$,
\begin{equation}
\label{eq.singpartb}
\int_a^b\varphi_\eps' (h) (\de_t h)^2 \leq 2 \sum_{i\in \N} \int_{\bar s_i}^{b_i} |\de_{tt}h (s) | ds \leq 2 \int_a^b |\de_{tt} h(s)|\, ds \leq C(h, a, b).
\end{equation}
Combining \eqref{eq.abscontpartb} and \eqref{eq.singpartb}, yields \eqref{eq.provemsf2} which gives the desired result \eqref{eq.provemsf}. We now want to understand the structure of the distributional derivative of $\de_t |h|$.

Notice that we can consider two different cases according to whether $h(I_i)$ contains the value $0$ or not (that is, $h$ has a root in $I_i$). On the one hand, if $|h|> 0$ in $I_i$, then for some $\eps $ small enough, $I_i^\eps = \varnothing$ and $\int_{I_i} \varphi_\eps' (h) (\de_t h)^2 = 0$. On the other hand, if $0 \in h(I_i)$ then $\sup_{s\in I_i^\eps} \de_th(s) \to \de_th(r_i)$ as $\eps\downarrow 0$ for the unique $r_i\in I_i$ such that $h(r_i) = 0$. That is, if $(r_i)_{i\in \N}$ denotes the set of roots of $h$ in $(a, b)$, then 
\[
\int_{a}^b  \varphi_\eps' (h) (\de_t h)^2 \to 2\sum_{i\in \N}|\de_t h(r_i)|,\quad\textrm{as}\quad \eps \downarrow 0.
\]
This corresponds to the singular part of the measure. Combining with \eqref{eq.abscontpartb} this gives
\[
 \int_a^b \left|\de_{tt} |h|(s) \right|\, ds = \|\de_{tt} h\|_{L^1(a, b)} + 2\sum_{i\in \N}|\de_t h(r_i)|,
\]
for any $(a, b)\Subset (t^-,t^+)$, and for $(r_i)_{i\in \N}$ the set of roots of $h$ in $(a, b)$.

Interpreting this in terms of $X_1(t, s, x, v)$ and $V_1(t, s, x, v)$ we obtain the desired result. 
\end{proof}

\section{Uniqueness of the continuity equation}
\label{sec.4}
In this section we address the problem of uniqueness of the continuity equation associated to Problem~B, \eqref{eq.B}. Namely, we want to check that the vector field in the mentioned problem satisfies the condition {\bf (A2)} from \cite[Section 4.1]{ACF17}.

Let us rewrite the problem here in a more convenient way,
\begin{equation}
\label{eq.B2}
  \left\{ \begin{array}{ll}
  \de_t g_t + v\cdot \nabla_x g_t + \tilde E_t\cdot\nabla_v g_t=0 & \textrm{ in } (0,\infty)\times\R^d\times\R^d\\
  \tilde \rho_t(x)={\rm sgn}(x_1)\int_{\R^d} g_t(x, v) dv & \textrm{ in } (0,\infty)\times\R^d\\
  \tilde E_t(x) = {\rm sgn}(x_1) c_d \int_{\R^d} \tilde \rho_t(y) \frac{x-y}{|x-y|^d}dy & \textrm{ in } (0,\infty)\times\R^d, \\ 
    g_t(x, v) = g_t(x', v')& \textrm{ in } (0,\infty)\times\R^d\times\R^d.\\
  \end{array}\right.
\end{equation}
In particular, $\tilde \rho_t(x) = -\tilde \rho_t(x')$, and $\tilde E_t(x') = (\tilde E_t(x))'$.

The existence and uniqueness of a Maximal Regular Flow (see Definition~\ref{defi.mrf_A}) transporting the solution of the continuity equation with vector field $b$ is discussed in \cite{ACF17} under the assumptions on $b$ induced by the results from the same authors in \cite{ACF15}. 

More precisely, in order to construct a Maximal Regular Flow associated to a general Borel vector field $b: (0, T)\times\R^d\to \R^d$, it is sufficient that the following assumptions are satisfied (see \cite{ACF15}):
\begin{enumerate}
\item[\bf (A1)] $\int_0^T\int_{B_R} |b_t(x)| dx dt <\infty$ for any $R > 0$,
\item[\bf (A2)] for any nonnegative $\bar \mu \in L^\infty_+(\R^d)$ with compact support and any closed interval $I = [a, b]\subset[0, T]$, the continuity equation 
\begin{equation}
\label{eq.A2}
\frac{d}{dt}\mu_t + {\rm div}(b\mu_t)  = 0,\quad\textrm{in }\quad (a, b)\times \R^d 
\end{equation}
has at most one weakly* continuous solution $I\ni t \mapsto \mu_t$ such that $\mu_a = \bar\mu$ and $\cup_{t\in [a, b]} {\rm supp}\,\mu_t \Subset\R^d$.
\end{enumerate}

We want to check that assumption {\bf (A2)} is satisfied whenever the vector field $b$ is of the form $b  =\tilde b_t(x, v) = (v, \tilde E_t(x))$ coming from \eqref{eq.B2}.

\begin{thm}
\label{thm.charact}
Let $b:(0, T)\times\R^{2d}\to \R^{2d}$ be a vector field given by $b_t(x, v) = (b_{1, t}(v), b_{2, t}(x))$, where
\begin{align*}
&b_{1, t}(v)  \in L^\infty((0, T); W^{1, \infty}_{\rm loc} (\R^d; \R^d))\\
& b_{2, t}(x)  = {\rm sgn}(x_1) K * \rho_t,
\end{align*}
with $K = x/|x|^d$ and $|\rho|\in L^\infty((0, T); \mathcal{M}_+(\R^d))$ an odd measure with respect to $x_1$; that is, for any Borel set $E \subset \R^d$, $\rho(E) = -\rho(E')$ (where $x\in E'\Leftrightarrow x' \in E$). In particular, $\rho(\{x_1 = 0\}) = 0$. 

Then $b$ satisfies assumption {\bf (A2)}, that is, the uniqueness of bounded compactly supported nonnegative distributional solutions of the continuity equation. 
\end{thm}
\begin{proof}
This proof follows along the lines of \cite[Theorem 4.4]{ACF17} where Ambrosio, Colombo, and Figalli, deal with a vector field $b_2$ given by a full convolution. At the same time, it is a variant of a result by Bohun, Bouchut, and Crippa, \cite{BBC16}. 

Since our vector field $b_2$ is \emph{almost} a convolution, we will be able to repeat the proof in \cite[Theorem 4.4]{ACF17} in many steps. For the sake of completeness we repeat the main steps of their proof here, which after small adaptations, help us bound all but the first component of $b_{2}$. In order to get the results for the first component, we need to proceed differently and use the precise form of $b_2$ and the symmetry of the problem with respect to $x_1$. 

For the sake of readability, we do not explicitly write the time dependence on the vector field $b$. Let us denote $\mathscr P(X) $ the space of probability measures on $X$, and $e_t: C([0, T]; \R^k) \to \R^k$ the evaluation map at $t$, $e_t(\eta) := \eta(t)$, where $k = d$ or $2d$ depending on the context. 

Let us denote $\Gamma_T := C([0, T); \R^{2d})$, and $\Gamma_T^x :=\{\gamma\in \Gamma_T : \gamma(0) = x\}$, for $x\in \R^{2d}$. The proof is based on the superposition principle, \cite[Theorem 8.2.1]{AGS05}, which says that nonnegative solutions to the continuity equation $\mu_t^\eta$ (with a suitable vector field) starting from $\bar\mu$ can be represented by probability measures $\eta\in \mathscr P(\Gamma_T)$ concentrated on curves $\gamma$ solutions to the $\dot\gamma(t) = b_t(\gamma(t))$. More precisely, we can write 
\[
\int_{\R^{2d}} \varphi~d\mu_t^\eta = \int_{\Gamma_T} \varphi \circ\gamma (t)~d\eta(\gamma) = \int_{\R^{2d}}\left(\int_{\Gamma_T^x} \varphi(\gamma(t)) d\eta_x(\gamma)\right)d\bar\mu(x),
\]
where $\bar\mu = {e_0}_\#\eta$. We have also considered the disintegration of $\eta$ with respect to the map $e_0$ (i.e., the initial value), so that $\eta = \eta_x\otimes d\bar\mu(x)$. Notice that, in particular, if there was a unique solution to the ODE starting from $x$, then $\eta_x = \delta_{\gamma_x}$, and the unique solution to the continuity equation starting from $\bar \mu$ would be given by $\mu_t = {e_t}_\#\eta$, with $\eta = \delta_{\gamma_x}\otimes d\bar\mu(x)$. 

To prove uniqueness, suppose that $\mu_t^1$ and $\mu_t^2$ are compactly supported solutions to the continuity equation, and let $\eta_1$ and $\eta_2$ be the corresponding measures in $\mathscr{P}(\Gamma_t)$ given by the superposition principle. Define $\eta = \frac{\eta_1+\eta_2}{2}$. If we can prove that $\eta_x$ (the disintegration with respect to $e_0$) is a Dirac delta for $\bar\mu$-a.e. $x$, then this implies that $(\eta_1)_x$ and $(\eta_2)_x$ are also a Dirac delta and hence, $\eta_1 = \eta_2$ and $\mu_1 = \mu_2$.


In our case, we need to consider the extended superposition principle under local integrability bounds, \cite[Theorem 5.1]{ACF17}. Let $B_R \subset \R^d$ and $\eta\in \mathscr P(C([0, T); B_R\times B_R))$ be concentrated on integral curves of the vector field $b$ with the no concentration condition $(e_t)_\# \eta \leq C_0 \mathscr{L}^{2d}$ for any $t\in [0, T]$. Then, arguing as before, by \cite[Theorem 5.1]{ACF17}, in order to show that assumption {\bf (A2)} holds, it is enough to prove that $\eta_x$ is a Dirac delta for ${e_0}_\#\eta$-a.e. $x$.


Let $\delta, \zeta \in (0, 1)$ be two small parameters to be chosen. Let $t \in [0, T]$, $\bar \mu := (e_0)_\# \eta$, and we denote $\gamma(t) = (\gamma^1(t), \gamma^2(t)) \in \R^d\times\R^d$. Define 
\begin{equation}
\label{eq.PHI}
\Phi_{\delta, \zeta}(t) := \iiint \log \left(1+\frac{|\gamma^1(t) -\xi^1(t)|}{\zeta\delta}+\frac{|\gamma^2(t)-\xi^2(t) |}{\delta}\right) d\eta_x(\gamma) d\eta_x(\xi) d\bar \mu(x).
\end{equation}
In order to simplify the notation, we denote $d\mu (x, \xi, \gamma) := d\eta_x (\xi) d\eta_x(\gamma) d\bar \mu(x) $ for $\mu \in \mathscr P (\R^d\times C([0, T) ; \R^d)^2)$. As seen in \cite[Theorem 4.4]{ACF17}, if we assume that $\eta_x$ is not a Dirac delta for $\bar \mu$-a.e. $x$ then there exists some constant $0<a < 2T$ and some $t_0 \in (0, T]$ such that 
\begin{equation}
\label{eq.chicont}
\Phi_{\delta, \zeta}(t_0 ) \geq \frac{a}{2T}\log \left(1+\frac{a}{2\delta T} \right),
\end{equation}
and we want to get a contradiction. By differentiating \eqref{eq.PHI}, we get 
\begin{equation}
\label{eq.ts1}
\frac{d\Phi_{\delta,\zeta}}{dt}(t) \leq \iiint \left(\frac{|b_1 (\gamma^2(t))-b_1 (\xi^2(t))}{\zeta(\delta +|\gamma^2(t) - \xi^2(t)|)}+\frac{\zeta |b_2 (\gamma^1(t))-b_2 (\xi^1(t))}{\zeta\delta +|\gamma^1(t) - \xi^1(t)|}\right) d\mu (x, \xi, \gamma).
\end{equation}
The first term in the previous sum can be bounded by means of the Lipschitz regularity of $b_1$ in $B_R$ as 
\begin{equation}
\label{eq.ts2}
\iiint \frac{|b_1 (\gamma^2(t))-b_1 (\xi^2(t))}{\zeta(\delta +|\gamma^2(t) - \xi^2(t)|)}d\mu (x, \xi, \gamma) \leq \frac{\|\nabla b_1\|_{L^\infty(B_R)}}{\zeta}.
\end{equation}
For the second term, we will see that, as in the proof of \cite[Theorem 4.4]{ACF17}, in order to get a contradiction it is enough to show 
\begin{equation}
\label{eq.toshow}
\iiint \frac{|b_2 (\gamma^1(t)) - b_2 (\xi^1(t))|}{\zeta\delta + |\gamma^1(t) - \xi^1(t)|} d\mu(x, \xi, \gamma) \leq C\left(1+\log\left(\frac{1}{\zeta\delta}\right)\right)
\end{equation}
for some constant $C$ depending only on $d$, $|\rho|(\R^d)$, and $R$. Notice that we just need to bound
\begin{equation}
\label{eq.toshow_2}
\iiint_{\gamma^1_1(t) \xi_1^1(t) < 0} \frac{|b_2 (\gamma^1(t)) - b_2 (\xi^1(t))|}{\zeta\delta + |\gamma^1(t) - \xi^1(t)|} d\mu(x, \xi, \gamma),
\end{equation}
where $\gamma^1_1(t)$ and $\xi_1^1(t)$ denote the first component of $\gamma^1(t)$ and $\xi^1(t)$ respectively. Indeed, if $\gamma^1_1(t)\xi_1^1(t) > 0$ then 
\[
|b_2 (\gamma^1(t)) - b_2 (\xi^1(t))| = |K * \rho (\gamma^1(t)) - K * \rho (\xi^1(t))|,
\]
and we are in the situation treated in the proof of \cite[Theorem 4.4]{ACF17}, where the authors deal with vector fields given by full convolutions. 

Now suppose $j \in \{2,\dots,d\}$ fixed. Let us show that,
\begin{equation}
\iiint_{\gamma^1_1(t) \xi_1^1(t) < 0} \frac{|b_2^{(j)} (\gamma^1(t)) - b_2^{(j)} (\xi^1(t))|}{\zeta\delta + |\gamma^1(t) - \xi^1(t)|} d\mu(x, \xi, \gamma)\leq C\left(1+\log\left(\frac{1}{\zeta\delta}\right)\right),
\end{equation}
where $b_2(x) = (b_2^{(1)}(x), \dots, b_2^{(d)}(x))$. We will also denote $K = (K_1, \dots, K_d)  = (x_1/|x|^d,\dots,x_d/|x|^d)$ (as an abuse of notation, here $x\in \R^d$, while in the integral, $x\in\R^{2d}$).

It follows by noticing that under these hypotheses, $b_2^{(j)}(x) = b_2^{(j)}(x')$, $|b_2^{(j)} (\gamma^1(t)) - b_2^{(j)} (\xi^1(t))| = |b_2^{(j)} (\gamma^1(t)) - b_2^{(j)} \big(\xi^1(t)'\big)| = |K_j * \rho (\gamma^1(t)) - K_j * \rho (\xi^1(t)')|$, and $|\gamma^1(t) - \xi^1(t)| \geq |\gamma^1(t) - \xi^1(t)'|$, so that 
\begin{align*}
\iiint_{\gamma^1_1(t) \xi_1^1(t) < 0}&  \frac{|b_2^{(j)} (\gamma^1(t)) - b_2^{(j)} (\xi^1(t))|}{\zeta\delta + |\gamma^1(t) - \xi^1(t)|} d\mu(x, \xi, \gamma)\leq
\\ & \leq \iiint_{\gamma^1_1(t) \xi_1^1(t) < 0} \frac{|K_j * \rho (\gamma^1(t)) - K_j * \rho (\xi^1(t)')|}{\zeta\delta + |\gamma^1(t) - \xi^1(t)'|} d\mu(x, \xi, \gamma),
\end{align*}
and it follows again as in \cite[Theorem 4.4]{ACF17}. We have critically used here that $\rho$ is odd with respect to $x_1$. The fact that we are integrating with respect to $\xi_1$ and not $\xi_1'$ does not play a role in the proof of \cite[Theorem 4.4]{ACF17}.

In all, we just need to bound 
\begin{align*}
& \iiint_{\gamma^1_1(t) \xi_1^1(t) < 0}  \frac{|K_1 * \rho (\gamma^1(t)) +  K_1 * \rho (\xi^1(t))|}{\zeta\delta + |\gamma^1(t) - \xi^1(t)|} d\mu(x, \xi, \gamma) \leq\\
& ~~~~~~~~~~\leq \iiint_{\gamma^1_1(t) \xi_1^1(t) < 0}  \frac{|K_1 * \rho (\gamma^1(t))|}{\zeta\delta + |\gamma^1_1(t)|} d\mu + \iiint_{\gamma^1_1(t) \xi_1^1(t) < 0}  \frac{|K_1 * \rho (\xi^1(t))|}{\zeta\delta + |\xi^1_1(t)|} d\mu.
\end{align*}
We are using here that, since $\gamma_1^1(t) \xi_1^1(t) < 0$ then $ |\gamma^1(t) - \xi^1(t)| \geq |\gamma_1^1(t)|+|\xi_1^1(t)|$. By symmetry, it will be enough to bound the first term in the previous sum. Thanks to the no-concentration condition $(e_t)_\# \eta \leq C_0 \mathscr{L}^d$,
\[
\iiint_{\gamma^1_1(t) \xi_1^1(t) < 0}  \frac{|K_1 * \rho (\gamma^1(t))|}{\zeta\delta + |\gamma^1_1(t)|} d\mu \leq C_0 \mathscr{L}^d(B_R) \int_{B_R}\frac{|K_1 * \rho(x)|}{\zeta\delta + |x_1|} dx.
\]

Let us find a bound of the kind \eqref{eq.toshow} for $R = 1$ (other values of $R > 0$ follow analogously). We define $I := \int_{B_1} |K_1* \rho(x)| (\zeta\delta + |x_1|)^{-1}dx$, and let $A_k := [0, 2^{-k}]\times B_1^{(d-1)}$, where $B_1^{(d-1)}$ denotes the $d-1$ dimensional unit ball. Define also $U_k := [2^{-k}, 2^{-k+1}]\times B_1^{(d-1)}$. Then, for any $N\in \N$, 
\begin{align}
\label{eq.charact1}I \leq \int_{A_0}  \frac{|K_1 * \rho(x)|}{\zeta\delta + |x_1|} dx & = \left\|\frac{K_1*\rho}{\zeta\delta+|x_1|}\right\|_{L^1(A_0, \mathscr{L}^d)}\\
& \nonumber\leq \frac{1}{\zeta\delta}\left\|K_1*\rho\right\|_{L^1(A_N, \mathscr{L}^d)}+\sum_{k = 1}^N \frac{1}{\zeta\delta + 2^{-k}}\left\|K_1*\rho\right\|_{L^1(U_k, \mathscr{L}^d)}.
\end{align}

Notice that $\left\|K_1*\rho\right\|_{L^1(U_k, \mathscr{L}^d)} \leq \left\|K_1*\rho\right\|_{L^1(A_{k-1}, \mathscr{L}^d)}$. We will see that it is enough to bound $\left\|K_1*\rho\right\|_{L^1(A_{k}, \mathscr{L}^d)}$ for any $k\in \N\cup\{0\}$. Indeed,
\begin{align*}
\left\|K_1*\rho\right\|_{L^1(A_{k}, \mathscr{L}^d)} = \int_{A_k\times\R^d} |K_1(x-z)\rho(z)|\,dz\,dx  \leq 2\int_{\R^d}|\rho|\int_{A_k} |K_1| = 2|\rho|(\R^d)\|K_1\|_{L^1(A_k,\mathscr{L}^d)}.
\end{align*}
Also, with the notation $\|K_1\|_{L^1_k}:= \|K_1\|_{L^1(A_k,\mathscr{L}^d)}$,
\begin{equation}
\label{eq.K1ineq}
\begin{split}
\|K_1\|_{L^1_k}  = \int_0^{2^{-k}}x_1\int_{B_1^{(d-1)}} \frac{dx_2\dots dx_d}{|x|^d} dx_1 = \int_{[0,2^{-k}]\times B_{1/x_1}^{(d-1)}} \frac{dy_2\dots dy_d}{\left(1+y_2^2+\dots+y_d^2\right)^{d/2}} dx_1 \leq C 2^{-k},
\end{split}
\end{equation}
for some constant $C_d$ that depends only on the dimension $d$. Thus, for any $N \geq 1$,
\begin{equation}
\label{eq.charact2}
 I \leq \frac{C}{2^N\zeta\delta}+ C\sum_{k = 1}^N \frac{2^{-k+1}}{\zeta\delta + 2^{-k}} \leq \frac{C}{2^N\zeta\delta} + CN.
\end{equation}
Choose $N$ such that $2^N \zeta\delta = 1$, so that $N = C\log\left(\frac{1}{\zeta\delta}\right)$, then
\[
 I \leq C\left(1 + \log\left(\frac{1}{\zeta\delta}\right)\right).
\]
This proves the bound \eqref{eq.toshow}. From here, one can proceed as in \cite[Theorem 4.4]{ACF17} to get a contradiction with \eqref{eq.chicont}. Indeed, combining \eqref{eq.ts1}, \eqref{eq.ts2}, and \eqref{eq.toshow} we have 
\begin{equation}
\label{eq.enoughuniq}
\frac{d\Phi_{\delta, \zeta}}{dt}(t) \leq C\left(\frac{1}{\zeta}+\zeta+\zeta\log\left(\frac{1}{\zeta\delta}\right)\right),
\end{equation}
for some $C$ that depends only on $d$, $R$, $|\rho|(\R^d)$, and $\|\nabla b_1\|_{L^\infty(\R^d)}$. Integrating from $0$ to $t_0$, and using $\Phi_{\delta,\zeta}(0) = 0$ ($\eta$ is concentrated on curves with fixed initial datum), we reach 
\[
 \Phi_{\delta, \zeta}(t_0) \leq Ct_0 \left(\frac{1}{\zeta}+\zeta+\zeta\log\left(\frac{1}{\delta}\right)+\zeta\log\left(\frac{1}{\zeta}\right)\right).
\]
Now choosing $\zeta > 0$ such that $Ct_0 \zeta < \frac{a}{2T}$, and letting $\delta \downarrow 0$ we get a contradiction with \eqref{eq.chicont}. 
\end{proof}

\section{Existence of solutions for Problem B}
\label{sec.PbB}
In this section we want to show the existence of renormalized solutions to Problem B, and its Lagrangian structure. In particular, we will show the analogous of Theorems~\ref{thm.main1_Om} and \ref{thm.main2_d} for Problem B, \eqref{eq.B}. 

\subsection{A regularised problem}
\label{ssec.regpb}
In this subsection we prove the existence and conservation of energy of solutions to a regularisation of Problem B, \eqref{eq.B}. 

In order to find a solution to Problem B we will need to solve regularised versions of the same problem to generate an approximating sequence. In the classical Vlasov--Poisson in the whole space, the existence and conservation of energy of solutions to a regularised problem (obtained by regularisation of the convolution kernel) is known (see, e.g., \cite{Dob79}), and follows by a fixed point argument in the Wasserstein metric. 

In this case, the same approximation also works. The only detail one has to consider is the choice of the regularisation of Problem B. Moreover, a small error will appear on the conservation of energy coming from the regularisation of the odd density.

\begin{thm}
\label{thm.regvp}
Let $h_0(x, v) = h_0\in C^\infty_c(\R^d\times \R^d)$ be even with respect to $(x_1, v_1)$, that is, $h_0(x, v) =  h_0 (x', v')$. Let $\bar H\in C^\infty(\R^d)$ be a rotationally invariant ($\bar H(x) = \bar H (|x|)$) regularisation of $H(x) = \frac{c_d}{d-2}|x|^{2-d}$, and let $\bar s\in C^\infty(\R)$ be an odd regularisation of ${\rm sgn}(x_1)$, the sign function. Then, the following problem has an even (with respect to $(x_1, v_1)$) distributional solution $\bar g_t= \bar g(t, x, v)$,
\begin{equation}
\label{eq.B_reg}
  \left\{ \begin{array}{ll}
  \de_t \bar g_t + v\cdot \nabla_x \bar g_t + \bar E_t\cdot\nabla_v \bar g_t=0 & \textrm{ in } (0,\infty)\times\R^d\times\R^d\\
    \bar \rho_t(x)=\int_{\R^d} \bar g_t(x, v) dv,\quad \quad \bar \rho_t^o(x) = {\rm sgn}(x_1) \bar\rho_t(x) & \textrm{ in } (0,\infty)\times\R^d,\\
\bar E_t(x)= -\bar s (x_1) \big(\nabla \bar H * (\bar \rho^o_t) \big)(x)  \\
~~~~~~~~~~~~~~~~~~= - \bar s (x_1)  \int \nabla \bar H (x-y)\bar \rho_t^o(y)dy   & \textrm{ in } (0,\infty)\times\R^d\\
\bar g_0 = h_0& \textrm{ in } \R^d\times\R^d.\\
  \end{array}\right.
\end{equation}
Moreover,
\begin{align}
\label{eq.energy}\int_{\R^d\times\R^d} & |v|^2 \bar g_t(x, v)\,dx\,dv  + \int_{\R^d} \Big[\bar H * \bar \rho_t^o\Big](x)\bar \rho_t^o(x)\,dx = \\
& \nonumber = \int_{\R^d\times\R^d} |v|^2 \bar g_0(x, v)\,dx\,dv + \int_{\R^d} \Big[\bar H * \bar \rho_0^o\Big](x) \bar \rho_0^o(x)\,dx + \int_0^t \mathcal{K}_\tau(\bar s, \bar H, h_0)\,d\tau
\end{align}
 for any $t >0$, where $\mathcal{K}_\tau$ is defined as
 \begin{equation}
 \label{eq.mk}
 \mathcal{K}_\tau(\bar s, \bar H, h_0) = 2 \int_{\R^d\times\R^d} \big({\rm sgn}(x_1)- \bar s(x_1) \big) \,v\cdot \Big[\nabla \bar H * \bar \rho_\tau^o\Big](x)\bar g_\tau \,dx\,dv,
 \end{equation}
for $\bar g_\tau$ the solution to \eqref{eq.B_reg} with initial datum $h_0$. 
\end{thm}
\begin{proof}
We divide the proof into the two parts of the statement. In the first step we prove the existence of a distributional solution, while in the second step we check that the energy defined in \eqref{eq.energy} is conserved in time up to an error. 
\\[0.2cm]
{\bf Step 1: Existence.} To prove existence we proceed with a standard fixed point argument where we build functions iteratively that converge to a distributional solution to the previous problem. 

Let $T_0 > 0$ to be chosen, and let $\mu_t^{n+1}:(0, T_0)\times\R^d\times\R^d$ for $n \in \N\cup\{0\}$ defined iteratively as the solution to  
\begin{equation}
\label{eq.B_reg_n}
  \left\{ \begin{array}{ll}
  \de_t \mu^{n+1}_t + v\cdot \nabla_x \bar \mu^{n+1}_t + E_t^n\cdot\nabla_v \mu^{n+1}_t=0 & \textrm{ in } (0,T_0)\times\R^d\times\R^d\\
    \rho_t^n(x)=\int_{\R^d} \mu^{n}_t(x, v) dv,\quad  \rho_t^{o,n}(x)={\rm sgn}(x_1)\rho_t^n(x) & \textrm{ in } (0,T_0)\times\R^d\\
E^n_t(x)= -\bar s (x_1) \big(\nabla \bar H * (\rho^{o,n}_t) \big) & \textrm{ in } (0,T_0)\times\R^d\\
\mu^{n+1}_0 = h_0& \textrm{ in } \R^d\times\R^d,\\
  \end{array}\right.
\end{equation}
with $\mu^0_t = h_0$ for $t \in (0, T_0)$. By standard Cauchy-Lipschitz theory, if $b^n_t(x, v) = (v, E^n_t(x))$, then there exists a regular flow $Z_n : [0, T_0]\times\R^d\times\R^d\to \R^d\times\R^d$, $Z_n(t) = (X_n(t), V_n(t))$, such that it solves
\begin{equation}
  \left\{ \begin{array}{ll}
\frac{d}{dt}Z_n(t) = b_t^n(Z_n(t))& \quad \textrm{in}\quad (0, T_0)\times\R^d\times\R^d\\
    Z_n(0)(x, v) = (x, v),& \quad \textrm{in}\quad \R^d\times\R^d.\\
  \end{array}\right.
\end{equation}
and the solution $\mu_t^{n+1}$ is given by the push-forward $\mu_t^{n+1} = Z_n(t)_\# h_0$.

We will prove that $\mu_t^{n+1}$ converge to some $\mu_t$ a distributional solution to the continuity equation. Since each $\mu_t^{n+1}$ is even with respect to $(x_1, v_1)$  by construction and uniqueness of \eqref{eq.B_reg_n}, if the limit $\mu_t$ exists, it  must also be even with respect to $(x_1, v_1)$.

To do so, we study the convergence of the flows in the $L^1$ norm. Before doing that, let us define the following distance from the Wasserstein metric $W_1$. That is, given $\nu_1, \nu_2 \in L^\infty((0, T_0); \mathcal{M}_+(\R^d\times\R^d))$ such that $\nu_1(t)(\R^d\times\R^d) = \nu_2(t)(\R^d\times\R^d) = C$ with $C$ independent of time, we define 
\begin{align*}
W_1^{T_0}(\nu_1, \nu_2) & := \sup_{t\in [0, T_0]} W_1(\nu_1(t), \nu_2(t))\\
& = \sup_{t\in [0, T_0]} \sup\left\{\int_{\R^d\times\R^d} \varphi (x, v) d(\nu_1 - \nu_2)(x, v) : {\rm Lip}(\varphi) \leq 1\,\textrm{in}\,\,\R^d\times\R^d\right\}.
\end{align*}
Analogously we also define $W_1^{T_0}(\rho_1,\rho_2)$ for $\rho_1, \rho_2 \in L^\infty((0, T_0); \mathcal{M}_+(\R^d))$. We want to compute $W_1^{T_0}(\mu^{n+1}, \mu^n)$ (notice that conservation of mass for $\mu^n$ follows from the fact that the vector field $(v, E_t^n)$ is divergence-free). Let us call $L_1^d := \{\varphi \in C^0(\R^d) : {\rm Lip}(\varphi) \leq 1\}$, and fix $t\in [0, T_0]$.

Notice that
\begin{align}
\label{eq.w1pp} W_1( \mu^{n+1}_t ,\mu^n_t)   
 & = \sup_{\varphi\in L_1^{2d}} \int \left\{\varphi(Z_n(t))-\varphi(Z_{n-1}(t)) \right\}dh_0 \leq \int |Z_n(t)-Z_{n-1}(t)|dh_0,
\end{align}
so that, in particular, 
\begin{equation}
\label{eq.w1pp2} 
\mathscr{Z}_n^{T_0}:= \sup_{t\in[0,T_0]} \int |Z_n(t)-Z_{n-1}(t)|dh_0 \geq W_1^{T_0}( \mu^{n+1}_t ,\mu^n_t).
\end{equation}

This is the term whose convergence we want to study. On the other hand, 
\begin{align*}
\int |   Z_n  (t) -Z_{n-1}&(t)|\,dh_0  \le \int \left\{\int_0^t |b_s^n(Z_n(s))-b_s^n(Z_{n-1}(s))|ds\right\}dh_0\\
& \leq t \mathscr{Z}_n^{t} + \int_0^t \left\{ \int\left|E_s^n(X_{n}(s)) - E_s^{n}(X_{n-1}(s)) \right|\,dh_0 +\int \left|E_s^n - E_s^{n-1}\right| d\mu_s^{n}\right\}ds,
 \end{align*}
from which
\begin{equation}
\label{eq.supbs}
\int |   Z_n (t) -Z_{n-1}(t)|\,dh_0 \leq C t \mathscr{Z}_n^{t}+ th_0(\R^{2d})\sup_{s\in [0, t], x\in \R^d} |E_s^n(x) - E_s^{n-1}(x)|.
\end{equation}
We have used here that $E_s^n$ are uniformly Lipschitz independently of $s$ and $n$ by construction, and that the total mass is fixed for any $n$ and for all times in $[0, T_0]$. The constant $C$, then, is fixed depending only on the regularised functions $\bar H$ and $\bar s$. 

We can now compute  
\begin{align*}
&\sup_{s\in [0, t], x\in \R^d} |E_s^n(x)  - E_s^{n-1}(x)| = \\
& ~~~~~~~~~ = \sup_{s\in [0, t], x\in \R^d} \left|\int \bar s(x_1)\nabla \bar H(x-y)\left[\rho^{o,n}_s(y)-\rho^{o,n-1}_s(y)\right] dy\right| \leq CW_1^t(\rho^{o,n}, \rho^{o,n-1}),
\end{align*}
where we have used that $\bar s$ and $\nabla \bar H$ are globally Lipschitz, and they do not depend on $n$. Let us now see that, for any $s\in [0,t]$,
\begin{equation}
\label{eq.arg1}
W_1(\rho^{o,n}(s), \rho^{o,n-1}(s)) \leq W_1(\rho^{n}(s), \rho^{n-1}(s)).
\end{equation}

Indeed, since $\rho^{o,n}_s-\rho^{o,n-1}_s$ is odd with respect to $x_1$, then for any $\varphi \in L_1^d$, 
\begin{align*}
\int_{\R^d}\varphi(y)\left(\rho^{o,n}_s-\rho^{o,n-1}_s\right)(y) dy & = \int_{\R^d}-\varphi(y')\left(\rho^{o,n}_s-\rho^{o,n-1}_s\right)(y) dy \\
& = \int_{\R^d}\frac{\varphi(y)-\varphi(y')}{2}\left(\rho^{o,n}_s-\rho^{o,n-1}_s\right)(y) dy,
\end{align*}
and therefore
\begin{align*}
W_1(\rho^{o,n}(s), \rho^{o,n-1}(s)) 
& = \sup_{\varphi\in L_1^d} \left\{\int_{\R^d}{\rm sgn}(y_1)\frac{\varphi(y)-\varphi(y')}{2}\left(\rho^{n}_s-\rho^{n-1}_s\right)(y) dy\right\}  \leq W_1(\rho^{n}(s), \rho^{n-1}(s)),
\end{align*}
where we have used that if $\varphi\in L_1^d$ then ${\rm sgn}(y_1)\frac{\varphi(y)-\varphi(y')}{2}\in L_1^d$. Thus,
\[
\sup_{s\in [0, t], x\in \R^d} |E_s^n(x)  - E_s^{n-1}(x)|\leq C W_1^t(\rho^{o,n}, \rho^{o,n-1}) \leq C W_1^t(\mu^{n}, \mu^{n-1})\leq C\mathscr{Z}_{n-1}^t,
\]
where in the last inequality we have used \eqref{eq.w1pp2}.

Putting all together and taking the supremum over $t\in [0, T_0]$ in \eqref{eq.supbs} we get 
\[
\mathscr{Z}^{T_0}_n \leq \frac{CT_0}{1-CT_0} \mathscr{Z}^{T_0}_{n-1}
\]
for some constant $C$ depending only on the choice of regularisation functions $\bar s$ and $\bar H$, and on $h_0(\R^{2d})$. Thus, for some $T_0>0$ small enough we have constructed a contraction for the sequence $(\mathscr{Z}^{T_0}_n)_n$, and since $T_0$ is independent of $n$ we can repeat the argument to reach any positive time. In particular, this yields the $L^1$ convergence of flows with respect to $n$. Using the bound \eqref{eq.w1pp2}, we get a limiting measure $\mu_t$ of the sequence $\mu^n_t$. 

One can easily check that $E^n_t(x)$ converge uniformly (in time and space) to 
\[
E_t(x) = \bar s(x_1) \left\{\nabla \bar H * \rho_t^o\right\}(x),
\]
where $\rho_t^o = {\rm sgn}(x_1)\int \mu_t dv$. Thus, taking limits in \eqref{eq.B_reg_n} we obtain that $\mu_t$ solves \eqref{eq.B_reg} in the distributional sense, and is even with respect to $(x_1, v_1)$ by construction, so that we have constructed our solution $\bar g_t$. Note, moreover, that by Cauchy-Lipschitz theory, $\bar g_t$ is smooth.
\\[0.2cm]
{\bf Step 2: Conservation of energy.} This is standard. We refer the reader to the proof of Step 2 in Theorem~\ref{thm.regvp_d} for a similar situation.
\end{proof}

\subsection{Existence of solutions}
We start by introducing a rather general result involving either bounded or renormalized solutions to the reflected Vlasov--Poisson problem \eqref{eq.B}. 

This theorem essentially uses the results of \cite{ACF17}, where the authors establish a general principle on the conditions necessary to have equivalence between renormalized and Lagrangian solutions. They present an analogous statement for the Vlasov--Poisson system without boundary in \cite[Theorem 2.2]{ACF17}.

In particular we use \cite[Theorem 5.1]{ACF17}. This result proved by Ambrosio, Colombo, and Figalli, states that bounded or renormalized solutions to a continuity equation whose vector field satisfies certain conditions, are transported by the Maximal Regular Flow. Thus, we just need to check that solutions to our reflected Vlasov--Poisson system \eqref{eq.B} fulfil the hypotheses from \cite[Theorem 5.1]{ACF17}.

\begin{thm}
\label{thm.22}
Let $T > 0$, and $g_t \in L^\infty((0, T); L^1_+(\R^{2d}))$ a weakly continuous function. Suppose that 
\begin{enumerate}[(i)]
\item either $g_t \in L^\infty((0, T); L^\infty(\R^{2d}))$ is a distributional solution to the reflected Vlasov--Poisson system, Problem B \eqref{eq.B}.
\item or $g_t$ is a renormalized solution of the reflected Vlasov--Poisson system, Problem B \eqref{eq.B} (see Definition~\ref{defi.WS}). 
\end{enumerate}
Then $g_t$ is a Lagrangian solution transported by the Maximal Regular Flow associated to the vector field $b_t$; and in particular, $g_t$ is renormalized. 
\end{thm}
\begin{proof}
To prove it we simply apply \cite[Theorem 5.1]{ACF17} noting that the vector field $b_t$ fulfils the conditions {\bf (A1)} from Section~\ref{sec.4} and {\bf (A2)}, as proved in Theorem~\ref{thm.charact}; and therefore, the solution is transported by the Maximal Regular Flow. 

In particular, by \cite[Theorem 4.10]{ACF17}, a solution transported by the Maximal Regular Flow is renormalized. 
\end{proof}

\begin{thm}
\label{thm.mainws}
Let $d \geq 3$, and consider $g_0 \in L_+^1(\R^{2d})$ even with respect to $(x_1, v_1)$, $\rho_0^o (x) = {\rm sgn}(x_1) \int_{\R^d} g_0(x, v) dv$, satisfying
\begin{equation}
\label{eq.hyp}
\int_{\R^d\times\R^d} |v|^2 g_0(x, v)\, dx\, dv + \int_{\R^d}H*\rho_0^o \,\rho_0^o\,dx < \infty,\quad\quad H(x) = \frac{c_d}{d-2} |x|^{2-d}.
\end{equation}
Then, there exists a global Lagrangian solution (transported by the Maximal Regular Flow) even with respect to $(x_1, v_1)$, $g_t \in C([0, \infty) ; L^1_{\rm loc}(\R^{2d})) $, of the reflected Vlasov--Poisson system \eqref{eq.B} with initial datum $g_0$. Moreover, the physical density $\rho_t = \int g_t dv$ and the electric field $E_t = {\rm sgn}(x_1)\, \rho_t^o * K $ are strongly continuous in $L^1_{\rm loc}(\R^d)$; $\rho_t, E_t \in C([0, \infty) ; L^1_{\rm loc}(\R^{d}))$.
\end{thm}
\begin{proof}
The proof follows along the lines of the proof of \cite[Corollary 2.7]{ACF17} with the modifications introduced until now for our vector field. In this case, the choice of the approximating sequence plays a more relevant role. We divide the proof into several steps.
\\[0.2cm]
{\bf Step 1: Approximating sequence.} Let $(H_n)_{n\in \N}$ with $H_n(x) = (\psi_n * H)(x) \in C^\infty(\R^d)$ be a sequence of even functions approximating $H(x)$ with $\psi_n(x) = n^d\psi(nx)$ and $\psi$ a standard rotational invariant convolution kernel in $C^\infty_c(\R^d)$, $\psi(x) = \psi(|x|)$, decreasing with respect to $|x|$. Let $(g_0^n)_{n\in \N}$ with $g_0^n \in C^\infty_c(\R^d)$ a sequence of nonnegative even functions with respect to $(x_1, v_1)$ approximating $g_0$; that is
\[
\begin{split}
H_n \to H\quad\quad & \textrm{in}~L^1(\R^{d}),
\\
g_0^n \to g_0\quad\quad  & \textrm{in}~L^1(\R^{2d}).
\end{split}
\]
Suppose also that $(\bar s_n)_{n\in \N}$ with $\bar s_n = \bar s_n(x_1)\in C^\infty(\R)$ is a sequence of functions approximating ${\rm sgn}(x_1)$, such that, for a positive sequence $r_n \downarrow 0$ as $n\to \infty$ to be chosen later, we have
\begin{align}
\nonumber & |\bar s_n| \leq 1,\quad \textrm{and}\quad \bar s_n (x_1) = - \bar s_n(-x_1),\quad\textrm{for}~ x_1\in \R\\
\label{eq.rn}& \bar s_n (x_1) \equiv 1,\quad\textrm{for}~ x_1\geq r_n,\quad\quad\bar s_n (x_1) \equiv -1,\quad\textrm{for}~ x_1\leq -r_n,\\
\nonumber & \bar s_n(x_1)\to {\rm sgn}(x_1)\quad\textrm{in}~L^1(\R).
\end{align}

Denote by $g_t^n$ the solutions (even with respect to $(x_1, v_1)$) of the regularised reflected Vlasov--Poisson system constructed in Theorem~\ref{thm.regvp},
\begin{equation}
\label{eq.B_reg_2}
  \left\{ \begin{array}{ll}
  \de_t g_t^n + v\cdot \nabla_x g_t^n + E_t^n\cdot\nabla_v g_t^n=0 & \textrm{ in } (0,\infty)\times\R^d\times\R^d\\
    \rho_t^n(x)=\int_{\R^d} g_t^n(x, v) dv,\quad \quad \rho_t^{n,o}(x) = {\rm sgn}(x_1) \rho_t^n(x) & \textrm{ in } (0,\infty)\times\R^d,\\
E_t^n(x)= -\bar s_n (x_1) \big(\nabla H_n * \rho^{n,o}_t \big)(x) & \textrm{ in } (0,\infty)\times\R^d,\\
  \end{array}\right.
\end{equation}
with initial datum $g_0^n$. Using the notation of Theorem~\ref{thm.regvp}, we are taking $\bar H = H_n$ and $\bar s = \bar s_n$. 

Notice that the vector field $b^n_t(x, v) = (v, E_t^n(x))$ is Lipschitz and divergence-free, and therefore, by standard Cauchy-Lipschitz theory, there exists a well defined and incompressible flow $Z^n(t): \R^{2d}\to \R^{2d}$ transporting the solution,
\begin{equation}
\label{eq.ex1}
g_t^n = g_0^n\circ Z^n(t)^{-1},\quad\textrm{for}~t\in (0, \infty),
\end{equation}
and
\begin{equation}
\|\rho_t^n\|_{L^1(\R^d)} = \|g_t^n\|_{L^1(\R^{2d})} = \|g_0^n\|_{L^1(\R^{2d})}. 
\end{equation}

In particular, by assuming that $g_0^n$ are equiintegrable with respect to $n$, we have that $g_t^n$ are equiintegrable independently of $n\in \N$, $t\in (0, \infty)$; but more importantly, independently of the choice $H_n$ and $\bar s_n$. That is, there exists a sequence $(\eps_m)_{m\in \N}$ with $\eps_m \downarrow 0$ as $m \to \infty$ such that
\begin{equation}
\label{eq.ei}
\int_{\R^{d}\times\R^d} g_t^n 1_{\{g_t^n > m\}} \,dx\, dv = \int_{\R^{d}\times\R^d} g_0^n 1_{\{g_0^n > m\}} \,dx\, dv \leq \eps_m\to 0,\quad\textrm{as}~m\to \infty,
\end{equation}
for all $t\in (0, \infty)$ and  $n\in \N$. The sequence $(\eps_m)_{m\in \N}$ depends only on the initial datum, $g_0$.
\\[0.2cm]
{\bf Step 2: Choice of the approximating sequence.} In this step we choose the approximating sequence in such a way that we keep a control on the kinetic energy of the system. Recall that, from \eqref{thm.regvp}, we have 
\begin{align}
\label{eq.energy2}\int_{\R^d\times\R^d} & |v|^2 g_t^n(x, v) dx dv  + \int_{\R^d} \Big[H_n * \rho_t^{n,o}\Big](x)\rho_t^{n,o}(x) dx = \\
& \nonumber = \int_{\R^d\times\R^d} |v|^2 g_0^n(x, v) dx dv + \int_{\R^d} \Big[H_n * \rho_0^{n,o}\Big](x) \rho_0^{n,o}(x) dx + \int_0^t \mathcal{K}_\tau(\bar s_n, H_n, g_0^n) d\tau
\end{align}
 for any $t >0$, where $\mathcal{K}_\tau$ for $\tau \in (0, t)$ is defined as
 \begin{equation}
 \label{eq.mk2}
 \mathcal{K}_\tau(\bar s_n,H_n, g_0^n) = 2 \int_{\R^d\times\R^d} \big({\rm sgn}(x_1) - \bar s_n(x_1) \big)\,v\cdot \Big[\nabla H_n * \rho_\tau^{n,o}\Big](x)g_\tau^n \,dx\,dv.
 \end{equation}

Proceeding as in \cite[Lemma 3.1]{ACF17} there exists a sequence $g_0^n \in C^\infty_c(\R^{2d})$ and $H_{k_n}$ for $k_n\to\infty$ as $n\to \infty$, such that 
\begin{equation}
\label{eq.lem31}
\lim_{n\to \infty} \left(\int_{\R^{2d}} |v|^2 g_0^n\,dx\,dv + \int_{\R^d} H_{k_n} * \rho_0^{n,o} \rho_0^{n,o} dx \right) = \int_{\R^{2d}} |v|^2 g_0\,dx\,dv + \int_{\R^d} H * \rho_0^{o}\,\rho_0^{o}\,dx.
\end{equation}
Notice that we can assume, without loss of generality after relabelling the indexes, that ${\rm supp}\, g_0^n \subset B_{M_n}(0)$ and $\|\nabla H_{k_n}\|_{L^\infty(\R^d)} \leq M_n$ for a given sequence $M_n \to \infty$ as $n\to \infty$. Moreover, since
\[
\frac{d}{dt} Z^n(t) = b_t^n(Z^n(t)),
\]
with $b_t^n(x, v) = (v, E_t^n(x))$, and $\|E_t^n\|_{L^\infty(\R^d)} \leq \|g_0^n\|_{L^1(\R^{2d})} M_n$, we have 
\[
\frac{d}{dt} |Z^n(t)| \leq |Z^n(t)| + M_n.
\]

Together with \eqref{eq.ex1} and ${\rm supp}\, g_0^n \subset B_{M_n}(0)$, this implies ${\rm supp}\, g_t^n \subset B_{2M_n e^t}(0)$. We can now bound $|\mathcal{K}_\tau|$,
\begin{align*}
| \mathcal{K}_\tau(\bar s_n,H_n, g_0^n) | & \leq 2 \int_{{\rm supp}\,g_\tau^n} \big|\bar s_n(x_1) - {\rm sgn}(x_1)\big|\cdot |v|\cdot \Big|\nabla H_n * \rho_\tau^{n,o}\Big|(x)|g_\tau^n| \,dx\,dv\\
& \leq 2 \|\nabla H_n \|_{L^\infty}\|\rho_\tau^{n,o}\|_{L^1}\int_{{\rm supp}\,g_\tau^n} |v|\cdot |g_\tau^n| \,dx\,dv\\
& \leq 4M_n^2 e^\tau \|g_0^n\|_{L^1(\R^{2d})}\int_{\{-r_n\leq x_1\leq r_n\}\cap\{{\rm supp}\,g_\tau^n\}} |g_\tau^n|\, dx\, dv.
\end{align*}
Now notice that, for every $m\in\N$, there exists $\eps_m\downarrow 0$ as $m\to \infty$ coming from \eqref{eq.ei} such that
\begin{align*}
\int_{\{-r_n\leq x_1\leq r_n\}\cap\{{\rm supp}\,g_\tau^n\}} |g_\tau^n|\, dx\, dv & \leq m\mathscr{L}^{2d} \left(\{-r_n\leq x_1\leq r_n\}\cap\{{\rm supp}\,g_\tau^n\}\right)+\eps_m\\
& \leq 2^{2d} m r_n e^{(2d-1)\tau} M_n^{2d-1}+\eps_m.
\end{align*}

Fix $m = n$, and putting all together, 
\[
| \mathcal{K}_\tau(\bar s_n,H_n, g_0^n) | \leq 4e^\tau \|g_0^n\|_{L^1{(\R^{2d})}} \left(2^{2d} n r_n e^{(2d-1)\tau} M_n^{2d+1}+M_n^2\eps_n\right).
\]

Now choose $M_n = \eps_n^{-1/4} \to \infty$ as $n\to \infty$, and $r_n = n^{-2}\eps_n^{\frac{2d-1}{4}} \to 0$ as $n\to \infty$, so that, for $\tau \in (0,t)$,
\begin{equation}
\label{eq.limk}
| \mathcal{K}_\tau(\bar s_n,H_n, g_0^n) | \to 0,\quad\textrm{as}\quad n \to \infty,
\end{equation}
uniformly for $\tau \in (0, t)$.

In particular, from \eqref{eq.energy2}, we have that for the sequence of functions constructed 
\begin{equation}
\label{eq.boundenergy_0}
\int_{\R^d\times\R^d} |v|^2 g_t^n(x, v)\,dx\,dv  + \int_{\R^d} \Big[H_n * \rho_t^{n,o}\Big](x)\rho_t^{n,o}(x)\,dx \leq C+ \int_0^t \mathcal{K}_\tau(\bar s_n, H_n, g_0^n) d\tau,
\end{equation}
for some constant $C$ independent of $n$ and $t$, thanks to \eqref{eq.lem31}, and the hypothesis \eqref{eq.hyp}. From here, a uniform bound on the kinetic energy follows by noting that 
\[
{\rm sgn} \left( \Big[H_n * \rho_t^{n,o}\Big](x) \right) = {\rm sgn} (x_1) = {\rm sgn} \left( \rho_t^{n,o}(x)\right),
\]
so that the second term of the sum in the left-hand side of \eqref{eq.boundenergy_0} is nonnegative, and therefore, 
\begin{equation}
\label{eq.boundenergy_2}
\int_{\R^d\times\R^d} |v|^2 g_t^n(x, v)\,dx\,dv \leq C+ \int_0^t \mathcal{K}_\tau(\bar s_n, H_n, g_0^n) d\tau,
\end{equation}
for some $C$ independent of $n$ and $t$. Thus, for any fixed $t$, using \eqref{eq.limk} we obtain
\begin{equation}
\label{eq.boundke}
\liminf_{n\to\infty} \int_{\R^d\times\R^d} |v|^2 g_t^n(x, v)\,dx\,dv  \leq C,
\end{equation}
for some constant $C$ independent of $n$ and $t$. 
\\[0.2cm]
{\bf Step 3: Limiting solution.}
Once we have the approximating sequence, we need to build the limiting solution. To do so, we proceed as in the proof of \cite[Theorem 2.6]{ACF17}, so that we can look at each approximating solution as a transport for each level set. That is, for every $k \in \N$ (without loss of generality, $\mathscr{L}^{2d} (\{f_0 = k\} )= 0$) we have 
\begin{equation}
\label{eq.g0k}
g_0^{n,k} := 1_{\{k \leq g_0^n < k+1\}}g_0^n \to \bar g_0^k := 1_{\{k \leq g_0 < k+1\}}g_0,\quad\textrm{in}~L^1(\R^{2d}),
\end{equation}
so that, from \eqref{eq.ex1}, for any $n, k \in \N$, $t\in (0, \infty)$, 
\[
g_t^{n,k} := 1_{\{k\leq g_0^n\circ Z^n(t)^{-1} < k+1\}}g_0^n\circ Z^n(t)^{-1}
\]
is a distributional solution of the continuity equation with vector field $b_t^n$; and 
\[
\|g_t^{n,k}\|_{L^1(\R^{2d})} = \|g_0^{n,k}\|_{L^1(\R^{2d})}.
\]

By construction, for each $n, k\in \N$, $g^{n, k}$ is nonnegative and bounded by $k+1$, so that there exists some $\bar g^k \in L^\infty((0, \infty)\times \R^{2d})$ nonnegative such that, up to subsequences,
\begin{equation}
\label{eq.gk}
g^{n,k}\rightharpoonup \bar g^k\quad\quad \textrm{weakly}^*\textrm{ in } L^\infty((0, \infty)\times\R^{2d}) \textrm{ as } n\to \infty,\quad\textrm{for all }k\in \N.
\end{equation}
Proceeding as in \cite[Theorem 2.6]{ACF17},
\[
\|\bar g^k_t \|_{L^1(\R^{2d})} \leq \|g_0^k \|_{L^1(\R^{2d})}\quad\textrm{ for all } t\in [0, \infty).
\]

Defining 
\[
g := \sum_{k = 0}^\infty \bar g^k \quad \textrm{ in } (0, \infty)\times\R^{2d},
\]
then, again as in \cite[Theorem 2.6]{ACF17}, we have 
\begin{equation}
\label{eq.convseries}
\|g_t\|_{L^1(\R^{2d})} \leq \sum_{k \geq 0} \|\bar g_t^k\|_{L^1(\R^{2d})} \leq \sum_{k \geq 0} \|g_0^k\|_{L^1(\R^{2d})} = \|g_0\|_{L^1(\R^{2d})}\quad\textrm{ for a.e. } t\in [0, \infty),
\end{equation}
and 
\begin{equation}
\label{eq.weakg}
g^n \rightharpoonup g\quad\textrm{ weakly in }L^1_{\rm loc}([0, T]\times\R^{2d}),
\end{equation}
for every $T > 0$. 
\\[0.2cm]
{\bf Step 4: Limiting densities.} Let us now study what happens in the limit of the sequence of densities $\{\rho^n\}_{n\in \N}$. Since $\rho^n$ are bounded in $L^\infty((0, \infty);\mathscr{M}_+(\R^d))$, we already know that they converge, up to subsequences, weakly$^*$ in $L^\infty((0, \infty);\mathscr{M}_+(\R^d))$, to some $\rho^* \in L^\infty((0, \infty);\mathscr{M}_+(\R^d))$. This is not enough, as we would like to identify the limit. 

Let us define
\[
\rho_t(x) = \int_{\R^d} g_t(x, v)\,dv,\quad\textrm{ for } x\in \R^d, t\in [0, \infty),
\]
and let us show that for some subsequence, the limit $\rho^*$ coincides with $\rho$. We will prove a stronger result, namely, 
\begin{equation}
\label{eq.limdens}
\lim_{n\to\infty} \int_0^T \int_{\R^d} \varphi \rho_t^n\,dx\,dt = \int_0^T \int_{\R^d} \varphi\rho_t \, dx\,dt,\quad \textrm{ for all } \varphi \in L^\infty_c ([0, T]\times\R^{d}),
\end{equation}
for any $T> 0$, and where $L^\infty_c$ denotes the space of $L^\infty$ functions with compact support.

To do so, we have to exploit that we already know $g_t^n$ is weakly converging to $g$ in $L^1$, and that from the bound on the kinetic energy, problems do not arise from integrating the $v$ variable. 

First of all, from the lower semicontinuity of the kinetic energy and from \eqref{eq.boundke}, we have that for any $T > 0$, 
\begin{equation}
\label{eq.liminf}
\int_0^T \int_{\R^d\times\R^d} |v|^2\,g_t\,dx\,dv\,dt \leq \liminf_{n\to\infty}\int_0^T \int_{\R^d\times\R^d} |v|^2\,g_t^n\,dx\,dv\,dt\leq CT,
\end{equation} 
for some $C$ that depends only on the initial bound of the kinetic energy. Let us check \eqref{eq.limdens}. Let us consider, for each $m \in \N$, a nonnegative function $\xi_m\in C^\infty_c(B_{m+1})$ such that $\xi_m  \equiv 1$ in $B_m$ and $0\leq \xi_m\leq 1$ in $\R^d$, and compute
\begin{align*}
\left|\int_0^T \int_{\R^d} \varphi(t, x) \left(\rho_t^n-\rho_t\right)\,dx\,dt \right|& \leq \left| \int_0^T \int_{\R^d\times\R^d} \varphi(t, x) \left(g_t^n-g_t\right)\xi_m(v)\,dv\,dx\,dt\right|\\
& ~~+\int_0^T \int_{B_m^c \times\R^d} |\varphi(t, x) | g_t^n\left(1-\xi_m(v)\right)\,dv\,dx\,dt\\
&~~ +\int_0^T \int_{B_m^c \times\R^d} |\varphi(t, x) | g_t\left(1-\xi_m(v)\right)\,dv\,dx\,dt.
\end{align*}
 Now we take the $\liminf$ in both sides. Note that,
\begin{align*}
\liminf_{n\to \infty} \int_0^T \int_{B_m^c \times\R^d} & |\varphi(t, x) | g_t^n\left(1-\xi_m(v)\right)\,dv\,dx\,dt\leq \\
& \leq \frac{\|\varphi\|_{L^\infty}}{m^2}\liminf_{n\to \infty} \int_0^T \int_{\R^d \times\R^d} |v|^2 g_t^n \,dv\,dx\,dt  \leq CT \frac{\|\varphi\|_{L^\infty}}{m^2}
\end{align*}
for every $m \in \N$, thanks to \eqref{eq.liminf}, and the same occurs for the last term. This, together with the weak convergence of $g^n$ to $g$ in $L^1$, gives the desired result, \eqref{eq.limdens}. We, therefore, have that, up to subsequences and for every $T > 0$,
\begin{equation}
\label{eq.convdens}
\int \rho^n\varphi \to \int \rho\,\varphi \quad\textrm{ for all } \varphi\in L^\infty_c([0, T]\times\R^d).
\end{equation}

Combining this with the fact that $\rho_t^{n,o} = {\rm sgn}(x_1)\rho_t^n$ and $\rho_t^{o} = {\rm sgn}(x_1)\rho_t$, we also get that 
\begin{equation}
\label{eq.convdensodd}
\int \rho^{n,o}\varphi \to \int \rho^o\,\varphi \quad\textrm{ for all } \varphi\in L^\infty_c([0, T]\times\R^d).
\end{equation}
The key point in reaching this conclusion has been the avoidance of accumulation of mass for $\rho$ around $x_1 = 0$, thanks to the bound on the kinetic energy. 
\\[0.2cm]
{\bf Step 5: Limiting vector fields.} Define the limiting electric field $E_t$ as 
\begin{equation}
E_t(x) = -{\rm sgn}(x_1) (\rho_t^o * \nabla H)(x) = {\rm sgn}(x_1) (\rho_t^o * K)\quad\textrm{ for } x \in \R^d, 
\end{equation}
where $K ( x) = c_dx/|x|^d$. At this point, we would like to apply the stability results (analogous to \cite[Theorem II.7]{DL89b}) to each bounded function $\bar g^k$ (defined in \eqref{eq.gk}) to check that they are distributional solutions of the continuity equation with vector field $b_t = (v, E_t)$ and with initial datum $\bar g^k_0$ (defined in \eqref{eq.g0k}). That is, we have to check that 
\begin{equation}
\label{eq.blim}
b^n \rightharpoonup b\quad\textrm{ weakly in } L^1_{\rm loc}((0, \infty)\times\R^{2d}; \R^{2d})
\end{equation}
and 
\begin{equation}
\label{eq.l1conv}
E_t^n (x+h)\to E_t^n (x) \quad\textrm{ as } |h| \to 0,\textrm{ in }L^1_{\rm loc} ((0, \infty); L^1_{\rm loc}(\R^d)),
\end{equation}
uniformly in $n$. To prove it, we proceed along the lines of \cite[Theorem 2.6]{ACF17}. In this step we will prove \eqref{eq.blim}, and in the next step we will prove \eqref{eq.l1conv}.

Recall that the sequence $(H_n)_{n\in \N}$ is formed by terms of the form $H_n(x) = (\psi_n * H)(x)$, for some sequence of convolution kernels in $\R^d$, $\psi_n(x) = \psi_n(|x|)$, converging to a Dirac delta. Let us call $K_n (x) = \nabla H_n(x) = -(\psi_n*K)(x)$, and we start by checking that $\{E^n\}_{n\in\N}$ is bounded, for $1\leq p < \frac{d}{d-1}$, in $L^p_{\rm loc} ((0, \infty)\times\R^d; \R^d)$ uniformly in $n$; which will yield that $b^n$ has a weak limit.

Indeed, applying twice the local version of Young's inequality for convolutions introduced in the first part of the proof of Lemma~\ref{lem.testfunctions},
\begin{align*}
\|E_t^n\|_{L^p(B_R)} & = \|\bar s_n (x_1)  \big(K_n * \rho^{n,o}_t \big)(x) \|_{L^p(B_R)} \leq  \|\psi_n * K * \rho^{n,o}_t \|_{L^p(B_R)}\\
& \leq \|\rho^{n,o}_t \|_{L^1(\R^d)} \|\psi_n * K \|_{L^p(B_R)} \leq \|\rho^{n,o}_t \|_{L^1(\R^d)} \|\psi_n \|_{L^1(\R^d)}\|K \|_{L^p(B_R)}\\
 & \leq \|\rho^{n,o}_0 \|_{L^1(\R^d)} \|K \|_{L^p(B_R)},
\end{align*}
which is bounded independently of $n$ for every $R > 0$. We have proceeded as in the proof Lemma~\ref{lem.testfunctions}, by using that $\psi$ is rotationally invariant and decreasing with respect to $|x|$, and the same occurs with $\psi_n*K$.

Thus, $\{b^n\}_{n\in \N}$ converges weakly in $L^p_{\rm loc}((0, \infty)\times\R^{2d};\R^{2d})$, and we want to check that the limit is, indeed, $b_t = (v, E_t)$. That is, we will show
\begin{equation}
\label{eq.weaklimvf}
\lim_{n \to \infty} \int_0^\infty \int_{\R^d} E_t^n \varphi \, dx\, dt =\int_0^\infty \int_{\R^d} E_t \varphi \, dx\, dt,\quad\textrm{ for all }\varphi \in C^\infty_c((0, \infty)\times\R^d).
\end{equation}

Let
\[
\left|\int_0^\infty \int_{\R^d} (E_t^n-E_t) \varphi \, dx\, dt \right| \leq I_n + II_n + III_n,
\]
with
\begin{align*}
I_n & = \left|\int_0^\infty \int_{\R^d} \left\{({\rm sgn}(x_1)\,\varphi)\,K * \rho_t^{n,o} -({\rm sgn}(x_1)\,\varphi)\,K * \rho_t^{o}\right\}   \, dx\, dt \right|\\
& = \left|\int_0^\infty \int_{\R^d} \left\{\rho_t^{n,o}\,K *({\rm sgn}(x_1)\,\varphi)  - \rho_t^{o}\,K * ({\rm sgn}(x_1)\,\varphi)\right\}   \, dx\, dt \right|,
\end{align*}
\begin{align*}
II_n & =  \left|\int_0^\infty \int_{\R^d} \left\{({\rm sgn}(x_1)\, \varphi) K_n * \rho_t^{n,o} -({\rm sgn}(x_1)\,\varphi) K * \rho_t^{n,o}\right\}  \, dx\, dt \right|\\
& =  \left|\int_0^\infty \int_{\R^d} \left\{ \rho_t^{n,o}\, K_n * ({\rm sgn}(x_1)\, \varphi) -\rho_t^{n,o} K * ({\rm sgn}(x_1)\,\varphi) \right\}  \, dx\, dt \right|,
\end{align*}
and
\begin{align*}
III_n & =  \left|\int_0^\infty \int_{\R^d} \left\{(\bar s_n \varphi)\, K_n * \rho_t^{n,o} -({\rm sgn}(x_1)\varphi)\, K_n * \rho_t^{n,o} \right\}  \, dx\, dt \right|\\
& =   \left|\int_0^\infty \int_{\R^d} \left\{\,\rho_t^{n,o} K_n * (\bar s_n \varphi) -\, \rho_t^{n,o} ~K_n * ({\rm sgn}(x_1)\varphi) \right\}  \, dx\, dt \right|,
\end{align*}
where we have used standard convolution properties and the fact that $K(x) = -K(-x)$ and $K_n (x) = -K_n(-x)$.

In order to bound the first term, notice that $K *({\rm sgn}(x_1)\,\varphi)$ is bounded and continuous (being the convolution of an $L^1$ function and an $L^\infty$ function), decaying at infinity and with compact support in time; thus, from the weak convergence \eqref{eq.convdensodd} we get 
\[
I_n \to 0 \quad\textrm{ as }n \to \infty.
\]

For the second and third term, we start by claiming that, for any fixed $\varphi \in C^\infty_c((0, \infty)\times\R^d)$,
\begin{equation}
\label{eq.claimpf}
\lim_{n\to \infty} \left\|K*(\varphi\,\bar s_n(x_1)) - K*(\varphi\,{\rm sgn}(x_1))\right\|_{L^\infty((0, \infty)\times\R^d)} = 0,
\end{equation}
and 
\begin{equation}
\label{eq.claimpf_new}
\lim_{n\to \infty} \left\|K_n*(\varphi\,\bar s_n(x_1)) - K_n*(\varphi\,{\rm sgn}(x_1))\right\|_{L^\infty((0, \infty)\times\R^d)} = 0,
\end{equation}
and in particular, $(\varphi\,{\rm sgn}(x_1))*K$ is continuous.

Indeed, denoting $\R^d\Supset D_n = \{-r_n \leq x_1 \leq r_n \} \cap \left(\cup_{t>0}\,{\rm supp}~\varphi_t\right)$, with $r_n$ from \eqref{eq.rn}, then 
\begin{align*}
\Big|\int_{\R^d} &( \bar s_n(y_1) - {\rm sgn}(y_1))  \varphi(t, y) K_n(x-y)\,dy\Big| \leq\\
& \leq  2 \|\varphi\|_{L^\infty((0, \infty)\times\R^d)} \int_{\R^d} \psi_n(z) \int_{D_n} K(x-y-z) dy \,dz\to 0\quad\textrm{ as } n \to \infty,\textrm{ uniformly in } x \in \R^d.
\end{align*}

It similarly follows that 
\begin{equation}
\label{eq.claimpf_2}
\lim_{n\to \infty} \left\|K_n * ({\rm sgn}(x_1)\, \varphi) - K * ({\rm sgn}(x_1)\, \varphi) \right\|_{L^\infty((0, \infty)\times\R^d)} = 0,
\end{equation}
using that $\left[ \psi_n *({\rm sgn}(x_1)\varphi)\right](y)$ converges to ${\rm sgn}(y_1)\varphi(y)$ whenever $y_1 \neq 0$, and is bounded otherwise.

Thus, in order to bound the second term, $II_n$, we use that $\rho_t^{n,o}$ are uniformly in $L^\infty((0, \infty); L^1(\R^d))$ (with respect to $n$) and that, by \eqref{eq.claimpf_2}, $K_n * ({\rm sgn}(x_1)\, \varphi) $ converges uniformly (in $x$) to $ K * ({\rm sgn}(x_1)\, \varphi) $, so that 
\[
II_n \to 0 \quad\textrm{ as }n \to \infty.
\]

Finally, for the third term we simply use that $\rho_t^{n,o}$ is uniformly in $L^\infty((0, \infty); L^1(\R^d))$ with respect to $n$ together with \eqref{eq.claimpf_new} to get 
\[
III_n \to 0 \quad\textrm{ as }n \to \infty.
\]

In all, we have proved \eqref{eq.weaklimvf}, which at the same time implies \eqref{eq.blim}.
\\[0.2cm]
{\bf Step 6: Proof of the second stability condition.} We now want to prove \eqref{eq.l1conv}, which we rewrite as 
\begin{equation}
\label{eq.l1conv2}
\left( \bar s_n \, \rho_t^{n,o} * K_n \right) (x+h)\to \left(\bar s_n \,\rho_t^{n,o} * K_n\right) (x) \quad\textrm{ as } |h| \to 0,\textrm{ in }L^1_{\rm loc} ((0, \infty); L^1_{\rm loc}(\R^d)),
\end{equation}
uniformly in $n$. 

Let us denote, for $f: \R^d\to \R$, $\delta_h f(x) := f(x+h)-f(x)$. We will prove that, for $B_R\subset\R^d$, 
\begin{equation}
\label{eq.l1conv3}
\int_{B_R} |\delta_h \left( \bar s_n \, \rho_t^{n,o} * K_n \right) (x)|dx \to 0,\quad\textrm{as } |h| \to 0,
\end{equation}
uniformly with respect to $n\in \N$ and $t\in (0, \infty)$. By triangular inequality and the definition of $\bar s_n$, 
\begin{equation}
\label{eq.l1conv4}
|\delta_h \left( \bar s_n \, \rho_t^{n,o} * K_n \right) (x)| \leq |\delta_h \bar s_n(x)|\cdot |\rho_t^{n,o} * K_n (x)|+|\delta_h (\rho_t^{n,o} * K_n) (x)|.
\end{equation}

Let us first prove the convergence of the second term. Given $\alpha \in (0, 1)$ and $p \in \big[1,\frac{d}{d-1+\alpha}\big) $, then $K = c_d x/|x|^d\in W^{\alpha, p}_{\rm loc}(\R^d)$, so that by Young's inequality, 
\[
\|\rho_t^{n,o} * K_n\|_{W^{\alpha, p}(B_R)} = \|\rho_t^{n,o} * \psi_n * K\|_{W^{\alpha, p}(B_R)} \leq C \|\rho_t^{n,o}\|_{L^1(\R^d)} \|\psi_n \|_{L^1(\R^d)} \leq C,
\]
for some constant $C$ that depends on $R$, $d$, $p$, $\alpha$; but is independent of $n$ and $t$. 

Now, from a classical embedding of $W^{\alpha, 1}$ into the Nikolskii space $N^{\alpha, 1}$ (see \cite[Chapter 7]{Ada75}) we obtain 
\begin{equation}
\label{eq.l1conv5}
\int_{B_R} |\delta_h (\rho_t^{n,o} * K_n)(x)|\,dx  \leq C|h|^{\alpha}\to 0,\quad\textrm{as }|h |\to 0,
\end{equation}
for some constant $C$ that depends only on $d$ and $R$, and $|h|\leq R$.

Let us now bound 
\[
\int_{B_R}|\delta_h \bar s_n(x)|\cdot |\rho_t^{n,o} * K_n (x)|\,dx.
\]

Notice that $|\delta_h \bar s_n| \leq \min\{2, 2|h_1|/r_n\}$, and that $\delta_h \bar s_n (x) = 0$ whenever $|x_1|\geq r_n + |h_1|$. Moreover, 
\[
\|\rho_t^{n,o} * \psi_n * K \|_{L^1(\{|x_1|\leq r_n +|h_1|\}\cap B_R)} \leq \|g_0^n\|_{L^1(\R^{2d})} \|K \|_{L^1(\{|x_1|\leq r_n +|h_1|\}\cap B_R)},
\]
and $\|g_0^n\|_{L^1(\R^{2d})} \leq C$ for some constant $C$ that depends only on $g_0$. Let 
\[
\ell_K (\zeta) := \|K \|_{L^1(\{|x_1|\leq \zeta\}\cap B_R)} \to 0,\quad\textrm{as }\zeta \to 0,
\]
thanks to the local integrability of $K$. Putting all together we have that
\begin{equation}
\label{eq.supl1}
\int_{B_R}|\delta_h \bar s_n(x)|\cdot |\rho_t^{n,o} * K_n (x)|\,dx \leq  C\min\{2, 2|h_1|/r_n\}\,\ell_K(r_n + |h_1|).
\end{equation}
We denote
\[
m(h_1) := \min\{m \in \N : r_m \leq |h_1|\} \to \infty,\quad \textrm{as } |h_1|\to 0.
\]
Taking the supremum with respect to $n$ in \eqref{eq.supl1} (separating the cases $|h_1| \geq r_n$ and $|h_1|\leq r_n$), we get 
\begin{equation}
\label{eq.l1conv6}
\int_{B_R}|\delta_h \bar s_n(x)|\cdot |\rho_t^{n,o} * K_n (x)|\,dx \leq  C\left(\ell_K(2|h_1|) + 2\ell_K(r_{m(h_1)})\right) \to 0,\quad\textrm{as } |h|\to 0,
\end{equation}
independently of $n$ and $t$.

Thus, combining \eqref{eq.l1conv4}-\eqref{eq.l1conv5}-\eqref{eq.l1conv6} we get \eqref{eq.l1conv3}, which yields \eqref{eq.l1conv2}, as we wanted to see.
\\[0.2cm]
{\bf Step 7: Conclusion of existence.}
Since conditions \eqref{eq.blim}-\eqref{eq.l1conv} are fulfilled,  we can apply the stability result by DiPerna-Lions, \cite[Theorem II.7]{DL89b} (see also \cite[Proposition 6.5]{ACF15}), to get that the vector fields $b_n$ are converging strongly in $L^1$. Therefore, weakly continuous bounded solutions of the approximating problems converging weakly$^*$ in $L^\infty$ are distributional solutions in the limit. In particular, for every $m \in \N$, $G^m_t = \sum_{k = 0}^m \bar g^k_t$ (recall $\bar g^k_t$ defined in \eqref{eq.gk}) is a distributional solution of the continuity equation with initial datum $G_0^m = \sum_{k = 0}^m \bar g^k_0$; as it is bounded by $m+1$. 

By Theorem~\ref{thm.22} (i), since $G^m$ is bounded, it is a renormalized solution and it is transported by the corresponding Maximal Regular Flow. Since $G^m$ converges to $g_t$ in $L_{\rm loc}^1((0, \infty)\times\R^{2d})$, the limiting $g_t$ is also a renormalized solution; and by Theorem~\ref{thm.22} (ii), it is transported by the Maximal Regular Flow. Moreover, by \cite[Theorem 4.10]{ACF17}, $g_t \in C([0, \infty); L^1_{\rm loc}(\R^{2d}))$.
\\[0.2cm]
{\bf Step 8: Strong $L^1_{\rm loc}$ continuity of density and electric field.} We finally prove that $\rho_t, E_t\in C([0, \infty); L^1_{\rm loc}(\R^d))$. 

Let us start with the physical densities, $\rho_t$. Fix a time $t_\infty\in [0, \infty)$, and let $(t_n)_{n\in \N}$ be a sequence such that $t_n \to t_\infty$. Let $R > 0$ be fixed; then 
\begin{align*}
\int_{B_R} |\rho_{t_n} - \rho_{t_\infty}|\,dx & \leq \int_{B_R}\int_{\R^d}|g_{t_n}-g_{t_\infty}|\, dv\,dx\\
& \leq \int_{B_R}\int_{B_{\overline{R}}}|g_{t_n} - g_{t_\infty}|\,dv\,dx + \int_{B_R}\int_{\R^d\setminus B_{\overline{R}}} \frac{|v|^2}{\overline{R}^2}\left(f_{t_n}+f_ {t_\infty}\right)\,dv\,dx\\
& \leq \int_{B_R}\int_{B_{\overline{R}}}|g_{t_n} - g_{t_\infty}|\,dv\,dx + \frac{C}{\overline{R}^2},\quad \quad\quad\quad\quad\quad\textrm{ for all }\overline{R} >0,
\end{align*}
for some constant $C$ that depends only on the initial bound of the kinetic energy. We have used here the bound on the kinetic energy in the limit, that follows, as in the Step 4, from the lower continuity of the kinetic energy and \eqref{eq.boundke}. Taking limits on both sides, using that $g_t\in C([0, \infty);L^1_{\rm loc}(\R^{2d}))$, and that the previous inequality holds for all $\overline{R}$, we obtain
\[
\lim_{n \to \infty} \int_{B_R} |\rho_{t_n}-\rho_{t_\infty}|\, dx = 0;
\]
that is, $\rho_t$ is strongly continuous in $L^1_{\rm loc}$. In particular, $\rho_t^o = {\rm sgn}(x_1)\rho_t$ is also strongly continuous in $L^1_{\rm loc}$. 

On the other hand, we recall that $E_t = {\rm sgn}\,(x_1) \rho_t^o * K$, with $K = x/|x|^d$; and we want to check the strong continuity of $E_t$ in $L^1_{\rm loc}$. As before, we consider $t_\infty\in [0, \infty)$ and a sequence $(t_n)_{n \in \N}$ with $t_n \to t_\infty$. Let us fix some $R > 0$. Then, we have 
\begin{align*}
\int_{B_R} |E_{t_n} - E_{t_\infty}|\, dx & = \int_{B_R} |K * (\rho_{t_n}^o - \rho_{t_\infty}^o)|\, dx\\
& \leq \int_{B_R}\int_{\R^d} |K(x-y)|\cdot|\rho_{t_n}^o(y) - \rho_{t_\infty}^o(y)| \, dy \, dx\\
 & = \int_{\R^d}|\rho_{t_n}^o(y) - \rho_{t_\infty}^o(y)|Z_R(|y|)\, dy,
\end{align*}
where we have defined, for $y\in \R^d$,
\[
Z_R(|y|) := \int_{B_R} |K(x-y)|\, dx = \int_{B_R(y)} |z|^{-d+1}\, dz.
\]
Note that the previous definition depends only on $|y|$ and not $y$. It trivially holds that $Z$ is bounded and decreasing, going to 0 in the limit $|y|\to\infty$. 

Now, for any $\overline{R} > 0$, and due to the bounds on the densities, $\|\rho_t^o\|_{L^1(\R^d)}\leq C$, we have 
\begin{align*}
\int_{B_R} |E_{t_n} - E_{t_\infty}|\, dx &\leq C\left(  \int_{B_{\overline{R}}}|\rho_{t_n}^o(y) - \rho_{t_\infty}^o(y)| \, dy + Z_R(\overline{R})\right),\quad\textrm{for all }\overline{R}>  0,
\end{align*}
for some constant $C$ that depends only on the initial datum mass, $\|\rho_0\|_{L^1(\R^d)}$. Now, taking first limits as $n$ goes to infinity, using the strong continuity of $\rho_t^o$ in $L^1_{\rm loc}(\R^d)$ and letting $\overline{R}$ go to infinity, we get the desired result,
\[
\lim_{n \to \infty} \int_{B_R} |E_{t_n} - E_{t_\infty}|\, dx = 0;
\]
that is, $E_t$ is strongly continuous in $L^1_{\rm loc}(\R^d)$.
\end{proof}

\section{Proof of main results in the half-space}
\label{sec.mainres}
In this section we prove the main results in the half-space, Theorems~\ref{thm.main1} and \ref{thm.main2}. Throughout the proofs we will be using the equivalence between Problem A \eqref{eq.A} and Problem B \eqref{eq.B} studied in Section~\ref{sec.vphs}. Notice that Problem A \eqref{eq.A} corresponds to \eqref{eq.original}.

\begin{proof}[Proof of Theorem~\ref{thm.main1}]
Extend $f_0$ evenly with respect to $(x_1, v_1)$ to the whole space; that is, consider $g_0(x, v) = f_0(x, v)$ in $\{x_1 \geq 0\}$ and $g_0(x, v) = f_0(x', v')$ in $\{x_1 < 0\}$.

Now notice that Theorem~\ref{thm.main1} corresponds to Theorem~\ref{thm.22} combined with Lemma~\ref{lem.mrfmsf} and Proposition~\ref{prop.AimpB}.  

Indeed, once we have a renormalized solution in the whole space, $g_t$, transported by $(X_t, V_t)$, then, for any  $\phi$ supported in $\overline{\R^d_+}\times\R^d$,
\begin{align*}
\int_{\R^d_+\times\R^d} \phi(x, v) f_t(x, v) \,dx\,dv & = \frac12\int_{\R^{2d}}\phi^e(x, v) g_t(x, v) \,dx\,dv \\
& = \frac12\int_{\R^{2d}}\phi^e(X_t(x, v), V_t(x, v)) g_0(x, v) \,dx\,dv\\
& = \int_{\R^d_+\times\R^d} \phi(\tilde X_t(x, v), \tilde V_t(x, v)) f_0(x, v) \,dx\,dv.
\end{align*}
Here, $\phi^e$ denotes the even extension with respect to $(x, v)$ of $\phi$ ($\phi^e(x', v') = \phi^e(x, v)$), and $(\tilde X_t, \tilde V_t)$ denotes the Maximal Specular Flow as in Lemma~\ref{lem.mrfmsf}. Notice that the hypotheses of Lemma~\ref{lem.mrfmsf} are fulfilled, i.e., the flow has the desired symmetry in its domain of definition, since using the symmetries on $g_t$  for all $t\ge 0$,
\[
\int_{\R^{2d}}\phi^e(X'_t(x, v), V'_t(x, v)) g_0(x, v) \,dx\,dv = \int_{\R^{2d}}\phi^e(X_t(x', v'), V_t(x', v')) g_0(x, v) \,dx\,dv,
\]
for any even test function $\phi^e$. 

The only thing that remains to be checked is that $f_t$ fulfils the commutativity property (Definition~\ref{defi.comm2}). 

That is, we want to check that the weak trace of $g_t$ at $\{x_1 = 0\}$ found in Lemma~\ref{lem.trace} can actually be taken in the strong sense. In order to check the commutativity property it is enough to show (using the same notation as in Lemma~\ref{lem.trace}), 
\begin{equation}
\label{eq.commp}
\lim_{x_1 \to 0} \int \rho \left(\beta(g(t, x_1,\bar x, v)) - \beta(\Gamma(g))\right)dt\,d\bar x\, dv= 0,
\end{equation}
for all $\rho \in C^\infty_c((0, T)\times\R^{d-1}\times\R^d)$ compactly supported in $\{v_1 \neq 0\}$; and for all $\beta\in C^1 \cap L^\infty$. We are assuming that $g_t\in L^\infty$, otherwise take $\arctan(g_t)$ instead. If $x_1$ is small enough, and since in the support of $\rho$, $|v_1|, t \geq \delta(\rho) > 0$ is strictly positive depending only on $\rho$, there exists a flow in $x_1$, 
\[
F_{x_1} = (T_{x_1}, \bar X_{x_1}, V_{x_1}) : D_F\subset (-\eps, \eps)\times (0, T) \times \R^{d-1} \times \R^d \to (0, T)\times\R^{d-1}\times\R^d,
\]
such that $\Gamma(g) = g(x_1, F_{x_1}) := g(T_{x_1}, x_1, \bar X_{x_1}, V_{x_1})$. That is, we can look at the flow as a function of $x_1$ instead of $t$, if $x_1$ is small enough, and then $g$ is flowed in $x_1$ through the path $(T_{x_1}, x_1, \bar X_{x_1}, V_{x_1})$. This flow in $x_1$ can be taken thanks to the inverse function theorem applied to the standard flow in time for $g$, since in the domain of $\rho$, $|v_1| > 0$ allows us to invert the flow with respect to $x_1$ (the derivative of the $X_1$ component in the original flow, $V_1$, does not vanish in the domain). 

Now, simply use that for $x_1$ small enough, $\Gamma(\beta\circ g) = \beta\circ g(x_1, F_{x_1}) = \beta(\Gamma(g))$, and \eqref{eq.commp} follows from the fact that we already know that $\Gamma(g)$ is the trace of $g$ in the weak sense.
\end{proof}

\begin{proof}[Proof of Theorem~\ref{thm.main2}]
This directly follows again by switching between Problem A \eqref{eq.A} and Problem B \eqref{eq.B} and using Theorem~\ref{thm.mainws}. That is, combine Lemma~\ref{lem.mrfmsf} and Proposition~\ref{prop.BimpA} with Theorem~\ref{thm.mainws} to get the result. The commutativity property follows as in the proof of Theorem~\ref{thm.main1}. 

One also needs to notice that
\[
\int_{\R^d}H*\rho_t^o \,\rho_t^o\,dx = 2\int_{\R^d_+\times\R^d_+} (H(x-y)-H(x'-y))\rho_t(x)\rho_t(y)\,dx\,dy,
\]
and the term $H(x-y)-H(x'-y)$ corresponds to the Green function in the half-space. 
\end{proof}

\section{General domains}
\label{sec.gendom}

We want to use the ideas developed for the problem in the half-space to this problem. To do so, we notice that the regularity theory for the existence and uniqueness of a Maximal Regular Flow developed in \cite{ACF15} is completely local. Thus, we can consider small open sets around the boundary, change variables, and encounter a situation close to the half-space solution. If a local Maximal Regular Flow exists in each of such open sets, we have a flow in the whole domain.

Let $T > 0 $, and let us suppose that we have $f_t$ a renormalized solution to
\begin{equation}
\label{eq.original_Om_pp}
  \left\{ \begin{array}{ll}
  \de_t f + v\cdot \nabla_x f_t + E_t\cdot\nabla_v f_t=0 & \textrm{ in } (0,T)\times\Omega\times\R^d\\
  \rho_t(x)=\int_{\R^d} f_t(x, v) dv & \textrm{ in } (0,T)\times\Omega\\
  E_t(x) = -\int_{\Omega} \nabla_x G_\Omega(x, z)\rho_t(z)\,dz & \textrm{ in } (0,T)\times\Omega \\
  f_t(x, v) = f_t(x, R_x v) & \textrm{ on } (0, T)\times\de\Omega\times\R^d.\\
  \end{array}\right.
\end{equation}

Suppose that the domain $\Omega$ fulfils the exterior and interior ball condition uniformly at each boundary point with balls of radius $2$ (otherwise, we can rescale). Let us also assume, without loss of generality, that 
\begin{equation}
\label{eq.domcond}
0\in \de\Omega~~~\textrm{   and   }~~~n(0) = e_1.
\end{equation}

Let us now consider a change of variables (to be determined) to go to a half-space situation. Let 
\[
y = \phi (x) = \phi :  B_2 \to \R^d
\]
be such that $\phi(\Omega\cap B_2) \cap B_1 = \{x_1 > 0\}\cap B_1$. We also define the inverse,
\[
x= \psi(y) = \psi := \phi^{-1} : B_1 \to B_2.
\]

Let $J$ be the Jacobian of the change of variables, 
\[
J(x) := \begin{bmatrix}
    \frac{\de\phi_1}{\de x_1} & \dots    & \frac{\de\phi_1}{\de x_d}  \\
    \vdots& \ddots & \vdots  \\
    \frac{\de\phi_d}{\de x_1}  & \dots & \frac{\de\phi_d}{\de x_d} 
\end{bmatrix},
\]
and we can assume that there is a constant $C_d\ge 1$ depending only on $d$ such that
\begin{equation}
\label{eq.cd}
\frac{1}{C_d}\le \textrm{det}\,J(x) \leq C_d,\quad\textrm{in }\quad B_1. 
\end{equation}
We change the velocities accordingly, so that the change of variables becomes,
\begin{equation}
\label{eq.chvar}
(x, v) \mapsto (y, w) = (\phi(x), J(x) v).
\end{equation}

Let 
\begin{equation}
\label{eq.chvar2}
g_t(y, w) = f_t(x, v).
\end{equation}

One can formally check that, under this situation, $g_t$ fulfills the following equation
\begin{equation}
\label{eq.eqtransf}
\de_t g_t(y, w) +  w\cdot \nabla_y g_t(y, w) + \left\{\left[(J^{-1}(x))^T\frac{\de w}{\de x}\right]^T w + J(x) E(x)\right\}\cdot\nabla_w g_t(y, w) = 0,
\end{equation}
for $(y, w)\in \left(B_1\cap \{y_1 > 0\}\right) \times \R^d$, $x = \psi(y)$, and 
\[
\frac{\de w}{\de x} := \begin{bmatrix}
    \frac{\de w_1}{\de x_1} & \dots    & \frac{\de w_d}{\de x_1}  \\
    \vdots& \ddots & \vdots  \\
    \frac{\de w_1}{\de x_d}  & \dots & \frac{\de w_d}{\de x_d} 
\end{bmatrix}.
\]

Notice that we can also rewrite the term containing $\frac{\de w}{\de x}$ in terms of the Hessian of the change of variables as
\begin{equation}
\left[(J^{-1}(x))^T\frac{\de w}{\de x}\right]^T w = v^T D^2\phi\,v,
\end{equation}
in the sense 
\[
(v^T D^2\phi\, v)_i = v^T D^2\phi_i\, v,\quad\quad\textrm{ for all } i \in \{1,\dots, n\}.
\]

If the specular reflection at the boundary was conserved, we could repeat the arguments from Proposition~\ref{prop.AimpB} and construct a renormalized solution in $B_1$ with respect to a reflected vector field. 

We claim that the following change of variables preserves the reflection, which will be seen in the proofs below. In particular, we will be using equations \eqref{eq.jacob1} and \eqref{eq.jacob2}.

Consider a diffeomorphism $\Phi:\R^d\to \R^d$ such that $\Phi(\de \Omega) \subset \{x_1 = 0\}$. In order to construct $\phi$ sending $\Omega$ to a half-space, we simply consider the orthogonal projection onto $\de \Omega$, use $\Phi$, and translate in the $e_1$ direction the same distance we projected. That is, if $x\in \Omega$ and $\pi_\Omega : \Omega \to \de\Omega$ is the orthogonal projection onto $\de \Omega$, then 
\begin{equation}
\label{eq.change}
\phi(x) = \Phi(\pi_\Omega (x) ) + \textrm{dist}(x, \pi_\Omega(x)) e_1.
\end{equation}

Notice that such operation is only a diffeomorphism in a uniform neighbourhood of $\de \Omega$, which we will denote $\Omega_\pi$, and we can extend by any other diffeomorphism in the rest of the domain $\Omega\cap B_2$.

Notice that, in particular, we have for any $x\in \de\Omega$,
\begin{equation}
\label{eq.jacob1}
J n(x) = e_1,
\end{equation}
and
\begin{equation}
\label{eq.jacob2}
v\cdot n(x) = J v \cdot e_1.
\end{equation}

Equation \eqref{eq.jacob2} above follows as in \cite[Section 2]{Hwa04}. Equation \eqref{eq.jacob1} follows by noticing that points of the form $x+tn(x)$ go to points of the form $\phi(x) +t e_1$ for all $x\in \de\Omega$. 
Similarly, the previous equalities also hold for points in $\Omega_\pi$. That is, if we denote by $N(x)$ the unit vector pointing inwards $\Omega$ in the direction of the projection ($N(x) = \nabla{\rm dist}(x, \de\Omega)$), then \eqref{eq.jacob1}-\eqref{eq.jacob2} also hold for $N(x)$ in $\Omega_\pi$. We remark that the size of the neighbourhood $\Omega_\pi$ depends only on the domain $\Omega$ and its local $C^{1, 1}$ norm. 

In the next proposition we show that this change of variables allows an even extension of the solution to solve a new distributional problem, constructed analogously to the half-space situation. 

\begin{prop}
\label{prop.renB1}
Let $T > 0$, let $\Omega$ be a $C^{1, 1}$ domain as described above, \eqref{eq.domcond}, and let $f_t\in L^1_{\rm loc}([0, T]\times \Omega\times\R^d)$ be a renormalized solution to \eqref{eq.VPOm} (see Definition~\ref{defi.HS_Om}). Under the above change of variables, \eqref{eq.chvar}-\eqref{eq.chvar2}-\eqref{eq.change}, define the even extension with respect to $(y_1, w_1)$ of $g_t$ in $B_1$, 
\begin{equation}
\label{eq.gte}
g_t^e(y, w) =  \left\{ \begin{array}{ll}
  g_t(y, w) & \textrm{ if } y_1 \ge 0 \\
  g_t(y', w') & \textrm{ if } y_1< 0.\\
  \end{array}\right.
\end{equation}
Let $j(y) := {\rm det}\,J(\psi(y))$, and define also $j_e(y)$ as the even extension with respect to $y_1$ of $j(y)$; $j_e(y) = j(y)$ if $y_1 \ge 0$ and $j_e(y) = j(y')$ otherwise. 

For any $\beta\in C^1\cap L^\infty(\R)$, let
\begin{equation}
\label{eq.gjbeta}
g^{j,\beta}_t(y, w) := \beta(g_t^e(y, w)) j^{-2}_e(y).
\end{equation}
Then, $g^{j,\beta}_t$ is a distributional solution to the continuity equation 
\begin{equation}
\label{eq.gerenorm}
\de_t g^{j,\beta}_t(y, w) + {\rm div}_{y, w} \left(\tilde b_t (y, w) g^{j,\beta}_t(y, w)\right) = 0, \textrm{ in } (0, T)\times B_1\times\R^d,
\end{equation}
where $\tilde b_t (y, w) = (w, b^{2, o}_t (y, w)) : (0, T)\times B_1\times\R^d\to \R^d\times\R^d$,
\begin{equation}
\label{eq.gerenorm2}
b_t^2(y, w)  = \left[(J^{-1}(x))^T\frac{\de w}{\de x}\right]^T w + J(x) E_t(x),
\end{equation}
for $(y, w)\in \left(B_1\cap \{y_1 > 0\}\right) \times \R^d$, $x = \psi(y)$, and where $b_t^{2, o}$ is defined as 
\begin{equation}
\label{eq.gerenorm3}
b_t^{2, o}(y, w) =  \left\{ \begin{array}{ll}
  b_t^2(y, w) & \textrm{ if } y_1 \ge 0 \\
  \left(b_t^2(y', w')\right)' & \textrm{ if } y_1< 0,\\
  \end{array}\right.
\end{equation}
that is, the odd extension with respect to $(y_1, w_1)$ to $y_1 < 0$.
\end{prop}
\begin{proof}
The proof follows along the lines of Proposition~\ref{prop.AimpB}. Let us assume for simplicity that $f_t\in L^\infty((0, T); L^\infty(\Omega\times\R^d))$, so that we can forget about the function $\beta$, and denote $g_t^j = g_t^ej_e^{-2}$. We know that there exists $f_{t}^+\in L^\infty_{\rm loc}(\gamma^+_{\Omega,T})$ such that for every $\varphi \in \mathcal{T}_{\Omega,T}$,
\begin{align}
\nonumber \int_{ \Omega\times \R^{d}}& \varphi_0(x, v) f_0(x, v) \,dx\, dv\, +  \\
\label{eq.defi_HS_Om_2}& + \int_0^T \int_{\Omega\times\R^{d}} [\de_t \varphi_t(x, v) + \nabla_{x, v} \varphi_t(x, v)\cdot  b_t(x, v) ] f_t(x, v)\,dt \,dx\,dv 
\\ \nonumber & + \int_{\gamma^+_{\Omega,T}} v\cdot n(x) \big(\varphi_t( x, v) - \varphi_t(x, R_x v)\big) f^+_{t}(x, v)\, dt\,d\sigma^\Omega_x\,dv = 0.
\end{align}

We now perform the change of variables, $(x, v) \mapsto (y, w) = (\phi(x), J(x)v)$, with Jacobian determinant $j^{-2}(y)$, that is $dx\,dv \mapsto j^{-2}(y)dy\,dw$. Let $\eta_t(y, w) = \varphi_t(x, v)$, and let us suppose that $\eta_t(y, w)$ is supported in $[0, T)\times \left(B_1\cap \{y_1 > 0\}\right) \times \R^d$. Thanks to \eqref{eq.jacob2}, $\eta_t \in \mathcal{T}_T$, the space defined in \eqref{eq.testf_Om}. Similarly, due to \eqref{eq.jacob2}, we also have that $\gamma^+_{\Omega, T}$ becomes $\gamma^+_T$ (see Definition~\ref{defi.spaces_Om}). 

Proceeding as in \eqref{eq.eqtransf} by changing variables in \eqref{eq.defi_HS_Om_2} we get 
\begin{align}
\nonumber &\int_{ \{y_1 \geq 0\}\times \R^{d}} \eta_0(y, w) g_0(y, w) j^{-2}(y) \,dy\, dw\, +  \\
\label{eq.defi_HS_Om_3}& ~~+ \int_0^T \int_{\{y_1 \geq 0\}\times\R^{d}} [\de_t \eta_t(y, w) + \nabla_{y, w} \eta_t(y, w)\cdot  \tilde b_t(y, w) ] g_t(y, w)j^{-2}(y) \,dt \,dy\,dw 
\\ \nonumber & ~~+ \int_{\gamma^+_{T}} v\cdot n(x) \big(\eta_t(y, w) - \varphi_t(\psi(y), R_x J^{-1} w)\big) f^+_{t}(\psi(y), J^{-1} w)j^{-2}(y) \, dt\,d\sigma_y\,dw = 0.
\end{align}
We have also used here that the \emph{boundary} measure $d\sigma_x\, dv \mapsto j^{-2}(y) d\sigma_y\,dw$, since in the change of variables there is no stretching in the direction normal to $\de\Omega$.

From \eqref{eq.jacob1}-\eqref{eq.jacob2} we also get $\phi_t(\psi(y), R_xJ^{-1} w) = \eta_t(y, w')$ and $v\cdot n(x) = w_1$. Calling $g^+_{t}(y, w):= f^+_{t}(\psi(y), J^{-1} w)$, we notice that \eqref{eq.defi_HS_Om_3} corresponds to the definition distributional solution to the continuity equation \eqref{eq.gerenorm} in the half-space with specular reflection, for test functions spatially supported in $B_1\cap \{y_1 \ge 0\}$, (see Remark~\ref{rem.distr.gen}). That is, we have
\begin{align}
\nonumber &\int_{ \{y_1 \geq 0\}\times \R^{d}} \eta_0(y, w) g_0^j(y, w)  \,dy\, dw\, +  \\
\label{eq.defi_HS_Om_4}& ~~+ \int_0^T \int_{\{y_1 \geq 0\}\times\R^{d}} [\de_t \eta_t(y, w) + \nabla_{y, w} \eta_t(y, w)\cdot  \tilde b_t(y, w) ] g_t^j(y, w) \,dt \,dy\,dw 
\\ \nonumber & ~~+ \int_{\gamma^+_{T}} w_1 \big(\eta_t(y, w) - \eta(y, w')\big) g^{j,+}_{t}(y, w) \, dt\,d\sigma_y\,dw = 0,
\end{align}
where $g^{j,+}_{t}(y, w) = g^{+}_{t}(y, w) j^{-2}(y)$.

To finish the proof, we proceed as in the proof of Proposition~\ref{prop.AimpB}, by noticing that there we only used the symmetry of the vector field $\tilde b_t(y, w)$, i.e., $b_t^{2, o}(y', w') = \left(b_t^{2, o}(y, w)\right)'$. Notice that we are also using here that by Remark~\ref{rem.moregen} from Lemma~\ref{lem.testfunctions} it is enough to check it for test functions in $\mathcal{T}_T$.
\end{proof}

In the previous proposition, Proposition~\eqref{prop.renB1}, we were interested in producing a solution to a continuity equation, namely \eqref{eq.gerenorm}. We would also like yet another result regarding the problem solved by the even extension, $g_t^e$. In this case, thus, we obtain an analogous result but now for a transport equation. Notice that, after changing variables the vector field is no longer divergence-free, and therefore continuity and transport equations are no longer equivalent. This results will be useful in the next pages.

\begin{prop}
\label{prop.renB1_T}
Let $T > 0$, let $\Omega$ be a $C^{1, 1}$ domain as described above, \eqref{eq.domcond}, and let $f_t\in L^1_{\rm loc}([0, T]\times \Omega\times\R^d)$ be a renormalized solution to \eqref{eq.VPOm} (see Definition~\ref{defi.HS_Om}). Under the above change of variables, \eqref{eq.chvar}-\eqref{eq.chvar2}-\eqref{eq.change}, define the even extension with respect to $(y_1, w_1)$ of $g_t$ in $B_1$, 
\begin{equation}
\label{eq.gte_T}
g_t^e(y, w) =  \left\{ \begin{array}{ll}
  g_t(y, w) & \textrm{ if } y_1 \ge 0 \\
  g_t(y', w') & \textrm{ if } y_1< 0.\\
  \end{array}\right.
\end{equation}

Then, $g^e_t$ is a renormalized solution to the transport equation
\begin{equation}
\label{eq.gerenorm_T}
\de_t g^{e}_t(y, w) + \tilde b_t (y, w) \cdot \nabla_{y, w} g^{e}_t(y, w) = 0, \textrm{ in } (0, T)\times B_1\times\R^d,
\end{equation}
where $\tilde b_t (y, w)$ is defined by \eqref{eq.gerenorm2}-\eqref{eq.gerenorm3}.
\end{prop}
\begin{proof}
The proof again follows along the lines of Proposition~\ref{prop.AimpB}. As before, we assume for simplicity that $f_t\in L^\infty((0, T); L^\infty(\Omega\times\R^d))$. We know that there exists $f_{t}^+\in L^\infty_{\rm loc}(\gamma^+_{\Omega,T})$ such that for every $\varphi \in \mathcal{T}_{\Omega,T}$,
\[
\begin{split}
 \int_{ \Omega\times \R^{d}}& \varphi_0(x, v) f_0(x, v) \,dx\, dv\, + \int_0^T \int_{\Omega\times\R^{d}} [\de_t \varphi_t(x, v) + \nabla_{x, v} \varphi_t(x, v)\cdot  b_t(x, v) ] f_t(x, v)\,dt \,dx\,dv\,+
\\  &~~~~~~~~~~~~~~~~~~~~~~~~~~~~~~~~~~~~~ + \int_{\gamma^+_{\Omega,T}} v\cdot n(x) \big(\varphi_t( x, v) - \varphi_t(x, R_x v)\big) f^+_{t}(x, v)\, dt\,d\sigma^\Omega_x\,dv = 0.
\end{split}
\]

Let $\eta_t(y, w) = \varphi_t(x, v)j^{-2}(x)$ as before, where $j(x) = {\rm det}(J(x))$ the Jacobian determinant, and let us suppose that $\eta_t(y, w)$ is supported in $[0, T)\times \left(B_1\cap \{y_1 > 0\}\right) \times \R^d$. Performing the change of variables, in \eqref{eq.defi_HS_Om_2} we get 
\begin{align}
\nonumber &\int_{ \{y_1 \geq 0\}\times \R^{d}} \eta_0(y, w) g_0(y, w) \,dy\, dw\, +  \\
\label{eq.defi_HS_Om_3_T}& ~~+ \int_0^T \int_{\{y_1 \geq 0\}\times\R^{d}} [\de_t \eta_t(y, w) + \nabla_{y, w} \eta_t(y, w)\cdot  \tilde b_t(y, w) ] g_t(y, w)\,dt \,dy\,dw 
\\ \nonumber & ~~+ \int_{\gamma^+_{T}} v\cdot n(x) \big(\eta_t(y, w) - j^{-2}(x)\varphi_t(\psi(y), R_x J^{-1} w)\big) f^+_{t}(\psi(y), J^{-1} w)\, dt\,d\sigma_y\,dw = 
\\ \nonumber & ~~~~~~~~~~~~~~~~~~~~~~~~~~~  = -\int_0^T \int_{\{y_1 \geq 0\}\times\R^{d}} \eta_t(y, w) g_t(y, w) {\rm div}_w b_t^2(y, w)\,dt\,dy\,dw.
\end{align}
From \eqref{eq.jacob1}-\eqref{eq.jacob2} we also get $j^{-2}(x)\phi_t(\psi(y), R_xJ^{-1} w) = \eta_t(y, w')$ and $v\cdot n(x) = w_1$. Calling $g^+_{t}(y, w):= f^+_{t}(\psi(y), J^{-1} w)$, we notice that \eqref{eq.defi_HS_Om_3} corresponds to the definition of normalized solution in the half-space, for test functions spatially supported in $B_1\cap \{y_1 \ge 0\}$, and for vector fields not necessarily divergence free (see Remark~\ref{rem.nondivfree}).

To finish the proof, we proceed as in the proof of Proposition~\ref{prop.AimpB}, by noticing that there we only used the symmetry of the vector field $\tilde b_t(y, w)$, i.e., $b_t^{2, o}(y', w') = \left(b_t^{2, o}(y, w)\right)'$. Moreover, we also need to use that $({\rm div}_w b_t^2)(y', w') = {\rm div}_w b_t^2 (y, w)$.
\end{proof}
\begin{rem}
\label{rem.incomp}
In the previous proof, we are actually using a distributional proof of the fact that the following equality holds, 
\[
2j^{-1}(x) \nabla_x j(x) \cdot v = {\rm div}_w \left\{\left[(J^{-1}(x))^T\frac{\de w}{\de x}\right]^T w\right\},
\]
which can also be directly checked. In a way, we are going from a continuity equation form $\partial_t u+{\rm div}(bu)=cu$ to a transport equation form $\partial_t u +b\nabla u=(c-{\rm div} b)\,u $. 
\end{rem}

\begin{prop}
\label{prop.renB1_T2}
Let $T > 0$, let $\Omega$ be a $C^{1, 1}$ domain as described above, \eqref{eq.domcond}; and consider the above change of variables, \eqref{eq.chvar}-\eqref{eq.chvar2}-\eqref{eq.change}.  

Suppose $g_t^e \in L^1_{\rm loc}([0, T]\times B_2\times\R^d)$ is an even function ($g_t(y', w') = g_t(y, w)$), and define $f_t \in L^1_{\rm loc}([0, T]\times \Omega\cap B_1 \times\R^d)$ as the restriction of $g_t^e$ to $\{y_1 \ge 0\}$ after changing variables. That is, 
\[
f_t(x, v) = g_t^e(y, w), \quad\textrm{ for } (t, x, v) \in [0, T]\times \Omega\cap B_1\times\R^d.
\]
Assume that $g_t^e$ is a renormalized solution to the transport equation,
\begin{equation}
\label{eq.gerenorm_T2}
\de_t g^{e}_t(y, w) + \tilde b_t (y, w) \cdot \nabla_{y, w} g^{e}_t(y, w) = 0, \textrm{ in } (0, T)\times B_1\times\R^d,
\end{equation}
where $\tilde b_t (y, w)$ is defined by \eqref{eq.gerenorm2}-\eqref{eq.gerenorm3} via $f_t$. Then, $f_t$ is a renormalized solution to 
\[
\left\{ \begin{array}{ll}
  \de_t f_t + b_t\cdot \nabla_{x,v} f_t=0 & \textrm{ in } (0,\infty)\times\Omega\times\R^d\\
  f_t(x, v) = f_t(x, R_x v) & \textrm{ in } (0,\infty)\times\de\Omega\times\R^d,
  \end{array}\right.
  \]
  according to Definition~\ref{defi.HS_Om}.
\end{prop}
\begin{proof}
The proof is the same as Proposition~\ref{prop.BimpA} and Lemma~\ref{lem.trace} for $g_t^e$, just using the symmetries of the vector field $\tilde b_t(y, w)$. A change of variables then yields the desired result. 
\end{proof}

Let us now state a theorem, relating renormalized and Lagrangian solutions. We show here that flowing a function $f_0$ via a Maximal Specular Flow produces a renormalized solution to a continuity equation. This corresponds to \cite[Theorem 4.10]{ACF17}, and in this case we want to focus on what is occurring at the boundary.

Before stating the theorem, let us state the following lemma regarding the incompressibility of the Maximal Specular Flow of divergence-free vector fields, $b = (v, E_t(x))$, in $\Omega\times\R^d$. We say that a flow is incompressible if the push-forward through the flow at any time $t$ of the Lebesgue measure, is the Lebesgue measure. Namely, the local Maximal Regular Flow is incompressible if condition (ii) in Definition~\ref{defi.mrf_A} holds with an equality and $C = 1$. Analogously, the Maximal Specular Flow is incompressible if condition (iii) in Definition~\ref{defi.msf_dom} holds with an equality and $C = 1$. 

Under enough regularity, incompressibility is equivalent to a divergence-free vector field. In this setting, this also holds. The Maximal Regular Flow of a divergence-free vector field is incompressible thanks to \cite[Theorem 4.3]{ACF17}. And using it, we can also show that adding the specular reflection condition at the boundary cannot add compression (or expansion) of the flow; namely, the Lebesgue measure is still preserved. The sketch of the proof of the following lemma is done at the end of the next section, where we will have introduced the reflection technique after changing variables in the setting of general domains.

\begin{lem}
\label{lem.incompress}
Let $T > 0$, and let $b : (0, T)\times\Omega\times\R^d\to \Omega\times\R^d$ be a divergence-free vector field, $b = (v, E_t(x))$, with $E\in L^\infty((0, T);L^p(\Omega))$ for $p > 1$; and let $(X(t, s, x, v), V(t, s, x, v))$ be its Maximal Specular Flow (see Definition~\ref{defi.msf_dom}). Then $(X, V)$ is incompressible; namely, condition (iii) of Definition~\ref{defi.msf_dom} holds with an equality and $C = 1$. 
\end{lem}

\begin{thm}
\label{thm.LagimpRen}
Let $T > 0$, and let $b : (0, T)\times\Omega\times\R^d\to \Omega\times\R^d$ be a divergence-free vector field, $b = (v, E_t(x))$, with $E\in L^\infty((0, T);L^p(\Omega))$ for $p > 1$; and let $(X(t, s, x, v), V(t, s, x, v))$ be its Maximal Specular Flow (see Definition~\ref{defi.msf_dom}). Let $f_0\in L^1(\Omega\times\R^d)$, and define 
\[
f_t := (X(t, 0, \cdot, \cdot), V(t, 0, \cdot, \cdot))_\# (f_0 \mres \{t_{0, X, V}^+ > t\}),\quad t\in[0, T).
\]
Then, $f_t$ is a renormalized solution of the continuity equation with specular reflection \eqref{eq.VPOm} according to Definition~\ref{defi.HS_Om}, fulfilling the commutativity property (Definition~\eqref{defi.comm2}). Moreover, the map $t\mapsto f_t$ is strongly continuous on $[0, T)$ in $L^1_{\rm loc}$. 
\end{thm}
\begin{proof}
Consider $s = 0$, and denote $Z_t(x, v) = (X_t(x, v), V_t(x, v)), := (X(t, 0, x, v), V(t, 0, x, v))$, $t_Z:= t_{0, X, V}^+$. Proceeding as in \cite[Theorem 4.10]{ACF17} just using that the flow is incompressible, Lemma~\ref{lem.incompress}, we get
\begin{equation}
\label{eq.pointwisetransport}
f_t(Z_t(x, v)) = f_0(x, v),\quad\textrm{for } \mathscr{L}^{2d}\textrm{-a.e. } (x, v)\in \{t_Z > t\}.
\end{equation}

In particular, we also have the same for $\beta\circ f_t$ for any $\beta \in C^1\cap L^\infty$, and $\beta\circ f_t$ are distributional solutions of the continuity equation in the interior of $\Omega$, by \cite[Theorem 4.10]{ACF17}. We still have to check the trace condition. 

Take any test function $\varphi = \varphi_t(x, v)\in \mathcal{T}_{\Omega, T}$. Let us now compute 
\[
\begin{split}
\int_0^T\int_{\Omega\times\R^d} \left(\de_t\varphi + b\cdot \nabla \varphi\right) \beta(f_t) \,dx\,dv\,dt & = \int_0^T\int_{\{t_Z > t\}} \left[\left(\de_t\varphi + b\cdot \nabla \varphi\right)\circ Z_t\right]\, \beta(f_0)\,dx\,dv\,dt\\
& = \int_{\{t_Z > t\}} \beta(f_0) \int_0^T \frac{d}{dt}\left(\varphi\circ Z_t\right)\,dt\,dx\,dv
\end{split}
\]
where we have used here the incompressibility of the flow and \eqref{eq.pointwisetransport}. The problem in the temporal domain integral appears only when $Z_t$ approaches the boundary of $\Omega$, where the temporal derivatives has jumps in the velocity component. To avoid that, let us take $\phi_\eps(x)$ a test function defined by 
\[
\phi_\eps(x) = 0, \textrm{ if } x\notin \Omega,\quad \phi_\eps (x) = \min\{\eps^{-1} {\rm dist}(x, \de\Omega), 1\},\textrm{ if } x\in \Omega.
\]
Then, proceeding as before, 
\[
\begin{split}
\iint \left(\de_t\varphi + b\cdot \nabla \varphi\right) \beta(f_t) \,dx\,dv\,dt & = \lim_{\eps \to 0} \iint \left(\de_t\varphi + b\cdot \nabla \varphi\right) \beta(f_t) \phi_\eps(x) \,dx\,dv\,dt\\
& = \lim_{\eps\to 0} \int_{\{t_Z > t\}} \beta(f_0) \int_0^T \frac{d}{dt}\left(\varphi\circ Z_t\right) \phi_\eps(X_t) \,dt\,dx\,dv.
\end{split}
\]
Let us denote, for a.e. $(x, v)\in \Omega\times\R^d$, $\{t_i\}_{i\in I(x, v)}$ for $t_i\in [0, t_Z)$ the set of times such that $X_{t_i}(x, v) \in \de\Omega$. We also denote $Z_t^\pm = \lim_{\eps \downarrow 0} Z_{t\pm\eps}$. Using integration by parts for absolutely continuous functions (notice that $Z_t$ is AC when $X$ is not on the boundary $\de\Omega$) we have 
\[
\begin{split}
& \int_0^T \frac{d}{dt}\left(\varphi\circ Z_t\right) \phi_\eps(X_t) \,dt  = \\
&~~~~~~~~ = \sum_{i\in I(x, v)} \left(  \varphi_{t_i}\circ Z_{t_i}^- -\varphi_{t_i}\circ Z_{t_i}^+\right)\phi_\eps(X_{t_i}) - \varphi_0(x, v) \phi_\eps(x) - \int_0^T\varphi\circ Z_t\,\nabla\phi_\eps(X_t)\cdot V_t\, dt.
\end{split}
\]
Notice that the sum for $i\in I(x, v)$ is actually equal to 0, since the term $\phi_\eps(X_{t_i}) = 0$ whenever $X_{t_i}\in \de\Omega$. 
Putting all together, and denoting $N(x) = \nabla{\rm dist}(x, \de\Omega)$ (note that $N|_{\de\Omega} = n$) we have 
\[
\begin{split}
& \iint \left(\de_t\varphi + b\cdot \nabla \varphi\right) \beta(f_t) \,dx\,dv\,dt +\int_{\Omega\times\R^d} \beta(f_0)\varphi \, dx\,dv =
\\ &~~~~~~~~~~~~~~~~~~~~ = - \lim_{\eps \to 0} \int_{\{t_Z > t\}} \beta(f_0) \int_0^T\varphi\circ Z_t\,\nabla\phi_\eps(X_t)\cdot V_t \,dt\,dx\,dv\\
&~~~~~~~~~~~~~~~~~~~~ = - \lim_{\eps \to 0}\eps^{-1} \int_0^T \int_{D_\eps} \beta(f_t)\varphi\,N(x)\cdot v \,dt\,dx\,dv,
\end{split}
\]
where in the last step we are using again the incompressibility of the flow and \eqref{eq.pointwisetransport}. Here, we have denoted $D_\eps = \left( \{x\in \Omega : {\rm dist}(x, \de\Omega) < \eps\}\times\R^d \right)\cap \left(Z_t(\cdot, \cdot)(\{t_Z > t\})\right) $.

And we claim that 
\begin{equation}
\label{eq.sepdom}
\lim_{\eps \to 0}\frac1\eps \int_0^T \int_{D_\eps} \beta(f_t)\varphi\,N(x)\cdot v \,dt\,dx\,dv = \int_{\gamma_{\Omega, T}^+\cup \gamma_{\Omega, T}^-} n(x) \cdot v\, \varphi \, \Gamma(\beta(f_t))(x, v)\, dt\,d\sigma^\Omega_x\, dv,
\end{equation}
for some trace function $\Gamma(\beta(f_t))\in L^\infty(\gamma_{\Omega, T}^+\cup \gamma_{\Omega, T}^-)$. This follows exactly as the proof of \eqref{eq.bimpa2} in Lemma~\ref{lem.trace}. We can also directly apply the lemma, by first changing variables and having a half-space situation, where the function $\beta \circ g_t$ is a distributional solution to a continuity equation in the interior of $(B_1\cap \R^d_+)\times\R^d$. 

From the transport structure \eqref{eq.pointwisetransport} and condition (ii) of the definition of Maximal Specular Flow (Definition~\ref{defi.msf_dom}), we must have 
\begin{equation}
\label{eq.gammarefl}
\Gamma(\beta\circ f_t)(x, v) = \Gamma(\beta\circ f_t)(x, R_x v),\quad \textrm{ for }\sigma^\Omega\otimes\mathscr{L}^d\textrm{-a.e.}~ (x, v) \in \de\Omega\times\R^d.
\end{equation}
The result immediately follows from here by splitting the integral in \eqref{eq.sepdom} into the domains $\gamma_{\Omega, T}^+$ and $\gamma_{\Omega, T}^-$. Alternatively, notice that Remark~\ref{rem.defi.HS} holds by \eqref{eq.sepdom} and \eqref{eq.gammarefl}.

Moreover, again thanks to the transport structure, the commutativity property holds proceeding as in \eqref{eq.commp} within the proof of Theorem~\ref{thm.main1}.

Finally, we check the strong continuity in $L^1_{\rm loc}$; and in order to simplify things we check it at $t = 0$. For any $\delta > 0$, fix $\psi\in C_c^\infty(\overline{\Omega}\times\R^d) $ such that $\|\psi - f_0\|_{L^1(\Omega\times\R^d)} \le \delta$. Then,
\[
\int_{\Omega\times\R^d} |f_t-\psi |\, dx\, dv \le \int_{Z_t(\cdot, \cdot)(\{t < t_Z\})} |f_t - \psi|\, dx\, dv + \int_{Z_t(\cdot, \cdot)(\{0 < t_Z\le t\})} |\psi|\,dx\, dv.
\] 
From $t_Z > 0$ a.e., the last term vanishes as $t\downarrow 0$. Therefore, using the incompressibility and \eqref{eq.pointwisetransport} as before, we have 
\[
\limsup_{t\downarrow 0}\int_{\Omega\times\R^d}|f_t - \psi|\, dx\,dv  \le \limsup_{t \downarrow 0} \int_{\{t < t_Z\}}|f_0 - \psi\circ Z_t|\, dx\,dv \le \int_{\Omega\times\R^d}|f_0-\psi|,
\]
where in the last inequality we are using that the set of discontinuity of $Z_t$ as $t\downarrow 0$ has measure zero. Hence, 
\[
\limsup_{t\downarrow 0}\int_{\Omega\times\R^d}|f_t - f_0|\, dx\,dv  \le\limsup_{t\downarrow 0}\int_{\Omega\times\R^d}|f_t - \psi|\, dx\,dv +\int_{\Omega\times\R^d}|f_0 - \psi|\, dx\,dv \le 2\delta,   
\]
and by the arbitrariness of $\delta > 0$, we are done. 
\end{proof}

\section{Uniqueness of the continuity equation and Theorem~\ref{thm.main1_Om}}

\label{sec.uniq}
As shown in \cite{ACF15}, there exist conditions analogous to {\bf (A1)} and {\bf (A2)} from Section~\ref{sec.4} that suffice to show existence and uniqueness of local Maximal Regular Flows. In this case, keeping the original notation in \cite[Section 3]{ACF15}, for a given open set $A \subset \R^d$ and a Borel vector field $b: (0, T)\times A\to\R^d$, the following conditions imply existence and uniqueness of local Maximal Regular Flows:
\begin{itemize}
\item[{\bf (a-A)}] $\int_0^T\int_{A'} |b_t(x)|\,dx\,dt < \infty$ for any $A'\Subset A$;
\item[{\bf (b-A)}] for any nonnegative $\bar \mu\in L^\infty_+(A)$ with compact support in $A$ and any closed interval $I = [a, b]\subset[0, T]$, the continuity equation 
\begin{equation}
\label{eq.bA}
\frac{d}{dt}\mu_t + {\rm div}(b\mu_t)  = 0,\quad\textrm{in }\quad (a, b)\times A 
\end{equation}
has at most one weakly* continuous solution $I\ni t \mapsto \rho_t$ such that $\rho_a = \bar\mu$ and $\cup_{t\in [a, b]} {\rm supp}\,\mu_t \Subset A$.
\end{itemize}

We want to check that condition {\bf (b-A)} for the continuity equation \eqref{eq.bA} holds for compactly supported measures in the open set $B_1\times \R^d$ for the vector field constructed in Proposition~\ref{prop.renB1}. That is, we consider a fixed vector field of the form $\tilde b_t (y, w) = (w, b^{2, o}_t (y, w))$,
\begin{equation}
\label{eq.uniqvf}
b_t^2(y, w)  = \left[J^{-1}(x)\frac{\de w}{\de x}\right]^T w + J(x) E_t(x),
\end{equation}
for $(y, w)\in \left(B_1\cap \{y_1 > 0\}\right) \times \R^d$, $x = \psi(y)$, and where $b_t^{2, o}$ is the odd extension with respect to $(y_1, w_1)$ (see \eqref{eq.gerenorm3}), such that the electric field $E_t$ is the one generated by a renormalized solution to \eqref{eq.original_Om}, with $f_t(x, v) \in L^\infty((0, T);L^1(\Omega\times\R^d))$. We recall that when performing the change of variables we are assuming that $0\in \de\Omega$ and that the normal vector at 0 is $e_1$, \eqref{eq.domcond}, while at the same time we assume a domain with an exterior and interior ball condition of radius 2. 

Let us first state a result regarding pointwise estimates of the Green function $G_\Omega(x_1, x_2)$ in regular domains $\Omega$. 

The following is a classical result that can be found, for example, in \cite[Theorem 3.3]{GW82}. 
\begin{lem}
\label{lem.GreenFct}
Let $\Omega\subset \R^d$ satisfy the exterior and interior ball condition uniformly in $\R^d$. Then, the corresponding Green function $G_\Omega(x_1, x_2)$ satisfies the following inequalities,
\begin{enumerate}[(i)]
\item $0\le G_\Omega(x_1, x_2) \leq C|x_1-x_2|^{2-d}$,
\item $|\nabla_{x_1} G_\Omega(x_1, x_2)| \leq C|x_1-x_2|^{1-d}$,
\item $|\nabla_{x_1}\nabla_{x_2} G_\Omega(x_1, x_2)| \leq C|x_1-x_2|^{-d}$,
\item $|D^2_{x_1,x_1} G_\Omega(x_1, x_2)|\leq C|x_1-x_2|^{1-d} \left(\min\{{\rm dist}(x_1, \de\Omega), |x_1-x_2|\}\right)^{-1}$,
\end{enumerate}
for any $x_1, x_2\in \Omega$, and for some constant $C$ depending only on $d$ and $\Omega$. Moreover, for any $x_1, x_2, z\in \Omega\cap B_2$, and for any $\alpha\in (0, 1)$ we have the bounds
\begin{enumerate}[(v)]
\item $|G_\Omega(x_1, z) - G_\Omega(x_2, z)|\le C|x_1-x_2|^\alpha \left(|x_1-z|^{2-d-\alpha}+|x_2-z|^{2-d-\alpha}\right)$,
\item[(vi)] $|\nabla_x G_\Omega(x_1, z) - \nabla_x G_\Omega(x_2, z)|\le C|x_1-x_2|^\alpha \left(|x_1-z|^{1-d-\alpha}+|x_2-z|^{1-d-\alpha}\right)$,
\end{enumerate}
for some constant $C$ depending only on $d$, $\Omega$, and $\alpha$. 
\end{lem}
\begin{proof}
The first three results, (i)-(ii)-(iii), can be found in \cite[Theorem 3.3]{GW82}. The following result, (iv), follows by the same arguments: 
Let us denote $d(x) = {\rm dist}(x, \de\Omega)$. If $d(x_1) \leq |x_1 - x_2|$ we apply \cite[Lemma 3.1]{GW82} to $\nabla_{x_1} G_\Omega(\cdot, x_2)$ in the ball $B = B_{\frac{1}{2}d(x_1)}(x_1)$, to get 
\[
|\nabla_{x_1}\nabla_{x_1} G_\Omega(x_1, x_2)| \leq C (d(x_1))^{-1}\, \sup_{z\in B} |\nabla_{x_1} G_\Omega(z, x_2)| \leq C (d(x_1))^{-1}|x_1-x_2|^{1-d},
\]
where in the last inequality we are using (ii). On the other hand, if $d(x_1) \ge |x_1-x_2|$ we apply \cite[Lemma 3.1]{GW82} to $\nabla_{x_1} G_\Omega(\cdot, x_2)$ in the ball $B' = B_{\frac{1}{2}|x_1-x_2|}(x_1)$, 
\[
|\nabla_{x_1} \nabla_{x_1}G_\Omega(x_1, x_2)| \leq C (d(x_1))^{-1}\, \sup_{z\in B'} |\nabla_{x_1} G_\Omega(z, x_2)| \leq C |x_1-x_2|^{-1}|x_1-x_2|^{1-d},
\]
where again, we used (ii). 

Finally, result (vi) corresponds to \cite[Theorem 3.5]{GW82}, while result (v) follows by using the same methods as for (vi). 
\end{proof}

In particular, notice that the first inequality (i) together with the argument to prove \eqref{eq.claimE} directly implies that condition {\bf (A-1)} holds, and in particular {\bf (a-A)} also holds for any open set $A$.

\begin{thm}
\label{thm.charact_2}
Let $b:(0, T)\times\R^{2d}\to \R^{2d}$ be the vector field given by $b_t(y, w) = \tilde b_t(y, w)$ defined above, \eqref{eq.uniqvf}, and suppose $\Omega\subset\R^d$ is a $C^{2, 1}$ domain.

Then $b$ satisfies assumption {\bf (b-A)} for $A = B_1\times\R^d$, that is, the uniqueness of bounded compactly supported in $A$ nonnegative distributional solutions of the continuity equation. 
\end{thm}
\begin{proof}
We proceed as in Theorem~\ref{thm.charact}, which at the same time is based on the ideas of \cite[Theorem 4.4]{ACF17} and \cite{BBC16}. We divide the proof into three steps.
\\[0.2cm]
{\bf Step 1: Setting of the problem.} We keep the notation from Theorem~\ref{thm.charact}, and in many steps we will refer to that proof to complete them. Again, we do not explicit the time dependence on the vector field $b$. 

Let $B_r \subset B_1\subset \R^d$, and $B_R\subset \R^d$, so that $B_r\times B_R\subset A$. Let $\eta \in \mathscr{P}(C([0, T); B_r\times B_R))$ be concentrated on integral curves of the vector field $b$ with no concentration condition, $(e_t)_{\#} \eta \leq C_0\left(\mathscr{L}^{2d} \mres A\right)$ for any $t\in [0, T]$. As in Theorem~\ref{thm.charact}, to show that {\bf (b-A)} holds it is enough to prove that the disintegration of $\eta$ with respect to the map $e_0$, $\eta_x$, is a Diract delta for $e_{0\#}\eta$-a.e., where we recall $(e_{0\#}\eta)$ represents the initial condition, $\bar\mu$.

Let $\delta, \zeta\in(0, 1)$ be small parameters to be chosen. We define $\Phi_{\delta, \zeta}(t)$ for $\gamma(t) = (\gamma^1(t), \gamma^2(t))\in \R^d\times\R^d$ as in \eqref{eq.PHI}, 
\begin{equation}
\label{eq.PHI_2}
\Phi_{\delta, \zeta}(t) := \iiint \log \left(1+\frac{|\gamma^1(t) -\xi^1(t)|}{\zeta\delta}+\frac{|\gamma^2(t)-\xi^2(t) |}{\delta}\right) d\mu(x, \xi, \gamma),
\end{equation}
where $d\mu (x, \xi, \gamma) := d\eta_x(\gamma)  d\eta_x(\xi) d\bar \mu(x)$. As in the proof of Theorem~\ref{thm.charact}, in particular \eqref{eq.enoughuniq}, it is enough to show that 
\begin{equation}
\label{eq.enoughuniq_2}
\frac{d\Phi_{\delta, \zeta}}{dt}(t) \leq C\left(\frac{1}{\zeta}+\zeta+\zeta\log\left(\frac{1}{\zeta\delta}\right)\right),
\end{equation}
for some constant $C$ independent of $\zeta$ and $\delta$. Let $b(y, w) = (w, F^o(y, w) + \tilde E(y))$, where 
\begin{equation}
\label{eq.uniqvf_2}
\begin{split}
F^o(y, w) & = \left(\left[J^{-1}(x)\frac{\de w}{\de x}\right]^T w\right)^o\\
\tilde E(y) & = \left(J(x) E(x)\right)^o,
\end{split}
\end{equation}
where the superindices $o$ denote odd extensions with respect to $(y_1, w_1)$ in the sense $V^o(y', w') = (V(y, w))'$. We keep using the notation $x = \psi(y)$, and we recall $E(x)$ is the electric field generated by a renormalized solution $f\in L^1(\Omega\times\R^d)$. 

From \eqref{eq.PHI_2} we separate the temporal derivative in three parts,
\begin{equation}
\label{eq.tempder}
\frac{d\Phi_{\delta, \zeta}}{dt}(t) \leq I + II + III,
\end{equation}
with 
\[
I = \iiint \frac{|\gamma^2(t) -\xi^2(t)|}{\zeta\delta + |\gamma^1(t)-\xi^1(t)| + \zeta|\gamma^2(t) - \xi^2(t)|} d\mu(x, \xi, \gamma),
\]
\[
II = \iiint \frac{|F^o(\gamma^1(t), \gamma^2(t)) - F^o(\xi^1(t), \xi^2(t))|}{\delta + \zeta^{-1}|\gamma^1(t)-\xi^1(t)| + |\gamma^2(t) - \xi^2(t)|} d\mu(x, \xi, \gamma),
\]
and 
\[
III = \iiint \zeta \frac{|\tilde E(\gamma^1(t)) - \tilde E(\xi^1(t))|}{\zeta\delta + |\gamma^1(t)-\xi^1(t)|} d\mu(x, \xi, \gamma).
\]
We can now proceed to bound each one of the three previous terms independently. 
\\[0.2cm]
{\bf Step 2: Bound on $I$ and $II$.} The bound on $I$ follows as in Theorem~\ref{thm.charact}.

Let us now bound the second term, $II$. We proceed by triangular inequality:
\[
II \le \iiint  \frac{|F^o(\gamma^1(t), \gamma^2(t)) - F^o(\gamma^1(t), \xi^2(t))|}{|\gamma^2(t) - \xi^2(t)|} d\mu + \iiint \frac{\zeta |F^o(\gamma^1(t), \xi^2(t)) - F^o(\xi^1(t), \xi^2(t))|}{\zeta\delta + |\gamma^1(t)-\xi^1(t)|} d\mu.
\]
The first term is bounded since, for each $\gamma^1(t)$ fixed, $F^o(\gamma^1(t), \cdot)$ is Lipschitz, with Lipschitz constant given by the maximum of $\left\| (J^{-1})^{T}\frac{\de^2 \phi}{\de x^2} J^{-1}\right\|$ (which is bounded because $\phi$ is $C^{1, 1}$).

For the second term, we notice that if $\gamma^1(t)$ and $\xi^1(t)$ are on the same side ($\gamma_1^1(t)\xi_1^1(t) \ge 0$) then for each fixed $\xi^2(t)\in B_R$, the incremental quotient can be bounded using that $F^o$ is Lipschitz on each side, and for each $\xi^2(t)$. This follows from the fact that the change of variables $\phi$ is $C^{2, 1}$, since we are dealing with $C^{2, 1}$ domains. 

On the other hand, if $\gamma^1(t)$ and $\xi^1(t)$ are on opposite sides, we have to bound 
\begin{align*}
\iiint_{\gamma_1^1(t)\xi_1^1(t) < 0} & \frac{ |F^o(\gamma^1(t), \xi^2(t)) - F^o(\xi^1(t), \xi^2(t))|}{\zeta\delta + |\gamma^1(t)-\xi^1(t)|} d\mu \le \\
& \le \iiint_{\gamma_1^1(t)\xi_1^1(t) < 0} \frac{C d\mu}{\zeta\delta + |\gamma^1_1(t)|}  + \iiint_{\gamma_1^1(t)\xi_1^1(t) < 0} \frac{C d\mu}{\zeta\delta + |\xi^1_1(t)|},
\end{align*}
for some constant $C$ bounding the $L^\infty$ norm of $F^o$ in $B_r\times B_R$, and thus, proportional to $R$. Using the no-concentration condition, $(e_t)_{\#} \eta \leq C_0\left(\mathscr{L}^{2d} \mres A\right)$, it follows
\[
\iiint_{\gamma_1^1(t)\xi_1^1(t) < 0} \frac{C d\mu}{\zeta\delta + |\gamma^1_1(t)|} \le C\int_{B_r} \frac{dx}{\zeta\delta + |x_1|} \leq C\log\left(\frac{1}{\zeta\delta}\right),
\]
so that the bound for $II$ holds. 
\\[0.2cm]
{\bf Step 3: Bound for $III$.} We refer to the appendix to bound the term $III$, since it involves a technical computation that follows analogously to Theorem~\ref{thm.charact}. 
\end{proof}

We can now prove Theorem~\ref{thm.main1_Om}.

\begin{proof}[Proof of Theorem~\ref{thm.main1_Om}]
We proceed from the result in Proposition~\ref{prop.renB1}, using similar ideas to those in the half-space situation. 
\\[0.2cm]
{\bf Step 1.} Take any ball $B_2(x_0)$ with $x_0\in \de\Omega$, and change variables such that, after a rotation and translation, we encounter the situation from Proposition~\ref{prop.renB1}. That is, we are dealing with a vector field $\tilde b_t(y, w) = (w, b_t^{2, o}):(0, T)\times B_1\times \R^d\to \R^d\times\R^d$ given by \eqref{eq.gerenorm2}-\eqref{eq.gerenorm3}. 

Thanks to Theorem~\ref{thm.charact_2}, $\tilde b_t$ fulfills conditions {\bf (a-A)} and {\bf (b-A)} which ensure, by \cite[Theorem 5.2]{ACF15}, the local existence of a Maximal Regular Flow $\Y(t, s, y, w)$ (see Definition~\ref{defi.mrf_A}) in $B_1\times\R^d$. In particular we have 
\[
\partial_t \Y(t, s, y, w) = \tilde b_t(\Y(t, s, y, w)),
\]
for a.e. $t > 0$ whenever it is defined. We claim (and prove in the next step) that this flow is transporting the solution $g_t^{ j} := g_t^e j^{-2}$ from Proposition~\ref{prop.renB1} (analogously, $g_t^{j, \beta} := \beta(g_t^e(y, w))j_e^{-2}(y)$ if the solution is not bounded, for $\beta\in C^1\cap L^\infty$). That is, if without loss of generality we assume $s = 0$ and $t\ge 0$, then
\[
\Y(t, \cdot, \cdot)_\# g_0^{j} \,dy\,dw= g_t^{j} dy\,dw,
\]
where we are denoting $\Y(t, y, w) = \Y(t, 0, y, w)$ and we recall that the push-forward of measures satisfies (see \eqref{eq.pf_cv})
\[
\int_{\R^d\times\R^d} \varphi(\Y(t, y, w)) g_0^{j}(y, w) \,dy\,dw = \int_{\R^d\times\R^d} \varphi(y, w)g_t^{j}(y, w)\,dy\,dw,
\]
for all $\varphi\in C_c^\infty(B_1\times\R^d)$. Thank to Lemma~\ref{lem.mrfmsf} the flow $\Y$ induces a local Maximal Specular Flow in the half space, $\bar \Y$, with vector field given by $(w, b_t^2(y, w))$ (see \eqref{eq.gerenorm2})\footnote{We are using here the natural definition of local Maximal Specular Flow, which arises analogously to that of local Maximal Regular Flow.}. Thanks to the symmetry of $\bar \Y$ and $g_t^{j}$ we have 
\[
\int_{\R^d_+\times\R^d} \varphi(\bar\Y(t, y, w)) g_0(y, w) j^{-2}(y)\,dy\,dw = \int_{\R^d_+\times\R^d} \varphi(y, w)g_t(y, w) j^{-2}(y)\,dy\,dw,
\]
for all $\varphi\in C_c^\infty(\overline{B_1^+}\times\R^d)$. Now, if we change variables back $\Phi(x, v) = (\phi(x), J(x)v)$ and $\Psi = \Phi^{-1}$ and denote $\mathcal{X}(t, x, v) = \Psi\circ \bar \Y(t, \Phi(x, v))$, then 
\[
\int_{\Omega\times\R^d} \bar \varphi(\mathcal{X}(t, x, v)) f_0(x, v) \,dx\,dv = \int_{\Omega\times\R^d} \bar \varphi(x, v)f_t(x, v) \,dx\,dv,
\]
for any $\bar\varphi \in C^\infty_c(\overline{\Omega\cap B_1(x_0)}\times\R^d)$. That is, the solution $f_t$ is transported by a specular flow $\mathcal{X}$. It is easy to check that $\mathcal{X}$ is the flow generated by the vector field $b_t$ by changing variables, and the specular condition at the boundary (condition (ii) in Definition~\ref{defi.msf_dom}) follows using \eqref{eq.jacob1}-\eqref{eq.jacob2} together with the fact that $\bar\Y$ was a local Maximal Specular Flow in the half-space. Therefore, $\mathcal{X}$ is a local Maximal Specular Flow in $\Omega\times\R^d$ transporting the solution.
Now, by a covering argument, gluing together the Maximal Specular Flows (see \cite[Lemma 4.2]{ACF17}), the result follows. 
\\[0.2cm]
{\bf Step 2.} The proof of the claim follows from the proof of \cite[Theorem 5.1]{ACF17}, where the authors show that renormalized (or bounded) solutions to a continuity equation are Lagrangian. In order to be able to apply the proof, we use Proposition~\ref{prop.renB1}. 

The modification of their proof is as follows:

The claim follows from \cite[Theorem 4.7]{ACF17} and \cite[Theorem 5.1]{ACF17} by noticing that the divergence-free condition on their proofs can be substituted by a vector field with locally bounded divergence. Then, Steps 3 and 4 of the proof of \cite[Theorem 5.1]{ACF17} also hold using that $g_t^{\beta, j}$ is a bounded distributional solution of the continuity equation and $\sum_{k \ge 0} g_t^{\beta_k, j} \mathscr{L}^{2d} = g_t^{j} \mathscr{L}^{2d}$, where $\beta_k$ is defined as in \cite[Theorem 5.1]{ACF17} via 
\begin{equation}
\label{eq.betak}
\beta_k(s) =  \left\{ \begin{array}{ll}
0 & \textrm{ if } s\le k, \\
  s-k & \textrm{ if } l\le s\le k+1,\\
    1& \textrm{ if } s\ge k+1.\\
  \end{array}\right.
\end{equation}

To finish the proof, if we are in situation (i) then the solution is transported, and in particular, it is renormalized and fulfils the commutativity property by Theorem~\ref{thm.LagimpRen}
\end{proof}

To finish the section, let us give a sketch of the proof of the incompressibility of the Maximal Specular Flow.

\begin{proof}[Proof of Lemma~\ref{lem.incompress}]
Notice that we can localize the problem, and just check that locally the push-forward of the Lebesgue measure is again the Lebesgue measure. On the other hand, away from the boundary, the vector field is divergence-free, so that the flow is incompressible there. We just need to check that there is no divergence being produced at the boundary: notice that if there was, we would have a singular part of the divergence of the vector field concentrated on the boundary. 

In order to do that, we just symmetrize as in the proof of Theorem~\ref{thm.charact_2}. Namely, we restrict ourselves to a small neighbourhood of the boundary of the domain (which we can do, by the localization). Then, we change variables, and we become a half-space situation. The reflection of the vector field is a new vector field (given by \eqref{eq.uniqvf_2}), which is no longer divergence-free due to the change of variables. Nonetheless, we have computed the divergence in Remark~\ref{rem.incomp}, which is bounded (given that the domain is $C^{1,1}$). I.e., there is no singular part of the divergence of the vector field concentrated on the boundary. On the other hand, after changing variables to a half-space situation, the new Maximal Specular Flow still preserves the specular reflection, and we can symmetrize it to get a Maximal Regular Flow (by taking the continuation across the boundary instead of the specular reflection). This symmetrised flow does not have instant mass destruction or production at the boundary, due to the boundedness of the divergence of the vector field.

This shows that there is no divergence production on the boundary of the domain, and thus the Maximal Specular Flow is incompressible. 
\end{proof}

\section{Proof of Theorems~\ref{thm.main2_d} and \ref{thm.main3_D}}
\label{sec.Main}
\subsection{A regularised problem in domains} In this section we proceed analogously to Subsection~\ref{ssec.regpb} by proving the existence of solutions to a regularised problem of the Vlasov--Poisson system in a $C^{2, 1}$ domain $\Omega\subset\R^d$. We also prove a result regarding conservation of energy. 

The result will follow as Theorem~\ref{thm.regvp}, via a fixed point argument. In this case, however, we must consider a regularised problem with no electric field near the boundary in order to prove the convergence of flows with jumps in the velocity coordinate. The lack of electric field near the boundary is because, a priori, the electric field obtained after regularization of the Green function could still have a Lipschitz constant degenerating when approaching $\de\Omega$.

Let us start by showing how we construct the electric field of the regularised problem. Let $G_\Omega(x_1, x_2)$ denote the Green function of the domain $\Omega$. Let $\bar r\in C^\infty([0, \infty))$ be a monotone function, such that 
\begin{equation}
\label{eq.barr}
\bar r \equiv 0 ~\textrm{  in  }~ [0, 1],\quad \bar r \equiv 1 ~\textrm{  in  }~ [2, \infty),\quad 0\le \bar r' \leq 2~\textrm{  in  }~ [0, \infty).
\end{equation}

For any $\zeta > 0$, we define $\bar r_\Omega^\zeta : \Omega\to [0, 1]$ by 
\begin{equation}
\label{eq.barrom}
\bar r_\Omega^\zeta(x) = \bar r\left(\zeta^{-1} {\rm dist}(x, \de \Omega)\right),\quad\textrm{ for } x \in \Omega.
\end{equation}
Notice that for $\zeta$ small enough depending on the domain, $\bar r_\Omega^\zeta$ is regular (at least $C^2$). We also define $G_\Omega^\delta: \Omega\times\Omega \to [0, \infty)$ by 
\begin{equation}
\label{eq.Gomdelt}
G_\Omega^\delta(x_1, x_2) = \bar r\left(\frac{|x_1-x_2|}{\delta}\right) G_\Omega(x_1, x_2).
\end{equation}
With these previous definitions we can now define a regularised electric field with respect to a density $\rho \in L^1(\Omega)$ by 
\begin{equation}
\label{eq.Eomzetdelta}
E_\Omega^{\zeta, \delta}(x) = -\bar r_\Omega^\zeta(x) \int_\Omega \nabla_{x} G^\delta_\Omega(x, z) \rho(z) \,dz.
\end{equation}

We now claim the following. 
\begin{lem}
\label{lem.claim.E}
The electric field defined as \eqref{eq.Eomzetdelta} is Lipschitz, with Lipschitz constant depending only on $\delta$, $\zeta$, the dimension $d$, and the $L^1$ norm of $\rho$. 
\end{lem}
\begin{proof}
Let us simply bound $\|\nabla_x E_\Omega^{\zeta, \delta}\|_{L^\infty(\Omega)}$:
\begin{equation}
\label{eq.Eomineq}
|\nabla_x E_\Omega^{\zeta, \delta}(x)|\leq C\zeta^{-1}\int_\Omega |\nabla_{x}G^\delta_\Omega(x, z)|\rho(z)\,dx_2+1_{\{{\rm dist}(x, \de\Omega) \ge \zeta\}}\int_\Omega |D^2_x G_\Omega^\delta(x, z)|\rho(z)\,dz,
\end{equation}
where we have used that $|\nabla_x \bar r^\zeta_\Omega(x)| \leq C\zeta^{-1}$ for some constant $C$ depending only on $\bar r$ and $\Omega$. 

Now notice that, thanks to Lemma~\ref{lem.GreenFct} (i)-(ii),
\begin{equation*}
\begin{split}
|\nabla_x G^\delta_\Omega(x, z)|& \le C 1_{\{|x-z|\geq \delta\}}  \left( \delta^{-1} |G_\Omega(x, z)| + |\nabla_x G_\Omega(x, z)|\right)\\
& \leq  C 1_{\{|x-z|\geq \delta\}}  \left(\delta^{-1} |x-z|^{2-d} + |x-z|^{1-d}\right) \le C \delta^{1-d},
\end{split}
\end{equation*}
for some constant $C$ depending only on $d$ and $\Omega$ (and the function $\bar r$ chosen). Similarly, using Lemma~\ref{lem.GreenFct} (iv), 
\begin{equation*}
\begin{split}
|D_x^2 G^\delta_\Omega(x, z)|& \le C 1_{\{|x-z|\geq \delta\}}  \left( \delta^{-2} |G_\Omega(x, z)| + \delta^{-1}|\nabla_x G_\Omega(x, z)| + |D^2_x G_\Omega(x, z)|\right)\\
& \leq  C\delta^{-d} +  C1_{\{|x-z|\geq \delta\}}  |x-z|^{1-d}\left(\min\{{\rm dist}(x, \de\Omega), |x-z|\}\right)^{-1} \\
& \le C \left( \delta^{-d} + \delta^{1-d}\left(\min\{{\rm dist}(x, \de\Omega), \delta\}\right)^{-1}\right).
\end{split}
\end{equation*}
But now notice that the second term in the right-hand side of \eqref{eq.Eomineq} is non-zero on the region $\{{\rm dist}(x, \de\Omega)\ge \zeta\}$, so that putting all together we obtain 
\[
|\nabla_x E_\Omega^{\zeta, \delta}(x)|\leq C\left(\zeta^{-1}\delta^{1-d} + \delta^{-d}\right)\int_\Omega \rho(z) dz = C\delta^{1-d}\left(\zeta^{-1}+\delta^{-1}\right) \|\rho\|_{L^1(\Omega)},
\]
for some constant $C$ depending only on $d$ and $\Omega$.
\end{proof}

The following theorem is analogous to Theorem~\ref{thm.regvp} in regular domains $\Omega\subset\R^d$, with the specular reflection boundary condition. 

\begin{thm}
\label{thm.regvp_d}
Let $\Omega\subset \R^d$ be a $C^{2, 1}$ domain. Let $h_0(x, v) = h_0\in C^\infty_c(\overline{\Omega}\times \R^d)$ be such that for all $\bar x\in \de\Omega$, $v\in R^d$, $h_0(\bar x, v) = h_0(\bar x, R_{\bar x} v)$. Let $\bar r$, $\bar r^\zeta_\Omega$, and $G_\Omega^\delta$ be as above, \eqref{eq.barr}-\eqref{eq.barrom}-\eqref{eq.Gomdelt}, for some $\delta, \zeta > 0$. Then, the following problem has a distributional solution $\bar f_t= \bar f(t, x, v)$,
\begin{equation}
\label{eq.B_reg_d}
  \left\{ \begin{array}{ll}
  \de_t \bar f_t + v\cdot \nabla_x \bar f_t + E_\Omega^{\zeta, \delta}\cdot\nabla_v \bar f_t=0 & \textrm{ in } (0,\infty)\times\Omega\times\R^d\\
    \bar \rho_t(x)=\int_{\R^d} \bar f_t(x, v) dv, & \textrm{ in } (0,\infty)\times\Omega,\\
E_\Omega^{\zeta, \delta}(x) = -\bar r_\Omega^\zeta(x) \int_\Omega \nabla_{x} G^\delta_\Omega(x, z) \bar \rho_t(z) \,dz,  & \textrm{ in } (0,\infty)\times\Omega,\\
\bar f_t(x, v) = \bar f_t(x, R_x v), & \textrm{ on }(0, T)\times\de\Omega\times\R^d,\\
\bar f_0 = h_0& \textrm{ in } \Omega\times\R^d.\\
  \end{array}\right.
\end{equation}
Moreover,
\begin{align}
\label{eq.energy_d}&\int_{\Omega\times\R^d}  |v|^2 \bar f_t(x, v)\,dx\,dv  +  \int_{\Omega\times\Omega} G^\delta_\Omega(x, z)\bar\rho_t(z) \bar \rho_t(x)\,dz\,dx = \\
& ~~\nonumber = \int_{\Omega\times\R^d} |v|^2 \bar f_0(x, v)\,dx\,dv  +  \int_{\Omega\times\Omega} G^\delta_\Omega(x, z)\bar\rho_0(z) \bar \rho_0(x)\,dz\,dx + \int_0^t \mathcal{K}^\Omega_\tau(\delta,\zeta,h_0) d\tau
\end{align}
 for any $t >0$, where $\mathcal{K}_\tau^\Omega$ is defined as
 \begin{equation}
 \label{eq.mk_d}
 \mathcal{K}_\tau^\Omega(\delta,\zeta,h_0) = 2 \int_{\Omega\times\R^d} \big(1- \bar r^{\zeta}_\Omega(x) \big) \,v\cdot \left[\int_\Omega \nabla_x G_\Omega^\delta(x, z) \bar \rho_\tau(z)\,dz\right]\bar f_\tau \,dx\,dv,
 \end{equation}
for $\bar f_\tau$ the solution to \eqref{eq.B_reg_d} with initial datum $h_0$. 
\end{thm}
\begin{proof}
We divide the proof into two steps.
\\[0.2cm]
{\bf Step 1: Existence.} 
We start by proving the existence of a solution for time in $[0, T]$ for $T>0$, using the same approach as in Theorem~\ref{thm.regvp}. Let us define $\bar f_t^n: (0, T)\times\Omega\times\R^d$ for $n\in\N\cup\{0\}$ as the solution to 
\begin{equation}
\label{eq.B_reg_d_n}
  \left\{ \begin{array}{ll}
  \de_t \bar f_t^{n+1} + v\cdot \nabla_x \bar f_t^{n+1} + E_{n,\Omega}^{\zeta, \delta}\cdot\nabla_v \bar f_t^{n+1}=0 & \textrm{ in } (0,\infty)\times\Omega\times\R^d\\
    \bar \rho_t^n(x)=\int_{\R^d} \bar f_t^n(x, v) dv, & \textrm{ in } (0,\infty)\times\Omega,\\
E_{n,\Omega}^{\zeta, \delta}(x) = -\bar r_\Omega^\zeta(x) \int_\Omega \nabla_{x} G^\delta_\Omega(x, z) \bar \rho_t^n(z) \,dz,  & \textrm{ in } (0,\infty)\times\Omega,\\
\bar f_t^{n+1}(x, v) = \bar f_t^{n+1}(x, R_x v), & \textrm{ on }(0, T)\times\de\Omega\times\R^d,\\
\bar f_0^{n+1} = h_0& \textrm{ in } \Omega\times\R^d,\\
  \end{array}\right.
\end{equation}
that is, we have created a sequence of functions iteratively solving the Vlasov--Poisson system with specular reflection in $\Omega$ where the electric field is generated by the density of the previous element of the sequence. Notice that, thanks to Lemma~\ref{lem.claim.E}, by standard Cauchy-Lipschitz theory we can build a flow with jumps transporting the solution $\bar f^1_t$ in $[0, T]$, and proceeding inductively (since the $L^1$ norm is conserved), we can do the same for each element of the sequence. That is, if $\bar b_t^n = (v, E^{\zeta, \delta}_{n, \Omega})$, then there exists a regular flow with jumps $\bar Z_n = (\bar X_n, \bar V_n) : [0, T]\times\overline{\Omega}\times\R^d\to \overline{\Omega}\times\R^d$ such that 
\begin{equation}
\label{eq.ODEsolve}
  \left\{ \begin{array}{ll}
\frac{d}{dt}\bar Z_n(t) = \bar b_t^n(\bar Z_n(t))& \quad \textrm{for}\quad (t, \bar X_n(t), \bar V_n(t)) \in (0, T)\times\Omega\times\R^d,\\
    \bar Z_n(0)(x, v) = (x, v),& \quad \textrm{in}\quad \overline{\Omega}\times\R^d,\\
  \end{array}\right.
\end{equation}
and for $t_*\in (0, T)$ such that $\bar X_n(t_*)\in \de\Omega$, $\bar V_n(t_*^+) = R_{\bar X_n(t_*)} \bar V_n(t_*^-)$, where $t^+_*$ and $t^-_*$ denote the upper and lower temporal limits at $t_*$; and $f_t^{n+1}$ is given by $f_t^{n+1} = \bar Z_n(t)_\# h_0$. Alternatively, since the vector field is divergence free and specular reflection does not produce divergence (see Lemma~\ref{lem.incompress}) one can write $f_t^{n+1}(\bar Z_n(t)) = h_0$.

Notice that we solve the equation \eqref{eq.B_reg_d_n} classically in the interior (since the vector flow is smooth), and the boundary conditions hold noticing that around the boundary the electric field is 0. 

Indeed, take any $(t, x, v)$ with $t\in (0, T)$, $x\in \de\Omega$, and $v\cdot n(x) \le 0$, so that we are looking at velocities \emph{hitting} the boundary instead of \emph{leaving} it. By solving \eqref{eq.ODEsolve} (with jumps) backwards, we obtain some $x_\circ, v_\circ$ such that $(x, v) = \bar Z_n(t)(x_\circ, v_\circ)$. Notice that, from the structure of the flow, since we are in a region where there is no electric field (around $x\in \de \Omega$), for $\eps $ small enough we have that $\bar Z_n(t+\eps)(x_\circ, v_\circ) = (x+\eps R_x v, R_x v)$ (where we also used that $v$ was pointing towards the boundary, so that there is a specular reflection being produced immediately). Now notice
\[
f^{n+1}_{t+\eps}(x+\eps R_x v, R_x v) = f^{n+1}_{t+\eps}(\bar Z_n(t+\eps)(x_\circ, v_\circ))= h_0(x_\circ, v_\circ) = f_t^{n+1}(\bar Z_n(t) (x_\circ, v_\circ)) = f_t^{n+1}(x, v)
\]
for every $\eps > 0$. Take $\eps\downarrow 0$, and we recover the specular reflection condition.

We would like to pass the flows to the limit and proceed as in Theorem~\ref{thm.regvp}. Let $R > 0$ be such that ${\rm supp }~h_0 \subset \R^d\times B_R$, and notice $\|E^{\zeta, \delta}_{n,\Omega}\|_{L^\infty(\Omega)} \le C(\delta)$ for some constant $C(\delta)$ depending only on $\delta$, $d$, $\Omega$, and $\|h_0\|_{L^1(\Omega\times\R^d)}$. From 
\[
\frac{d}{dt}|\bar Z_n(t)|\le |\bar Z_n(t)| + C(\delta),
\]
we obtain ${\rm supp }~\bar f_t^n\subset \R^d\times B_{R+2C(\delta) e^t}(0)$ for $t \ge 0$, and for all $n\in\N\cup\{0\}$. In particular, the vector field $\bar b_t^n(\bar Z_n(t))$ is bounded by some constant $C(\delta, T)$ in the interval $[0, T]$ independently of $n$, and therefore, $\frac{d}{dt}|\bar Z_n|$ is bounded independently of $n$ but depending on $T$. 

On the other hand, notice that the vector fields $\bar b_t^n$ are independent of $n$ in a $\zeta$-neighbourhood of $\de\Omega$, where $\bar b_t^n = (v, 0)$. Now, since $f_t^{n+1} = \bar Z_n (t)_\# h_0$, $\frac{d}{dt}|\bar Z_n|$ is bounded, and $\bar Z_n(t)$ is independent of $n$ in a $\zeta$-neighbourhood of $\de\Omega$, we obtain that up to a sufficiently small time $T_0>0$ independent of $n$, but depending on $\zeta$ and $\delta$, $f_t^{n+1}$ is independent of $n$ in a $\frac{\zeta}{2}$-neighbourhood of $\de\Omega$.

We can now proceed taking limits as in Theorem~\ref{thm.regvp}, by noticing that we only care about the interior of $\Omega$ and therefore the strategy there presented also works here up to minor modifications. This yields a solution up to time $T_0$, but repeating the procedure we can obtain a solution up to time $T$. Since $T>0$ is arbitrary, this gives the desired result.
\\[0.2cm]
{\bf Step 2: Conservation of energy.} Let us compute
\begin{equation}
\label{eq.energy_dt_d} \frac{d}{dt} \left( \int_{\Omega\times\R^d}  |v|^2 \bar f_t(x, v)\,dx\,dv  +  \int_{\Omega\times\Omega} G^\delta_\Omega(x, z)\bar\rho_t(z) \bar \rho_t(x)\,dz\,dx \right).
\end{equation}

We use that
\begin{equation}
\label{eq.conserv_1_d}
 \de_t \bar f_t  = - v\cdot \nabla_x \bar f_t- E_\Omega^{\zeta, \delta}\cdot\nabla_v \bar f_t = - {\rm div}_x (v\bar f_t )- {\rm div}_v (E_\Omega^{\zeta, \delta}\bar f_t ),\quad\textrm{ in } (0, \infty)\times\Omega\times\R^d,
\end{equation}
in the distributional sense, from which,
\begin{equation}
\label{eq.conserv_3_d}
 \de_t \bar \rho_t  = - \int_{\R^d} {\rm div}_x (v\bar f_t )\,dv.
\end{equation}
On the other hand,
\begin{equation}
\label{eq.conserv_2_d}
\begin{aligned}
\frac{d}{dt} \int_{\Omega\times\Omega} G^\delta_\Omega(x, z)\bar\rho_t(z) \bar \rho_t(x)\,dz\,dx   = 2 \int_{\Omega\times\Omega} G^\delta_\Omega(x, z)\bar\rho_t(z) \de_t \bar \rho_t(x)\,dz\,dx \end{aligned},
\end{equation}
so that, plugging in \eqref{eq.conserv_3_d},
\begin{equation}
\label{eq.firstderi_d_2}
\frac{d}{dt} \int_{\Omega\times\Omega} G^\delta_\Omega(x, z)\bar\rho_t(z) \bar \rho_t(x)\,dz\,dx   = -2 \int_{\Omega\times\Omega\times\R^d} G^\delta_\Omega(x, z)\bar\rho_t(z) {\rm div}_x (v\bar f_t(x, v))\,dx\,dz\,dv.
\end{equation}
The divergence theorem yields
\begin{equation*}
\begin{split}
\int_{\Omega} G^\delta_\Omega(x, z)\bar\rho_t(z) {\rm div}_x (v\bar f_t(x, v))\,dx &  = \int_{\de \Omega} G^\delta_\Omega(x, z)\bar\rho_t(z) \bar f_t(x, v)\,v\cdot\nu_\Omega(x) \,d\sigma(x)\,\, \\
& ~~~~~~~- \int_{\Omega} \nabla_x G^\delta_\Omega(x, z)\cdot v\,\bar\rho_t(z) \bar f_t(x, v)\,dx,
\end{split}
\end{equation*}
where $\nu_\Omega(x)$ denotes the outer unit normal of $\de\Omega$ at $x$, and $d\sigma$ is the standard measure at the boundary. Integrating with respect to $v$ and using that $\bar f_t(x, v) = \bar f_t(x, R_x v)$ for $x\in \de\Omega$, the first term vanishes. Therefore, we obtain 
\begin{equation}
\label{eq.firstderi_d}
\frac{d}{dt} \int_{\Omega\times\Omega} G^\delta_\Omega(x, z)\bar\rho_t(z) \bar \rho_t(x)\,dz\,dx   = 2 \int_{\Omega\times\Omega\times\R^d} \nabla_x G^\delta_\Omega(x, z)\cdot v\, \bar\rho_t(z) \bar f_t(x, v)\,dx\,dz\,dv.
\end{equation}

Finally, using \eqref{eq.conserv_1_d}, we obtain
\[
\frac{d}{dt} \int_{\Omega\times\R^d} |v|^2  \bar f_t(x, v)  \,dx\,dv = -\int_{\Omega\times\R^d} |v|^2{\rm div}_x(v\bar f_t)\,dx\,dv -\int_{\Omega\times\R^d} |v|^2{\rm div}_v(E^{\zeta, \delta}_\Omega \bar f_t)\,dx\,dv.
\]
From the divergence theorem, and arguing as before, the first term vanishes. Applying the divergence theorem on the second term, the boundary integral will also vanish, since $E^{\zeta, \delta}_\Omega \equiv 0$ in the neighbourhood of $\de\Omega$. Therefore, 
\[
\frac{d}{dt} \int_{\Omega\times\R^d} |v|^2  \bar f_t(x, v)  \,dx\,dv = 2\int_{\Omega\times\R^d} v\cdot E^{\zeta, \delta}_\Omega \bar f_t \,dx\,dv.
\]
Combining this with \eqref{eq.firstderi_d} and using the definition of $E^{\zeta, \delta}_\Omega$ we obtain the desired result. 
\end{proof}

\subsection{Proof of Theorem~\ref{thm.main2_d}}
Let us now prove Theorem~\ref{thm.main2_d}. We recall $G_\Omega(x_1, x_2)$ denotes the Green function of the domain $\Omega$.

\begin{proof}[Proof of Theorem~\ref{thm.main2_d}]
The proof follows analogously to the proof of Theorem~\ref{thm.mainws}. We divide the proof into several steps.
\\[0.2cm]
{\bf Step 1: Approximating sequence.} Let us construct an approximating sequence based on the regularised solution constructed in Theorem~\ref{thm.regvp_d}. Let $(f_0^n)_{n\in \N}$ with $f_0^n\in C_c^\infty(\overline{\Omega}\times\R^d)$ be a sequence such that for all $x\in \de\Omega$, $v\in \R^d$, $f_0^n(x, v) = f_0^n(x, R_x v)$ and approximating $g_0$ in $L^1$ norm,
\[
f_0^n\rightarrow f_0\quad\textrm{ in } L^1(\Omega\times\R^d).
\]
Consider also $f_t^n$ the sequence of solutions to \eqref{eq.B_reg_d} built in Theorem~\ref{thm.regvp_d} with $\delta = \delta_n$ and $\zeta = \zeta_n$ for some sequences $(\delta_n)_{n\in \N}$, and $(\zeta_n)_{n\in \N}$ with $\delta_n, \zeta_n \rightarrow 0$ to be determined. We recall that we are using the definition of renormalized electric field introduced in \eqref{eq.Eomzetdelta}-\eqref{eq.barrom}-\eqref{eq.Gomdelt}. For the sake of readability we will denote the corresponding vector field to the $n$-th element of the sequence as $b_t(x, v) = (v, E_t^n(x))$, as defined in \eqref{eq.Eomzetdelta}. 

By standard Cauchy-Lipschitz theory we can build an incompressible flow with specular jumps at the boundary transporting the solution, $Z^n(t):\overline{\Omega}\times\R^d\to \overline{\Omega}\times\R^d$ such that
\[
f_t^n = f_0^n \circ Z^n(t)^{-1},\quad\textrm{for}~t\in (0, \infty),
\]
and since they are incompressible,
\begin{equation}
\|\rho_t^n\|_{L^1(\Omega)} = \|f_t^n\|_{L^1(\Omega\times\R^d)} = \|f_0^n\|_{L^1(\Omega\times\R^{d})}, 
\end{equation}
where we recall that $\rho_t^n(x) = \int_{\R^d } f_t^n(x, v) \, dv$ is the physical density. As in \eqref{eq.ei} there exists a sequence $(\eps_m)_{m\in \N}$ depending only on the initial datum $f_0$ with $\eps_n \downarrow 0$ as $m \to \infty$ such that 
\begin{equation}
\label{eq.ei_d}
\int_{\Omega\times\R^d} f_t^n 1_{\{f_t^n > m\}} \,dx\, dv = \int_{\Omega\times\R^d} f_0^n 1_{\{f_0^n > m\}} \,dx\, dv \leq \eps_m\to 0,\quad\textrm{as}~m\to \infty,
\end{equation}
for all $t\in (0, \infty)$ and  $n\in \N$. That is, $f_t^n$ are equiintegrable independently of $n\in \N$ and $t$. 
\\[0.2cm]
{\bf Step 2: Choice of the approximating sequence.} The procedure follows in Step 2 of the proof of Theorem~\ref{thm.mainws} can be repeated here. Notice that, again, proceeding as in \cite[Lemma 3.1]{ACF17} there exists a sequence $f_0^n\in C_c^\infty(\overline{\Omega}\times\R^d)$ and $G_\Omega^{k_n^{-1}}$ such that 
\begin{equation}
\label{eq.lem31_d}
\begin{split}
& \lim_{n\to \infty}\left(\int_{\Omega\times\R^d} |v|^2 f_0^n(x, v)\,dx\,dv  + \int_{\Omega\times\Omega} G^{k_n^{-1}}_\Omega(x, z)\rho_0^n(z)\rho_0^n(x)\,dz\,dx \right) = \\
&~~~~~~~~~~~~~~~~~~~~~~~~~~~~~~= \int_{\Omega\times\R^d} |v|^2 f_0(x, v)\,dx\,dv  + \int_{\Omega\times\Omega} G_\Omega(x, z)\rho_0(z)\rho_0(x)\,dz\,dx.
\end{split}
\end{equation}
In order to prove it we have to use that whenever $\rho_0^n$ is bounded and compactly supported converging to $\rho_0$ then $\int_\Omega G_\Omega(x, z)\rho_0^n(z)\,dz$ converges locally in every $L^p_{\rm loc}$ to $\int_\Omega G_\Omega(x, z)\rho_0(z)\, dz$ by dominated convergence. We also have to use that for every $n\in \N$ fixed 
\[
\lim_{k\to \infty}\int_{\Omega\times\Omega} G_\Omega^{1/k}(x, z)\rho_0^n(x)\rho_0^n(z)\, dz\,dx = \int_{\Omega\times\Omega} G_\Omega(x, z)\rho_0^n(x)\rho_0^n(z)\,dz\,dx
\]
by monotone convergence. 

At this point, proceeding as in Step 2 of the proof of Theorem~\ref{thm.mainws} one can show that there exist sequences $\delta_n, \zeta_n\to 0$ such that 
\begin{equation}
\label{eq.Ktom}
|\mathcal{K}_\tau^\Omega(\delta_n, \zeta_n, h_0^n)| \to 0\quad\textrm{ as } n \to \infty,
\end{equation}
uniformly for $\tau \in (0, t)$, for every $t > 0$. We recall that $\mathcal{K}_\tau^\Omega$ is given by \eqref{eq.mk_d}. 
In particular, again as in Theorem~\ref{thm.mainws}, we obtain a uniform bound for the kinetic energy, 
\begin{equation}
\label{eq.boundke_d}
\liminf_{n\to\infty} \int_{\Omega\times\R^d} |v|^2 f_t^n(x, v)\,dx\,dv  \leq C,
\end{equation}
for some constant $C$ independent of $n$ and $t$. 
\\[0.2cm]
{\bf Step 3: Limiting solution.} Transporting level sets as in Step 3 of Theorem~\ref{thm.mainws} we can construct a sequence of functions $f^{n, k}_t$ bounded by $k+1$ transporting level sets of the initial datum, with 
\[
\|f_t^{n, k}\|_{L^1(\Omega\times\R^d} = \|f_0^{n, k}\|_{L^1(\Omega\times\R^d},
\]
and such that 
\begin{equation}
\label{eq.gk_d}
f^{n,k}\rightharpoonup \bar f^k\quad\quad \textrm{weakly}^*\textrm{ in } L^\infty((0, \infty)\times\Omega\times\R^{d}) \textrm{ as } n\to \infty,\quad\textrm{for all }k\in \N.
\end{equation}
Defining 
\begin{equation}
\label{eq.gk2_d}
f := \sum_{k = 0}^\infty \bar f^k \quad \textrm{ in } (0, \infty)\times\Omega\times\R^{d},
\end{equation}
then, 
\begin{equation}
\label{eq.convseries_d}
\|f_t\|_{L^1(\Omega\times\R^{d})} \leq \|f_0\|_{L^1(\Omega\times\R^{d})}\quad\textrm{ for a.e. } t\in [0, \infty),
\end{equation}
and 
\begin{equation}
\label{eq.weakg_d}
f^n \rightharpoonup f\quad\textrm{ weakly in }L^1_{\rm loc}([0, T]\times\Omega\times\R^{d}),
\end{equation}
for every $T > 0$.
\\[0.2cm]
{\bf Step 4: Limiting densities.} Again, as in Theorem~\ref{thm.mainws}, if we define the limiting density $\rho_t(x) = \int_{\R^d} f_t(x, v)\, dv$, one can prove using the bound on the kinetic energy \eqref{eq.boundke_d} that 
\begin{equation}
\label{eq.convdens_d}
\rho^n\rightharpoonup \rho\quad\textrm{ weakly in } L^1_{\rm loc}([0, T]\times\Omega),
\end{equation}
that is, the densities weakly converge to the density of the limit, up to subsequences. 
\\[0.2cm]
{\bf Step 5: Limiting vector fields.} We would like now to apply the results analogous to the Di Perna and Lions in \cite{DL89b} to get that the limiting solution is actually a renormalized solution. In order to do that, we have to get rid of the specular reflection condition, by invoking the change of variables and the odd reflections of the electric field already introduced in the previous sections. 

That is, let us fix any point on the boundary and after a translation and rotation let us assume $0\in \de\Omega$ and $n(0) = e_1$. After a rescaling we also assume that the domain $\Omega$ fulfils the exterior and interior ball condition at each boundary point with balls of radius 2. We will first prove that the limiting solution is actually renormalized in $B_1$, from which it will be renormalized everywhere (by a covering argument and taking further subsequences), and therefore it will be transported by a Maximal Specular Flow in $\Omega\times\R^d$.  

We perform the change of variables from Section~\ref{sec.gendom}, and we will keep using the notation there introduced; $(x, v)\mapsto(y, w) = (\phi(x), J(x)v)$. Let us define 
\[
g_t^n(y,w ) = f_t^n(x, v),
\]
and $g_t^{e, n}(y, w)$ its even extension with respect to $(y_1, w_1)$; that is, $g_t^{e, n}(y, w) = g_t^{n}(y, w)$ if $y_1 \ge 0$ and $g_t^{e, n}(y, w) = g_t^{n}(y', w')$ otherwise. 

We analogously define $\bar g^{e, k}_t(y, w)$ and $g^e_t(y, w)$ in $(0, T)\times B_1\times\R^d$ from $\bar f^k_t(x, v)$ and $f_t$, the limits in \eqref{eq.gk_d}-\eqref{eq.gk2_d}; so that
\begin{equation}
\label{eq.gbardom}
g_t^e = \sum_{k = 0}^\infty \bar g^{e, k}_t.
\end{equation} 

By Proposition~\ref{prop.renB1_T} we have that $g_t^{e, n}$ are renormalized solutions in $B_1\times \R^d$ to the transport equation
\begin{equation}
\label{eq.eqtransf_2}
\de_t g_t^{e, n}(y, w) +  w\cdot \nabla_y g_t^{e, n}(y, w) + F^n(y, w)\cdot\nabla_w g_t^{e, n}(y, w) = 0,
\end{equation}
where 
\[
F^n(y, w) := F_1^o(y, w) + F_2^{n,o}(y)
\]
and $F_1^o, F_2^{n,o}$ are the odd extensions with respect to $(y_1, w_1)$ of
\[
F_1(y, w) := \left[(J^{-1}(x))^T\frac{\de w}{\de x}\right]^T w = v^T D^2\phi\, v
\]
and 
\[
F_2^n(y) := J(x) E^n(x),
\]
that is, $F_1^o(y, w) = F_1(y, w)$, $F_2^{o, n}(y) = F_2^n(y)$, if $y_1\ge 0$, and $F_1^o(y, w) = (F_1(y', w'))'$, $F_2^{o, n}(y) = (F_2^n(y'))'$, otherwise. By means of the same reasoning as in Theorem~\ref{thm.mainws} we want to show that the even extension, $g_t^e$, to
\[
g_t(y, w) = f_t(y, w),
\]
is a distributional solution in $B_1\times\R^d$ to 
\begin{equation}
\label{eq.eqtransf_3}
\de_t g_t^{e}(y, w) +  w\cdot \nabla_y g_t^{e}(y, w) + F(y, w)\cdot\nabla_w g_t^{e}(y, w) = 0,
\end{equation}
where now $F = F_1^o + F_2^o$, and $F_2^o$ is the odd extension with respect to $(y_1, w_1)$ of $J(x)E(x)$, with 
\[
E(x) = -\int_{\Omega}\nabla_x G_\Omega(x, z)\rho(z)\,dz.
\]

In order to do that, as in Theorem~\ref{thm.mainws}, it will be enough to show 
\begin{equation}
\label{eq.1stab}
F_2^{n, o}\rightharpoonup F_2^o\quad\textrm{weakly in }~L^1_{\rm loc}((0, \infty)\times B_1; \R^d),
\end{equation}
and 
\begin{equation}
\label{eq.2stab}
F_2^{n, o}(y+h)\to F_2^n(y)\quad\textrm{as }|h|\to 0,~\textrm{ in } L^1_{\rm loc}((0, \infty); L^1_{\rm loc}(B_1))
\end{equation}
uniformly in $n$.
\\[0.2cm]
{\bf Step 6: Proof of the first and second stability condition.} We refer to the appendix to show that \eqref{eq.1stab}-\eqref{eq.2stab} actually hold, since it is a technical computation. 
 \\[0.2cm]
{\bf Step 7: Conclusion of existence.}
Since conditions \eqref{eq.1stab}-\eqref{eq.2stab} are fulfilled, by the same arguments as in \cite[Theorem II.7]{DL89b} the vector fields $(w, F^n)$ are converging strongly in $L^1$. Therefore, weakly continuous bounded solutions of the approximating problems converging weakly$^*$ in $L^\infty$ are distributional solutions in the limit (notice that the divergence of the vector field is constant in $n$). In particular, for every $m \in \N$, $G^m_t = \sum_{k = 0}^m \bar g^{e,k}_t$ (recall $\bar g^{e,k}_t$ from \eqref{eq.gbardom}) is a distributional solution of the continuity equation in $(0, T)\times B_1\times\R^d$ with vector field $(w, F)$ and with initial datum $G_0^m = \sum_{k = 0}^m \bar g^{e,k}_0$; as it is bounded by $m+1$. 

Then, by Proposition~\ref{prop.renB1_T2} (which is based on Proposition~\ref{prop.BimpA}) we have that $\mathcal{F}^m_t = \sum_{k = 0}^m \bar f_t^k$ are distributional solutions to \eqref{eq.original_Om} in $(0, T)\times (B_{1/2}\cap \Omega)\times\R^d$ according to Definition~\ref{defi.HS_Om}, and in particular, by a covering argument, they are distributional solutions in $(0, T)\times\Omega\times\R^d$. 

By Theorem~\ref{thm.main1_Om} (i), since $\mathcal{F}^m$ is bounded, it is a renormalized solution and it is transported by the corresponding Maximal Specular Flow. Since $\mathcal{F}^m$ converges to $f_t$ in $L_{\rm loc}^1((0, \infty)\times\Omega\times\R^{d})$, the limiting $f_t$ is also a renormalized solution; and by Theorem~\ref{thm.main1_Om} (ii), it is transported by the Maximal Specular Flow. Moreover, by Theorem~\ref{thm.LagimpRen}, $f_t \in C([0, \infty); L^1_{\rm loc}(\Omega\times\R^{d}))$ and fulfils the commutativity property.
\\[0.2cm]
{\bf Step 8: Strong $L^1_{\rm loc}$ continuity of density and electric field.} The strong continuity of the densities $\rho \in C([0, \infty); L^1_{\rm loc}(\Omega))$ follows exactly as in Step 8 of the proof of Theorem~\ref{thm.mainws}; and the strong continuity of the electric fields, $E\in  C([0, \infty); L^1_{\rm loc}(\Omega))$, also follows like in Step 8 of the proof of Theorem~\ref{thm.mainws} combined with the estimates on the Green function from Lemma~\ref{lem.GreenFct}. 
\end{proof}

\subsection{Proof of Theorem~\ref{thm.main3_D}} We can now prove the result regarding the bound of the total energy for positive times.

\begin{proof}[Proof of Theorem~\ref{thm.main3_D}]
We divide the proof into four steps. We will be using the notation from the proof of Theorem~\ref{thm.main2_d}, where we built a sequence of functions $f_t^n$ converging weakly to $f_t$. 
\\[0.2cm]
{\bf Step 1: Weak uniform equicontinuity in time of densities.} We start by proving the weak equicontinuity in time of the densities of the approximating sequence. Notice that, in the interior of $\Omega$,
\[
\de_t\rho_t + {\rm div}_x \left(\int_{\R^d} v\, f_t(x, v) \, dv\right) = 0,
\]
and the same holds for each $f_t^n$ and $\rho_t^n$. Let us now prove that $t\mapsto \int_{\R^d}\rho_t^n \xi $ are equicontinuous for any $\xi \in C^\infty_c(\overline{\Omega})$:
\[
\begin{split}
\left|\int_\Omega (\rho_r^n- \rho_s^n)\xi \right|&  = \left|\int_s^r \int_\Omega \de_t \rho_\tau^n \xi \, dx\, d\tau\right| 
= \left|\int_s^r\int_{\Omega} {\rm div}_x \left(\int_{\R^d} v f_\tau^n(x, v)\,dv \right)\,\xi\,dx\,d\tau\right|\\
 & \leq \|\nabla \xi\|_{L^\infty} \left|\int_s^r\int_{ \Omega\times\R^{d}}  |v|f_\tau^n(x, v) \,dx\,dv\,d\tau\right| +\\
 & ~~~~~~~~~~~+ \left|\int_s^r\int_{\de\Omega\times\R^d} n(x)\cdot v\,f_\tau^n(x, v) \xi(x) \, d\sigma^\Omega_x \,dv\, d\tau \right|.
\end{split}
\]
Using the symmetry of $g_\tau^n$ in the boundary, $g_\tau^n(x, v) = g_\tau^n(x, R_x v)$, and integrating first in the $v$ variable, the second term vanishes. By H\"older inequality, we have 
\[
\left(\int_{\Omega\times\R^d} |v| f_\tau^n(x, v)\, dx\, dv\right)^2 \le \left(\int_{\Omega\times\R^d} |v|^2 f_\tau^n(x, v)\, dx\, dv\right) \cdot \int_{\Omega}\rho_\tau^n(x)\, dx,
\]
which is bounded by the uniform bound in $n$ of the kinetic energy, \eqref{eq.boundke_d}, and the uniform bound on the $L^1$ norm of $f_t^n$. Therefore, we have 
\[
\left|\int_\Omega (\rho_r^n- \rho_s^n)\xi \right| \le C\|\nabla \xi \|_{L^\infty} |r- s|,
\]
the weak equicontinuity in time of the densities. This, together with the weak$^*$ convergence of measures \eqref{eq.convdens_d}, implies
\begin{equation}
\label{eq.unifwcmt_d}
\lim_{n\to \infty}\sup_{t\in [0, T]}\left|\int_\Omega \xi(x)(\rho_t^n - \rho_t)(x)\, dx\right| = 0,
\end{equation}
the uniform convergence in $t$ of $\int_\Omega\rho_t^n\xi$.
\\[0.2cm]
{\bf Step 2: Weak lower semicontinuity of the potential term and bound on energy.} We prove that, for any nonnegative $\phi \in C^\infty_c((0, \infty))$ we have 
\begin{equation}
\label{eq.lscpe_d}
\begin{split}
\int_0^\infty\phi(t)&  \int_{\Omega\times\Omega} G_\Omega(x, z) \rho_t(x)\rho_t(z)\, dx\, dz\, dt  \\
& ~~~~~~~~~~\le \liminf_{n \to \infty}\int_0^\infty\phi(t) \int_{\Omega\times\Omega} G^{k_n^{-1}}_\Omega(x, z) \rho_t^n(x)\rho_t^n(z)\, dx\, dz\, dt.
\end{split}
\end{equation}
Thanks to \eqref{eq.unifwcmt_d} and Fubini's theorem we have that $\rho_t^n(x) dx \otimes\rho_t^n(z) dz\otimes dt\in \mathcal{M}((0, T)\times\Omega\times\R^d)$ converge against continuous functions to $\rho_t(x) dx \otimes\rho_t(z) dz\otimes dt$. In particular, for any $m\in \N$ fixed, we have that 
\[
\begin{split}
\int_0^\infty\phi(t) \int_{\Omega\times\Omega} G_\Omega^{k_m^{-1}}(x, z) \rho_t(x)\rho_t(z)\,dx dt & \le  \liminf_{n \to \infty}\int_0^\infty\phi(t) \int_{\Omega\times\Omega} G_\Omega^{k_m^{-1}}(x, z) \rho_t^n(x)\rho_t^n(z)\,dx\,dz\, dt\\
&\le  \liminf_{n \to \infty}\int_0^\infty\phi(t) \int_{\Omega\times\Omega} G_\Omega^{k_n^{-1}}(x, z) \rho_t^n(x)\rho_t^n(z)\,dx\,dz\,dt,
\end{split}
\]
where in the last inequality we have used that $G_\Omega^{k_m^{-1}}$ is increasing in $m\in \N$ and nonnegative. For the same reason, the left-hand side converges (by monotone convergence theorem), and in particular we obtain the desired result, \eqref{eq.lscpe_d}. 

Finally, from the lower semicontinuity of the kinetic energy, \eqref{eq.boundke_d}, we immediately have that 
\[
\int_{0}^\infty\int_{\Omega\times\R^d} \phi(t) |v|^2 f_t \,dx\,dv\, dt\leq \liminf_{n \to \infty} \int_{0}^\infty\int_{\R^d\times\R^d} \phi(t) |v|^2 f_t^n \,dx\,dv\, dt;
\]
which combined with \eqref{eq.energy_d}-\eqref{eq.Ktom}-\eqref{eq.lscpe_d}-\eqref{eq.lem31_d} yields that, for any $\phi\in C_c^\infty((0, \infty))$,
\begin{align*}
& \int_0^\infty\phi(t) \left\{\int_{\Omega\times\R^d}|v|^2f_t \, dx\, dv + \int_{\Omega\times\Omega} G_\Omega(x, z) \rho_t(x) \,\rho_t(z) \, dx\,dz \right\}\, dt \\
& ~~~~~~~~~~\leq \liminf_{n \to \infty} \int_0^\infty\phi(t) \left\{\int_{\Omega\times\R^d}|v|^2f_t^n \, dx\, dv + \int_{\Omega\times\Omega} G_\Omega^{k_n^{-1}}(x, z) \rho_t^n(x)\rho_t^n(z)\,dx\,dz \right\}\, dt\\
& ~~~~~~~~~~\leq \liminf_{n \to \infty} \int_0^\infty\phi(t) \left\{\int_{\Omega\times\R^d}|v|^2f_0^n \, dx\, dv + \int_{\Omega\times\Omega} G_\Omega^{k_n^{-1}}(x, z) \rho_0^n(x)\rho_0^n(z)\,dx\,dz \right\}\, dt\\
& ~~~~~~~~~~= \int_0^\infty\phi(t) \left\{\int_{\Omega\times\R^d}|v|^2f_0 \, dx\, dv + \int_{\Omega\times\Omega} G_\Omega(x, z) \rho_0(x) \,\rho_0(z) \, dx\,dz \right\}\, dt
\end{align*}

In particular, \eqref{eq.boundenergy__D} holds for a.e. $t \geq 0$. The boundedness of the energy for all times follows form the fact that  $f_t\in C([0, \infty); \Omega\times\R^d)$ and $\rho_t\in C([0, \infty); \Omega)$.
\\[0.2cm]
{\bf Step 3: Bound on $L^2$ norm of electric field.} We will show that for any nonnegative $L^1$ function $\xi$, then 
\begin{equation}
\label{eq.pecond_d}
\int_{\Omega\times\Omega} G_\Omega(x, z)\, \xi(x)\, \xi(z) \, dx\,dz \ge \int_{\Omega} \left|\nabla \int_\Omega G_\Omega(x, z)\xi(z)\, dz\right|^2\, dx.
\end{equation} 
For $\xi\in C^\infty_c(\R^d)$ and $R > 0$, extending $G$ by $0$ outside $\Omega$, we have
\[
\begin{split}
\int_{B_R\times \Omega} G_\Omega(x, z) \xi(x)\xi(z) & = \int_{B_R}\left|\nabla \int_{\Omega} G_\Omega(x, z)\xi(z)  \right|^2 \\
&~~~~ - \int_{(\de B_R)\cap \Omega} \int_\Omega G_\Omega(x, z)\xi(z) \nabla \int_\Omega G_\Omega(x, z)\xi(z) \cdot \nu_{B_R}\\
&~~~~ + \int_{\de \Omega\cap B_R}\int_\Omega G_\Omega(x, z)\xi(z) \nabla \int_\Omega G_\Omega(x, z)\xi(z) \cdot n(x).
\end{split}
\]
Since $G_\Omega(x, z)\le C|x-z|^{2-d}$ and $|\nabla_x G_\Omega(x, z) |\le C|x-z|^{1-d}$, it is easy to check that the second term above tends to 0 as $R \to \infty$. The third term is equal to 0 exactly, since $G_\Omega(x, z) = 0$ for $x\in \de\Omega$. By approximation the same holds for $\xi\in L^\infty_c(\R^d)$. Finally, taking $\xi_k = \min\{1_{B_k} \xi(x), k\}$ and by monotone convergence we have that
\[
\begin{split}
\int_{\Omega\times\Omega} G_\Omega(x, z)\, \xi(x)\, \xi(z) \, dx\,dz  & = \lim_{k\to\infty} \int_{\Omega\times\Omega} G_\Omega(x, z)\, \xi_k(x)\, \xi_k(z) \, dx\,dz \\
& = \liminf_{k\to\infty} \int_{\Omega} \left|\nabla \int_\Omega G_\Omega(x, z)\xi_k(z)\, dz\right|^2\, dx\\
& \ge \int_{\Omega} \left|\nabla \int_\Omega G_\Omega(x, z)\xi(z)\, dz\right|^2\, dx,
\end{split}
\]
where in the last step we have used the lower semicontinuity of the $L^2$ norm with respect to the weak convergence of $\xi_k$ to $\xi$.
\\[0.2cm]
{\bf Step 4: Proof of (ii).} Let us prove the no-blow up criterion for $f_0$-a.e. $(x, v)\in \Omega\times\R^d$ of the Maximal Specular Flow, in $d = 3, 4$. 

Let call $Z_t^s(x, v) = Z(t, s, x, v) := (X(t, s, x, v), V(t, s, x,v))$ the Maximal Specular Flow, and similarly $t_{s, Z}^\pm (x, v) = t_{s, X, V}^\pm (x, v)$. Take $T > 0$, so that $(t_{s, Z}^-, t_{s, Z}^+)\subset [0, T]$. Let us show that for $f_s$-a.e. $(x, v)$ we can take $t_{s, Z}^- = 0$ and $t_{s, Z}^+ = T$, for $s\in[0, T]$.

In particular, it will be enough to show
\begin{equation}
\label{eq.toshownoblowup}
\sup_{t_{s, Z}^- \le r, \tau \le t_{s, Z}^+} |\log\log (2+|Z_r^{s}|) - \log\log(2+|Z_\tau^{s}|)| <  \infty,
\end{equation}
for $f_s$-a.e. $(x, v) \in \Omega\times\R^d$. Let us proceed,
\[
 \begin{split}
 \int_{\Omega\times\R^d} |\log\log & (2+|Z_r^{s}|) - \log\log(2+|Z_\tau^{s}|)| f_{s_n}\, dx\,dv\le\\
 & \le  \int_{\Omega\times\R^d} \int_{t_{s, Z}^-}^{t_{s, Z}^+} \left|\frac{d}{dt}\log\log  (2+|Z_t^{s}|)\right|\, dt f_{s}\, dx\,dv\\  
 & \le  \int_{t_{s, Z}^-}^{t_{s, Z}^+} \int_{\Omega\times\R^d} \frac{|b_t(Z_t^{s})|}{(1+|Z_t^{s}|) \log  (2+|Z_t^{s}|)} \, f_{s}\, dx\,dv\, dt\\
  & \le  \int_{0}^{T} \int_{\Omega\times\R^d} \frac{|b_t(x, v)|}{(1+|(x, v)|) \log  (2+|(x, v)|)} \, f_{t}(x, v)\, dx\,dv\, dt <\infty.\\
 \end{split}
\]

We have used here the incompressibility of the flow and the transport structure. The last inequality follows as in \cite[Theorem 2.3]{ACF17} by means of \eqref{eq.pecond_d}, choosing $\xi = \rho_t$ and noticing that the right-hand side of \eqref{eq.pecond_d} is now precisely the $L^2$ norm of the electric field $E_t$. Now, taking $s = 0$ shows that trajectories do not blow up in finite time for $f_0$-a.e. $(x, v)\in \Omega\times\R^d$; and since for every $s\in [0, T]$ we can set $t_{s, Z}^- = 0$ and $t_{s, Z}^+ = T$ for $f_s$-a.e. $(x, v)\in \Omega\times\R^d$, there is no appearance of mass from infinity at any time (given that the flow can be extended back up to $t_{s,Z}^- = 0$), and $f_t$ is the image of $f_0$ through an incompressible flow. 
\end{proof}

\section*{Appendix}

We include in this section the technical computations from the work. We start with the proof of Step 3 in Theorem~\ref{thm.charact_2}.

\subsection*{Proof of Step 3 Theorem~\ref{thm.charact_2}}

We have to bound the term, $III$. We divide this proof into three further steps.
\\[0.2 cm]
{\it Step 1.} We approach the problem parallelly to what we did for the half space, although now the electric field is not given by the convolution against an $L^1$ function. We consider two different cases, according to whether $\gamma^1(t)$ and $\xi^1(t)$ are on the same side or not. That is, we let 
\[
III\le III_+ + III_-, 
\]
with
\[
III_\pm = \iiint_{\pm\gamma_1^1(t)\xi_1^1(t) \ge 0 } \zeta \frac{|\tilde E(\gamma^1(t)) - \tilde E(\xi^1(t))|}{\zeta\delta + |\gamma^1(t)-\xi^1(t)|} d\mu(x, \xi, \gamma).
\]

Let us start focusing on $III_+$, and without loss of generality we will assume that both $\gamma^1(t), \xi^1(t)\in \{y_1 \ge 0\}$ (otherwise we use the symmetry of the field). 

Notice that, using a triangular inequality as in previous steps and the fact that $J(x)$ is Lipschitz, we only need to bound 
\begin{equation}
\label{eq.bdiiip}
\iiint_{\gamma_1^1(t)\xi_1^1(t) \ge 0 } \zeta \frac{|E\left(\psi(\gamma^1(t))\right) - E\left(\psi(\xi^1(t))\right)|}{\zeta\delta + |\gamma^1(t)-\xi^1(t)|} d\mu(x, \xi, \gamma).
\end{equation}

We recall that, from \eqref{eq.greenv},
\[
E(x) = \int_\Omega \nabla_x G_\Omega(x, x_2)\rho(x_2) dx_2,
\]
for some $\rho\in L^1(\Omega)$ by assumption. We also recall that the Green function satisfies
\[
\left\{ \begin{array}{ll}
  -\Delta_{x_1} G_\Omega(x_1, x_2) = -\Delta_{x_2} G_\Omega(x_1, x_2) = \delta(x_1 - x_2), & \textrm{ for } x_1, x_2 \in \Omega \\
  G_\Omega(x_1, x_2) = 0 & \textrm{ otherwise},\\
  \end{array}\right.
\]
where we have extended it by 0 outside the domain $\Omega$, in the whole $\R^d\times\R^d$. On the other hand, we define by $\Gamma(x_1 - x_2)$ the fundamental solution in the whole $\R^d$ for $d \ge 3$, that is, $\Gamma(x) = w_d^{-1} |x|^{2-d}$ for some dimensional constant $w_d$.

From the integrability of $E$ (see Lemma~\ref{lem.GreenFct}) it is easy to see that in \eqref{eq.bdiiip} we are only interested in the cases where $|\gamma^1(t)-\xi^1(t)|$ is small. On the other hand, if both $\gamma^1(t)$ and $\xi^1(t)$ are uniformly far from the boundary $\{y_1 = 0\}$ (that is, $\psi(\gamma^1(t))$ and $\psi(\xi^1(t))$ are uniformly far from $\de \Omega$), then $D_\Omega(x_1, x_2) := G_\Omega(x_1, x_2) - \Gamma(x_1 - x_2)$ is harmonic in $x_1$ and with uniform bounds in the boundary, thus smooth in the interior. Hence, if we denote $z_\xi = \psi(\xi^1(t))$ and $z_\gamma = \psi(\gamma^1(t))$, and $\nabla_1$ denotes the gradient with respect the first $d$ coordinates, we have
\begin{align*}
|E\left(z_\gamma\right) - E\left(z_\xi\right)| \leq & \left| \int_\Omega \big( \nabla_1 D_\Omega \left(z_\gamma, x \right) - \nabla_1 D_\Omega \left(z_\xi , x\right) \big)\rho(x) dx \right|
\\ & + \left| \int_{\R^d\setminus \Omega} \big( \nabla \Gamma \left(z_\gamma- x \right) - \nabla \Gamma \left(z_\xi - x\right) \big)\rho(x) dx \right|
\\ &~~~+ \left| \int_{\R^d} \big( \nabla \Gamma \left(z_\gamma- x \right) - \nabla \Gamma \left(z_\xi - x\right) \big)\rho(x) dx \right|.
\end{align*}

Notice that the first two terms in the previous expression are bounded when plugged into \eqref{eq.bdiiip} if $z_\gamma$ and $z_\xi$ are uniformly far from $\de \Omega$, due to the smoothness of $D_\Omega$ and $\Gamma$ in the corresponding integration areas. On the other hand, the last term corresponds to the case dealt in \cite[Theorem 4.4]{ACF17}, the convolution of a singular integral (given by the fundamental solution in $\R^d$) against an $L^1$ function. 

We can, therefore, assume that $z_\gamma$ and $z_\xi$ are close to the boundary. In particular, we will assume that they have unique projections, so that in the expression 
\begin{align*}
|E\left(z_\gamma\right) - E\left(z_\xi\right)| \leq & \left| \int_{\Omega_\pi} \big( \nabla_1 G_\Omega \left(z_\gamma, x \right) - \nabla_1 G_\Omega \left(z_\xi , x\right) \big)\rho(x) dx \right|
\\ & + \left| \int_{\Omega\setminus\Omega_\pi} \big( \nabla_1 G_\Omega \left(z_\gamma, x \right) - \nabla_1 G_\Omega \left(z_\xi , x\right) \big)\rho(x) dx \right|,
\end{align*}
the second term is immediately bounded due to the regularity of $G_\Omega$ in the integration domain (we recall $\Omega_\pi$ denotes the domain of unique projection). Thus, when computing the electric fields, we only care about the contribution of the densities close to the boundary. 

Let us define, for $x \in \Omega_\pi$, $Px$ as the reflected point with respect to $\de \Omega$. That is, 
\[
Px = x + 2(\pi(x)-x) \in \R^d\setminus \Omega.
\]
We analogously define the same operator for the points on $(\R^d\setminus \Omega)_\pi$. The sets of unique projection from either side are comparable, since we have exterior and interior ball condition for $\Omega$.

Using the same ideas as in Theorem~\ref{thm.charact}, we can consider, for $x_2\in \Omega_\pi$,
\[
G_\Omega(x_1, x_2) = \Gamma(x_1 - x_2) - \Gamma(x_1 - Px_2) - H(x_1, x_2),
\]
where $H$ fulfills, for each $x_2 \in \Omega_\pi$,
\begin{equation}
\label{eq.Hff}
\left\{ \begin{array}{ll}
  \Delta_{x_1} H(x_1, x_2) = 0, & \textrm{ for } x_1 \in \Omega \\
  H(x_1, x_2) = \Gamma(x_1 - x_2) - \Gamma(x_1 - Px_2) & \textrm{ for } x_1 \in \de \Omega.\\
  \end{array}\right.
\end{equation}

Putting all together we have
\begin{align*}
 \int_{\Omega_\pi} & \nabla_1 G_\Omega \left(z_\gamma, x \right) \rho(x) dx  =  \\ 
 & = \int_{\Omega_\pi} \nabla \Gamma \left(z_\gamma- x \right) \rho(x) dx + \int_{\Omega_\pi} \nabla \Gamma \left(z_\gamma- Px \right) \rho(x) dx + \int_{\Omega_\pi} \nabla_1 H \left(z_\gamma, x \right) \rho(x) dx.
\end{align*}
Let us now denote $\rho_\pi(x) = \rho(x) 1_{\{x\in \Omega_\pi\}} + j_P(x)\rho(Px) 1_{\{x\in P\Omega_\pi\}}$, where $P\Omega_\pi$ is the reflected of the $\Omega_\pi$ with respect to $\de\Omega$, and $j_P(x)$ is the Jacobian determinant of the change of variables $x\mapsto Px$.  Notice that the $L^1$ norm of $\rho_\pi$ is bounded by the $L^1$ norm of $\rho(x)$. If we change variables in the previous expression we have 
\[
 \int_{\Omega_\pi}  \nabla_1 G_\Omega \left(z_\gamma, x \right) \rho(x) dx   = \int_{\R^d} \nabla \Gamma \left(z_\gamma- x \right) \rho_\pi(x) dx + \int_{\Omega_\pi} \nabla_1 H \left(z_\gamma, x \right) \rho(x) dx.
\]

As can be seen, the first term is again of the form treated in \cite[Theorem 4.4]{ACF17}, a convolution of the gradient of the fundamental solution against an $L^1$ function. Putting it back in \eqref{eq.bdiiip} the corresponding bound follows. Therefore, we have reduced the bound on $III_+$ to finding a bound for 
\begin{equation}
\label{eq.bdiiip_2}
\iiint_{\gamma_1^1(t)\xi_1^1(t) \ge 0 } \zeta \frac{|E_H\left(\psi(\gamma^1(t))\right) - E_H\left(\psi(\xi^1(t))\right)|}{\zeta\delta + |\gamma^1(t)-\xi^1(t)|} d\mu(x, \xi, \gamma),
\end{equation}
where $E_H$ is given by 
\[
E_H(x_1) = \int_{\Omega_\pi} \nabla_1 H \left(x_1, x \right) \rho(x) dx.
\]
for some function $\rho\in L^1(\Omega)$, and $H$ is the solution to \eqref{eq.Hff}. 
\\[0.2cm]
{\it Step 2: Bound for $III_+$.} Let us denote by $M_\lambda h$ the local Maximal Function of a locally finite measure $\tilde \mu$ for $\lambda > 0$; that is, 
\[
M_\lambda \tilde \mu (x) = \sup_{0 < s < \lambda } \frac{1}{|B_s|}\int_{B_s(x)} d|\tilde\mu|(y),\quad\quad x \in \R^d.
\]
If $\tilde \mu = h\mathscr{L}^d$ we will denote $M_\lambda h$ instead. From standard theory for local Maximal Functions (see \cite{Ste70}) we know that if $h\in BV(\R^d)$ then there exists some set $N$ with $\mathscr{L}^d(N) = 0$ such that 
\begin{equation}
\label{eq.md}
|h(x) - h(y)| \le c_d |x-y| \left(M_\lambda Dh(x) + M_\lambda Dh(y)\right)
\end{equation}
for $x, y \in \R^d\setminus N$ and $|x-y| \leq \lambda$. It is also well known that, for any $p > 1$, 
\begin{equation}
\label{eq.mp}
\left\| M_\lambda h \right\|_{L^p(B_s)} \le c_{d, p} \left\| h \right\|_{L^p(B_{s+\lambda})} ,
\end{equation}
for any $s > 0$ and for a constant $c_{d, p}$ depending only on $d$ and $p$, which blows-up as $p \downarrow 1$. 

Let us now first combine \eqref{eq.md} with \eqref{eq.bdiiip_2}, to get 
\begin{align*}
&\iiint_{\gamma_1^1(t)\xi_1^1(t) \ge 0 } \int_{\Omega_\pi} \zeta \frac{|\nabla_1H\left(\psi(\gamma^1(t)), x\right) - \nabla_1H\left(\psi(\xi^1(t)), x\right)|}{ |\gamma^1(t)-\xi^1(t)|} \rho(x)\,dx\, d\mu \leq \\
&~~~~~\zeta \int_{\Omega_\pi} \iiint_{\gamma_1^1(t)\xi_1^1(t) \ge 0 }  |M_\lambda D^2_1H\left(\psi(\gamma^1(t)), x\right) + M_\lambda D^2_1H\left(\psi(\xi^1(t)), x\right)|\,d\mu\,\rho(x) dx \le\\
& ~~~~~~~~~~~~~~~~~~~~~~~~~~~~~~~~~~~~~~~~~~~~~~~~~~~~~~~~\le C\zeta \int_{\Omega_\pi} \int_{B_r} |M_\lambda D^2_1H\left(z, x\right)| \,dz \rho(x) \,dx,
\end{align*}
where in the last inequality we are using the no-concentration condition, $(e_t)_{\#} \eta \leq C_0\left(\mathscr{L}^{2d} \mres A\right)$, we are taking $\lambda$ small independently of the other parameters (say, for example, $\lambda = 1/8$), and we changed variables while using that the determinant of the Jacobian is bounded, \eqref{eq.cd}. Now, thanks to \eqref{eq.mp}, we can bound 
\begin{equation}
\label{eq.boundH}
\zeta \int_{\Omega_\pi} \int_{B_r} |M_\lambda D^2_1H\left(z, x\right)| \,dz \rho(x) \,dx \leq C_p\zeta \int_{\Omega_\pi} \|D_1^2 H(\cdot, x)\|_{L^p(B_1)} \rho(x) \,dx.
\end{equation}

Therefore, in order to complete the bound, it is enough to show that $H(\cdot, x) \in W^{2, p}(B_1\cap \Omega)$ uniformly for every $x\in \Omega_\pi$ and for some $p > 1$, since $\rho$ is in $L^1$. But since $H$ solves the Laplace equation \eqref{eq.Hff}, by standard elliptic theory (see, e.g., \cite{Mar87}) it is enough to show that $H(\cdot, x) = \Gamma(\cdot- x)- \Gamma(\cdot- Px)\in W^{2-1/p, p}(\de \Omega)$ uniformly for every $x\in \Omega_\pi$ and for some $p > 1$.

Let us suppose, after a rotation and translation, that $x_0 = (\xi,0,\dots,0)$, $\pi(x_0) = 0$, and $\xi>0$ is small independently of the other parameters. With this setting, $Px_0 = -x_0 = (-\xi,0,\dots,0)$, and from the regularity of the domain $\Omega$ we have that 
\begin{equation}
\label{eq.regdom}\de\Omega\cap B_1 \subset \{|x_1|\le |x_2|^2+\dots|x_d|^2 = |x'|^2\},
\end{equation} 
where we will also be using the notation $x = (x_1, x') \in \R\times\R^{d-1}$. We want 
\begin{equation}
\label{eq.wtp}
\Psi_{x_0} := \Gamma(\cdot- x_0)- \Gamma(\cdot+ x_0) = C_d\left\{ \frac{1}{|\cdot-x_0|^{d-2}} - \frac{1}{|\cdot+x_0|^{d-2}}\right\} \in W^{2-1/p, p}(\de \Omega),
\end{equation}
with a bound independent of $x_0$ and for some $p > 1$.

We start by claiming that, for any $x\in \de\Omega\cap B_1$, and for any $k \in \N$,  
\begin{equation}
\label{eq.claimb}
\left|\frac{1}{|x-x_0|^k} - \frac{1}{|x+x_0|^k} \right|\le \frac{C}{|x+x_0|^{k-1}},
\end{equation}
for some constant $C$ depending only on $d$. Indeed, since $|x-x_0|$ and $|x+x_0|$ are always comparable, we can assume $|x-x_0|\le |x+x_0|$ and compute 
\[
\frac{1}{|x-x_0|^k} - \frac{1}{|x+x_0|^k} \leq C\frac{|x+x_0| - |x-x_0|}{|x-x_0|^{k+1}}\leq C\frac{|x+x_0|^2 - |x-x_0|^2}{|x-x_0|^{k+2}}.
\]

Let us suppose, without loss of generality, that $x = (s^2, s, 0,\dots,0)$, for some $s \in (0, 1)$. Notice that
\[
|x+x_0|^2-|x-x_0|^2 = (s^2+\xi)^2 + s^2 - (s^2-\xi)^2 - s^2 = 4s^2\xi.
\]
Now notice that $s \leq |x-x_0|$ and that $\xi \leq |x-x_0|$, which yields the claim, \eqref{eq.claimb}. 

Thanks to \eqref{eq.claimb} together with \eqref{eq.wtp}, it is clear that $\Psi_{x_0} \in L^p_{\rm loc}(\de\Omega)$ whenever $p<\frac{d-1}{d-3}$.

Let us next prove that $\Psi_{x_0} \in W^{1, p}_{\rm loc}(\de \Omega)$. It is enough to show that tangential derivatives to $\de\Omega$ of $\Psi_{x_0}$ are in $L^p$. Let $D_\tau \Psi_{x_0}(x) = \tau(x)\cdot\nabla \Psi_{x_0}(x)$ denote any tangential derivative, for $x\in \de\Omega$ and $\tau(x)\in S^{d-1}$ tangent to $\de\Omega$ at $x$. From the $C^{1, 1}$ regularity of the domain, it follows that for some constant $C$ depending only on the dimension, $|\tau_1(x)| \le C|x'|$, and therefore, 
\begin{equation}
\label{eq.Dtau}
|D_\tau \Psi_{x_0}(x)| \le C\left( |x'||\de_{x_1} \Psi_{x_0}(x)| +  |\de_{x_2} \Psi_{x_0}(x)| +\dots+|\de_{x_d} \Psi_{x_0}(x)| \right),
\end{equation}
for any tangential derivative. Let us know compute , for $i \in \{2,\dots,d\}$,
\[
|\de_{x_i} \Psi_{x_0}(x)| = C|x_i| \left|\frac{1}{|x-x_0|^d} - \frac{1}{|x+x_0|^d}\right|\leq \frac{C}{|x-x_0|^{d-2}} ,
\]
where we used \eqref{eq.claimb} together with the fact that $|x_i|\le |x-x_0|$.

On the other hand, 
\begin{align*}
|x'||\de_{x_1} \Psi_{x_0}(x)|& = |x'|\left(\frac{x_1-\xi}{|x-x_0|^{d}} - \frac{x_1+\xi}{|x+x_0|^{d}} \right)\\
& = x_1|x'|\left(\frac{1}{|x-x_0|^{d}} - \frac{1}{|x+x_0|^{d}} \right) - \frac{2\xi|x'|}{|x-x_0|^d} \leq \frac{C}{|x-x_0|^{d-2}},
\end{align*}
where we combined again \eqref{eq.claimb} with $\xi\le |x-x_0|$ and $|x'|\le |x-x_0|$. In all, we have that , putting it back in \eqref{eq.Dtau},
\begin{equation}
\label{eq.Dtau_b}
|D_\tau \Psi_{x_0}(x)| \le  \frac{C}{|x-x_0|^{d-2}} \le  \frac{C}{|x|^{d-2}}\in L^p(\de\Omega), \textrm{ for } 1\le p < \frac{d-1}{d-2},
\end{equation}
and thus, $\Psi_{x_0}\in W^{1, p}_{\rm loc}(\de \Omega)$ for $p < \frac{d-1}{d-2}$.

Let us denote $D^2_{\tau_1\tau_2}$ second derivatives along $\de\Omega$, for $\tau_1, \tau_2\in S^{d-1}$ tangent vectors to $\de\Omega$. Similarly to derivation of \eqref{eq.Dtau_b} it follows that 
\[
|D^2_{\tau_1\tau_2} \Psi_{x_0}(x)| \le  \frac{C}{|x-x_0|^{d-1}}\le \frac{C}{|x|^{d-1}},
\]
which, unfortunately, does not belong to any $L^p$ space for $p > 1$ in $\de\Omega$.

We define $F_\tau^{x_0}(x) := |x|D_\tau\Psi_{x_0}$, so that using the previous inequalities one can show 
\[
\frac{C}{|x|^{d-3}} \ge |F_\tau^{x_0}| \in L^{r_1}(\de\Omega)\cap W^{1, r_2}(\de\Omega) \subset W^{s, r_s}(\de\Omega),\textrm{ for } r_1 < \frac{d-1}{d-3}, r_2 < \frac{d-1}{d-2},
\]
where $s$ and $r_s$ follow from the interpolation property between fractional Sovolev spaces (see \cite[Theorem 6.4.5]{BL76}) and fulfill
\[
r_s = \left(\frac{1-s}{r_1}+ \frac{s}{r_2}\right)^{-1}<\frac{d-1}{d-3+s}.
\]

On the other hand, it is also known that if $h_1\in W^{s, p_1}\cap L^{q_1}(\de\Omega\cap B_1)$, and $h_2\in W^{s, q_2}\cap L^{p_2}(\de\Omega\cap B_1)$, then $h_1h_2\in W^{s, p}(\de\Omega\cap B_{3/4})$, with $s\in (0, 1)$ and $1>\frac{1}{p} = \frac{1}{p_1} + \frac{1}{p_2} = \frac{1}{q_1} + \frac{1}{q_2}>0$; with a bound 
\begin{equation}
\label{eq.prodrule}\left\|h_1h_2\right\|_{W^{s, p}(\de\Omega\cap B_{3/4})} \le C\left(  \left\| h_1 \right\|_{W^{s, p_1}(\de\Omega\cap B_1)}   \left\|h_2\right\|_{L^{p_2}(\de\Omega\cap B_1)} +  \left\|h_1\right\|_{L^{q_1}(\de\Omega\cap B_1)}\left\| h_2 \right\|_{W^{s, q_2}(\de\Omega\cap B_1)}   \right) .
\end{equation}
This result is a local version of the Runst--Sickel lemma, that can be found, for example, in \cite[Section 5.3.7]{RS96} or in \cite[Theorem 3]{Gat02}. We use \eqref{eq.prodrule} with $h_1 = F_\tau^{x_0}$ and $h_2 = \frac{1}{|x|} \in W^{t, r_t}_{\rm loc} (\de\Omega)$ for $(1+t)r_t <d-1$, and therefore $h_1h_2 = D_\tau \Psi_{x_0}$. Putting all together in \eqref{eq.prodrule} we should have 
\[
p_1 < \frac{d-1}{d-3+s},~~~~p_2 < d-1,~~~~ q_1 < \frac{d-1}{d-3}, ~~~~ q_2 < \frac{d-1}{1+s},
\]
so that
\[
\frac{1}{p} > \frac{d-2+s}{d-1}.
\]

If we want $s = 1-\frac{1}{p}$ we must have $p < \frac{d}{d-1}$, and in this case we have a bound for $\|D_\tau \Psi_{x_0}\|_{W^{1-1/p, p}(\de\Omega\cap B_1)}$ independent of $x_0$. Therefore, 
\[
\|\Psi_{x_0}\|_{W^{2-1/p, p}(\de\Omega\cap B_1)}\leq C,
\]
from which $H(\cdot, x) \in W^{2, p}(B_1\cap \Omega)$ and \eqref{eq.boundH} can be bounded by $C\zeta$. We have, therefore, shown that a bound of the type \eqref{eq.enoughuniq_2} holds for the term $III_+$.
\\[0.2cm]
{\it Step 3: Bound for $III_-$.} We finally have to bound the remaining term, $III_-$, 
\[
III_- = \iiint_{\gamma_1^1(t)\xi_1^1(t) < 0 } \zeta \frac{|\tilde E(\gamma^1(t)) - \tilde E(\xi^1(t))|}{\zeta\delta + |\gamma^1(t)-\xi^1(t)|} d\mu(x, \xi, \gamma).
\]

As in the proof of Theorem~\ref{thm.charact}, thanks to the symmetries of the vector field $\tilde E$ and the fact that now $\gamma^1(t)$ and $\xi^1(t)$ are on opposite sides, we are only required to bound 
\[
\iiint_{\gamma_1^1(t)\xi_1^1(t) < 0 } \zeta \frac{|\tilde E_1(\gamma^1(t)) |}{\zeta\delta + |\gamma^1_1(t)|} d\mu\leq C\zeta\int_{B_r} \frac{|((JE)_1\circ \psi)(y)|}{\zeta\delta + |y_1|} dy,
\]
where we have used again the no-concentration condition no-concentration condition, $(e_t)_{\#} \eta \leq C_0\left(\mathscr{L}^{2d} \mres A\right)$, and where the subindex 1 denotes the first coordinate of the vector field. 

If we define $A_k := [0, 2^{-k}]\times B_r^{(d-1)}$ as in the proof of Theorem~\ref{thm.charact}, and we recall that $\nabla_1$ is the gradient in the first $d$ components, we just have to bound 
\begin{align*}
\|(JE)_1\circ \psi\|_{L^1(A_k, \mathscr{L}^d)}& = \int_{A_k} |(JE)_1\circ \psi)(y)| dy\\
&\le \int_{A_k} \int_\Omega \left|\left\{(J(\psi(y))\nabla_1 G_\Omega(\psi(y), z)\right\}_1\rho(z)\right|\,dz\,dy\\
& =   \int_\Omega \rho(z) \int_{A_k}\left| e_1\cdot (J(\psi(y)) \nabla_1 G_\Omega(\psi(y), z)\right| \,dy\,dz.
\end{align*}
Since $\psi(y)$ is close to the boundary, $ J^T(\psi(y)) e_1 = \nabla_x \phi_1(\psi(y)) = n(\psi(y)) $, and thus, $n(\psi(y))\cdot \nabla_1$ is the derivative in the direction normal to $\de\Omega$, which we will denote $\de_n^1$. That is,  
\[
\|(JE)_1\circ \psi\|_{L^1(A_k, \mathscr{L}^d)}\le  \int_\Omega \rho(z) \int_{A_k}\left| \de_n^1 G_\Omega(\psi(y), z)\right| \,dy\,dz.
\]
Following as in \eqref{eq.charact1} and \eqref{eq.charact2} from Theorem~\ref{thm.charact}, it is enough to show
\begin{equation}
\label{eq.charactenough}
\|(JE)_1\circ \psi\|_{L^1(A_k, \mathscr{L}^d)} \le C2^{-k}. 
\end{equation}
Let us actually prove that, for any $z \in \Omega$, 
\begin{equation}
\label{eq.charactenough2}
\int_{A_k}\left| \de_n^1 G_\Omega(\psi(y), z)\right| \,dy \le C2^{-k}.
\end{equation}
Notice that if $z$ is far from $\psi(y)$, from the regularity of the Green function the previous result follows immediately. Thus, we can assume that $z$ is close to $\psi(y)$, and for $k$ large, in particular, $z, \psi(y)\in \Omega_\pi$.

Let us start by bounding, for $x \in \Omega_\pi$, $\de_{n}^1 G_\Omega(x, z)$. Keeping the notation from Step 2, we write the Green function as 
\[
G_\Omega(x_1, x_2) = \Gamma(x_1 - x_2) - \Gamma(Px_1 - x_2) - \tilde H(x_1, x_2),
\]
where now $\tilde H(x_1, x_2)$ fulfills 
\begin{equation}
\label{eq.Hff2}
\left\{ \begin{array}{ll}
  \Delta_{x_2} H(x_1, x_2) = 0, & \textrm{ for } x_2 \in \Omega \\
  H(x_1, x_2) = \Gamma(x_1 - x_2) - \Gamma(x_1 - Px_2) & \textrm{ for } x_2 \in \de \Omega.\\
  \end{array}\right.
\end{equation}

In particular, we can compute $\de_{n}^1 G_\Omega(x, z)$ as 
\begin{equation}
\label{eq.defiG}
\de_{n}^1 G_\Omega(x, z) = \frac{n(x)\cdot (x- z)}{|x-z|^d} - \frac{n(x)\cdot (Px- z)}{|Px-z|^d} - H_n(x, z),
\end{equation}
where $H_n$ solves 
\begin{equation}
\label{eq.Hff3}
\left\{ \begin{array}{ll}
  \Delta_{z} H_n(x, z) = 0, & \textrm{ for } z \in \Omega \\
  H_n(x, z) = \frac{n(x) \cdot (x-z)}{|x-z|^d} - \frac{n(x) \cdot (Px-z)}{|Px-z|^d}& \textrm{ for } z \in \de \Omega.\\
  \end{array}\right.
\end{equation}
Notice, on the one hand, we have $|x-z|, |Px - z| \ge |x|$ and on the other hand, $|n(x) \cdot (x-z)|, |n(x)\cdot (Px-z)| \le 2\delta(x)$ for $z \in \de\Omega\cap B_1$ due to the regularity of the domain $\de\Omega$. Here, and it what comes, $\delta(x) := \textrm{dist}(x, \de\Omega)$. By the maximum principle, therefore, we have
\begin{equation}
\label{eq.Hn}
|H_n(x, z)| \le C\frac{\delta(x)}{|x|^d}.
\end{equation}

Now, using the properties of the change of variables together with the fact that $\psi(y)$ is close to the boundary, we have that 
\[
|H_n(\psi(y), z)| \le C\frac{|y_1|}{|y|^d}.
\]

On the other hand, if we denote $\phi(z) = \tilde z$, we have that $|\psi(y) - z| = |\psi(y) - \psi(\tilde z)| \ge c |y - \tilde z|$, and using \eqref{eq.jacob2} and the closeness of $z$ to $\psi(y)$ we also have $|n(\psi(y))\cdot (\psi(y) - \psi(\tilde z))| \le C|(y - \tilde z)\cdot e_1|$. Putting all together in \eqref{eq.defiG}, and using the analogous strategy for the term containing $Px - z$, we get 
\[
|\de_{n}^1 G_\Omega(\psi(y), z) |\le \sup_{\tilde z \in A_k} C\frac{|y_1 - \tilde z_1|}{|y-\tilde z|^d},
\]
and as in Theorem~\ref{thm.charact} we have, maybe for a bigger constant, 
\[
\int_{A_k}\left| \de_n^1 G_\Omega(\psi(y), z)\right| \le C \int_{A_k} \frac{|y_1|}{|y|^d}\, dy\le C2^{-k}.
\]
The last inequality now follows from \eqref{eq.K1ineq}. This proves the bound \eqref{eq.enoughuniq_2} following as in \eqref{eq.charact2} for the remaining term $III_-$; and the theorem is proved. 

\subsection*{Proof of Step 6 Theorem~\ref{thm.main2_d}}
We divide it into two parts, the proof of \eqref{eq.1stab} and \eqref{eq.2stab}.
\\[0.2cm]
{\it Step 1: Proof of \eqref{eq.1stab}:} We start by proving \eqref{eq.1stab}. Notice that the weak convergence in $L^1$ is not affected by the odd reflection and the change of variables. It is enough to show that
\[
E^n\rightharpoonup E\quad\textrm{weakly in }~L^1_{\rm loc}((0, \infty)\times (\Omega\cap B_2); \R^d),
\]
that is, it is enough to show that 
\begin{equation}
\label{eq.lim1stab}
\lim_{n\to \infty} \left|\int_0^\infty\int_{\Omega} (E_t^n-E_t)\varphi\,dx\,dt \right| = 0,\quad\textrm{ for all }\varphi \in L^\infty_c((0, \infty)\times\overline{\Omega\cap B_2}),
\end{equation}
where $L^\infty_c$ denotes the set of bounded functions with compact support. Let
\[
\left|\int_0^\infty\int_{\Omega} (E_t^n-E_t)\varphi\,dx\,dt \right|\le I_n+II_n+III_n,
\]
with 
\[
I_n = \left|\int_0^\infty\int_\Omega \left(\int_\Omega\nabla_x G_\Omega(x, z)[\rho^n_t(z)-\rho_t(z)]\,dz\right)\varphi\,dx\,dt\right|,
\]
\[
II_n = \left|\int_0^\infty\int_\Omega \left(\int_\Omega\left[\nabla_x G_\Omega(x, z)-\nabla_x G_\Omega^{\delta_n}(x, z)\right]\rho^n_t(z)\,dz\right)\varphi\,dx\,dt\right|,
\]
and 
\[
III_n = \left|\int_0^\infty\int_\Omega \left(1-\bar r_\Omega^{\zeta_n}(x)\right)\left(\int_\Omega\nabla_x G_\Omega^{\delta_n}(x, z)\rho^n_t(z)\,dz\right)\varphi\,dx\,dt\right|.
\]
We start with $I_n$. By Fubini we have 
\[
I_n = \left|\int_0^\infty\int_\Omega \left(\int_\Omega\nabla_x G_\Omega(x, z)\varphi(x)\, dx \right)[\rho^n_t(z)-\rho_t(z)]\,dz\,dt\right|.
\]
Now notice that, by the bounds on the Green function,
\[
\left|\int_\Omega\nabla_x G_\Omega(x, z)\varphi(x)\, dx \right|\leq \|\varphi\|_{L^\infty}\int_{B_2} \frac{dx}{|x|^{d-1}}\leq C\|\varphi\|_{L^\infty},
\]
for some constant $C$ depending only on the dimension $d$. Thus, we can use the weak convergence \eqref{eq.convdens_d} to get that $I_n \to 0$ as $n\to \infty$.

On the other hand, we have
\begin{equation}
\label{eq.Gdeltabound}
|\nabla_x G^\delta_\Omega(x, z)|  \le \frac{1}{\delta} \bar r '\left(\frac{|x-z|}{\delta}\right)|G_\Omega(x, z)| + \bar r \left(\frac{|x-z|}{\delta}\right)|\nabla_x G_\Omega(x, z)|.
\end{equation}
From the bounds on the Green function $G_\Omega$ and the definition of $\bar r$ it is easy to check that, in particular, we have 
\begin{equation}
\label{eq.Gdeltabound2}
|\nabla_x G^\delta_\Omega(x, z)|  \le  C|x-z|^{1-d},
\end{equation}
for some $C$ depending only on the dimension $d$ and $\Omega$, but independent of $\delta$.

Let us now bound $III_n$. We denote $\Omega_{\zeta} := \{ x\in \Omega : {\rm dist}(x, \de\Omega) <\zeta\}$. By Fubini, the bound \eqref{eq.Gdeltabound2}, and the definition of $\bar r_\Omega^{\zeta_n}$, we have
\[
III_n\le  C\|\varphi\|_{L^\infty} \left|\int_0^\infty\int_\Omega \left(\int_{\Omega_{2\zeta_n}\cap B_2}|x-z|^{1-d}\,dx\right)\rho^n_t(z)\,dz\,dt\right|.
\]
Now, since $\rho_t^n$ has $L^1$ norm bounded independently of $n$ and $t$, $|x|^{1-d}$ is locally integrable and $|\Omega_{2\zeta_n}\cap B_2|\to 0$ as $n \to \infty$, we have that $III_n \to 0 $ as $n \to \infty$. 

We can finally bound $II_n$. Proceeding as in \eqref{eq.Gdeltabound}
\[
\begin{aligned}
II_n \le \int_0^\infty\int_\Omega&  \left(\int_\Omega\left|1-\bar r\left(\frac{|x-z|}{\delta_n}\right)\right|\cdot |\nabla_x G_\Omega(x, z)| \rho^n_t(z)\,dz\right)|\varphi|\,dx\,dt,\\
& + \int_0^\infty\int_\Omega  \left(\int_\Omega\frac{1}{\delta_n} \bar r'\left(\frac{|x-z|}{\delta_n}\right) | G_\Omega(x, z)| \rho^n_t(z)\,dz\right)|\varphi|\,dx\,dt.
\end{aligned}
\]

Using Fubini, the bounds on $G_\Omega$, and the definition of $\bar r$ we obtain 
\[
\begin{aligned}
II_n & \le C\|\varphi\|_{L^\infty}  \int_0^\infty\int_\Omega  \left(\int_{B_{2\delta_n}(z)} |x-z|^{1-d}\,dx \right)\rho_t^n(z)\,dz\,dt, + \\
& ~~~~~~~~~~+C\|\varphi\|_{L^\infty} \delta_n^{-1} \int_0^\infty\int_\Omega  \left(\int_{B_{2\delta_n}(z)} |x-z|^{2-d}\,dx \right)\rho_t^n(z)\,dz\,dt\\
& \le C\|\varphi\|_{L^\infty} \|\rho_t^n\|_{L^1} \int_{B_{2\delta_n}} \left(|x|^{1-d} + \delta_n^{-1}|x|^{2-d}\right)\,dx = C\|\varphi\|_{L^\infty} \|\rho_t^n\|_{L^1} \delta_n,
\end{aligned}
\]
therefore, since $\delta_n \to 0$, $II_n\to 0$ as $n\to \infty$. This proves \eqref{eq.lim1stab}, as we wanted to see.
\\[0.2cm]
{\it Step 2: Proof of \eqref{eq.2stab}.} Let us now show that the second stability condition, \eqref{eq.2stab}, holds. 

As in Theorem~\ref{thm.mainws} it is enough to show that 
\begin{equation}
\label{eq.enough_dom}
[F_2^{n,o}]_{W^{\alpha, p}(B_{1/2})} \le C,
\end{equation}
for some $C$ independent of $n$ and $t$, and for some $\alpha > 0$, $p > 1$. Notice that $(F_2^{n, o})_i(y)$ can be expressed as the sum of two terms in each component $i\in \{1,\dots,n\}$; namely, $(F_2^n)_i(y) + (F_2^n)_i(y')$ for $2\le i \le n$ and $(F_2^n)_1(y) - (F_2^n)_1(y')$. Now, by the subadditivity of the seminorm, if one can show
\begin{equation}
\label{eq.enough_dom_2}
[F_2^{n}]_{W^{\alpha, p}(B_{1/2})} \le C,
\end{equation} 
we are done. Notice that we are considering the natural extension of the vector field to $y_1 < 0$ by $0$.

We are going to use the product rule result (\cite[Section 5.3.7]{RS96} or  \cite[Theorem 3]{Gat02}) that already appeared in \eqref{eq.prodrule} several times throughout the proof. Let us restate it here:

If $h_1\in W^{s, p_1}\cap L^{q_1}(B_1)$, and $h_2\in W^{s, q_2}\cap L^{p_2}(B_1)$, then $h_1h_2\in W^{s, p}( B_{3/4})$, with $s\in (0, 1)$ and $1>\frac{1}{p} = \frac{1}{p_1} + \frac{1}{p_2} = \frac{1}{q_1} + \frac{1}{q_2}>0$; with a bound
\begin{equation}
\label{eq.prodrule_2}\left\|h_1h_2\right\|_{W^{s, p}( B_{3/4})} \le C\left(  \left\| h_1 \right\|_{W^{s, p_1}( B_1)}   \left\|h_2\right\|_{L^{p_2}( B_1)} +  \left\|h_1\right\|_{L^{q_1}(B_1)}\left\| h_2 \right\|_{W^{s, q_2}( B_1)}   \right) .
\end{equation}

Thanks to this result and the regularity of the Jacobian matrix $J$, it immediately follows that it is enough to bound the $W^{\alpha, p}(B_{3/4})$ norm of $E^n(\psi(y))$ for some $\alpha >0$, $p > 1$. We recall 
\[
E^n(\psi(y)) = E^{\zeta_n, \delta_n}_\Omega(y) = -\bar r^{\zeta_n}_\Omega(\psi(y)) \int_\Omega \nabla_x G^{\delta_n}_\Omega(\psi(y), z)\rho_t^n(z)\, dz =:  -\bar r^{\zeta_n}_\Omega(\psi(y)) \bar E^{\delta_n}(y).
\]

By definition of the change of variables, notice that $\bar r_\Omega^{\zeta_n}(\psi(y)) = \bar r(\zeta_n^{-1} y_1)$, and it is a sequence in $n$ approximating the Heaviside step function in the direction $y_1$. In particular, it is easy to check that 
\[
[\bar r(\zeta_n^{-1}\cdot) ]_{W^{\bar s, \bar p}(B_{3/4})} \le C,
\]
for some $C$ independent of $n$, provided $\bar s \bar p < 1$. On the other hand, by Young's inequality and \eqref{eq.Gdeltabound2} it is easy to check that $\bar E^{\delta_n}\in L^\gamma_{\rm loc}$ for $\gamma < \frac{d}{d-1}$. Thus, putting all together and using \eqref{eq.prodrule_2}, we can check that it will be enough to bound $[\bar E^{\delta_n}]_{W^{s, p}(B_{3/4})}$ for some $0 <s <\frac{1}{d}$ and $1 < p < \frac{d}{d-1+sd}$.

A simple computation shows that, if $H(y) = \int_\Omega K(y, z)\rho(z) \,dz$, then 
\begin{equation}
\label{eq.simplebound}
[H]_{W^{s, p}(B)} \le \|\rho\|_{L^p(\Omega)}\, \sup_{z\in \Omega}\,[K(\cdot, z)]_{W^{s, p}(B)}.
\end{equation}
Indeed, using H\"older's inequality and Fubini's theorem,
\begin{align*}
[H]^p_{W^{s, p}(B)}& \le \int_{B\times B} \frac{\left|\int_\Omega |K(y_1, z)-K(y_2, z)|\rho(z)\, dz\right|^p}{|y_1 - y_2|^{d+sp}}\, dy_1\,dy_2\\
& \le \int_{B\times B} \|\rho\|^{p-1}_{L^p(\Omega)} \frac{\int_\Omega |K(y_1, z)-K(y_2, z)|^p \rho(z)\, dz}{|y_1 - y_2|^{d+sp}}\, dy_1\,dy_2 \le \|\rho\|^p_{L^p(\Omega)}\, \sup_{z\in \Omega}\,[K(\cdot, z)]^p_{W^{s, p}(B)}.
\end{align*}

Notice that 
\begin{equation}
\label{eq.barEdeltan}
\begin{aligned}
\bar E^{\delta_n}(y) & = \int_\Omega \bar r\left(\frac{|\psi(y)-z|}{\delta_n}\right)\nabla_x G_\Omega(\psi(y), z) \rho_t^n(z)\, dz+\\
&~~~~+\int_\Omega \frac{1}{\delta_n}\bar r '\left(\frac{|\psi(y)-z|}{\delta_n}\right)\frac{\psi(y)-z}{|\psi(y)-z|}G_\Omega(\psi(y), z) \rho_t^n(z)\, dz.
\end{aligned}
\end{equation}
We start by taking $K(y, z) = \bar r\left(\frac{|\psi(y)-z|}{\delta_n}\right)\nabla_x G_\Omega(\psi(y), z)$ in \eqref{eq.simplebound}. Using again \eqref{eq.prodrule_2} and noticing that the first term involving $\bar r$ belongs to $W^{\bar s, \bar p}(B_1)$ independently of $n$ for $\bar s\bar p < 1$ (as before for the approximations of the Heaviside function), we need to bound 
\[
[\nabla_x G_\Omega(\psi(\cdot), z)]_{W^{s, p}(B_{1})}
\] 
independently of $z$, for some $s > 0$ and $p > 1$. By the equivalence of Sobolev and Besov spaces for fractional order derivatives we have 
\begin{align*}
[\nabla_x  G_\Omega(\psi(\cdot), z)]_{W^{s, p}(B_1)} &  \le C[\nabla_x G_\Omega(\psi(\cdot), z)]_{B^{s}_{p, p}(B_1)}\\
& \le C\left(\int_0^1 t^{-sp-1}\sup_{|h|\le t} \left[\int_{B_1} |\nabla_x G_\Omega(\psi(y+h), z)- \nabla_x G_\Omega(\psi(y), z)|^p\,dy\right] dt\right)^{\frac{1}{p}}.
\end{align*}
Now, thanks to Lemma~\ref{lem.GreenFct} (vi) and the bound on the Jacobian of the change of variables, for any $\alpha \in (0, 1)$ there is a constant $C$ depending only on $d$, $\Omega$, and $\alpha$, such that
\begin{align*}
[\nabla_x & G_\Omega(\psi(\cdot), z)]_{W^{s, p}(B_1)}  \le \\
& \le C\left(\int_0^1 t^{-sp-1} t^{p\alpha}\sup_{|h|\le t} \left[\int_{B_1} \left(|\psi(y+h)-z|^{1-d-\alpha}+|\psi(y)-z|^{1-d-\alpha}\right)^p\,dy\right] dt\right)^{\frac{1}{p}}\\
& \le C\left(\int_0^1 t^{(\alpha-s)p-1} \left[\int_{B_1} |y|^{p-pd-p\alpha}\,dy\right] dt\right)^{\frac{1}{p}},
\end{align*}
which is going to be bounded if $1 > \alpha > s$ and $p < \frac{d}{d+\alpha - 1} < \frac{d}{d+s-1}$.

On the other hand, we can also take $K(y, z) = \frac{1}{\delta_n}\bar r '\left(\frac{|\psi(y)-z|}{\delta_n}\right)\frac{\psi(y)-z}{|\psi(y)-z|}G_\Omega(\psi(y), z)$ to complete the bound from \eqref{eq.barEdeltan}. In this case, proceeding via the Besov seminorm as before, and by means of the triangular inequality, we obtain
\begin{align*}
&[K (\cdot, z)]_{W^{s, p}(B_1)} \le \\
&~\le C\delta_n^{-1} \bigg(\int_0^1 t^{-sp-1}\sup_{|h|\le t} \bigg[\int_{B_1} \bigg|\bar r '\bigg(\frac{|\psi(y+h)-z|}{\delta_n}\bigg) - \bar r '\bigg(\frac{|\psi(y)-z|}{\delta_n}\bigg) \bigg|^p |G_\Omega(\psi(y+h), z)|^p\,dy\bigg] dt\bigg)^{\frac{1}{p}}\\
& ~~~~~+ C\delta_n^{-1} \bigg(\int_0^1 t^{-sp-1}\sup_{|h|\le t} \bigg[\int_{B_1} \bigg| \bar r '\bigg(\frac{|\psi(y)-z|}{\delta_n}\bigg) \bigg|^p |G_\Omega(\psi(y+h), z)-G_\Omega(\psi(y), z)|^p\,dy\bigg] dt\bigg)^{\frac{1}{p}}\\
&~= I_n + II_n.
\end{align*}
In order to bound $I_n$ we start by noticing that $\bar r '$ is smooth and, in particular, $C^\beta$ for $\beta\in [0, 1]$. For a $\beta$ to  determined, and using the bounds on the Green function, we have 
\begin{align*}
I_n & \le C\delta_n^{-1} \left(\int_0^1 t^{-sp-1}\sup_{|h|\le t} |h|^{p\beta}\delta_n^{-p\beta}\int_{D_{\delta_n}} |G_\Omega(\psi(y), z)|^p\,\,dy\, dt\right)^{\frac{1}{p}}\\ 
& \le C\delta_n^{-1-\beta} \left(\int_0^1 t^{(\beta-s)p-1}\int_{B_{2\delta_n}}|y|^{-dp+2p}\,dy\, dt\right)^{\frac{1}{p}}\le C\delta_n^{1-d+d/p-\beta}\left(\int_0^1t^{(\beta-s)p-1}\, dt\right)^{\frac1p},
\end{align*}
where $D_{\delta_n}$ denotes the set where $\bar r '\left(\delta_n^{-1}|\psi(y+h)-z|\right) - \bar r '\left(\delta_n^{-1}|\psi(y)-z|\right) $ is non-zero, and where we are using all the time that the change of variables $\psi$ is regular. If we want this last term to be bounded we need $\beta > s$ and $1-d+d/p-\beta \ge 0$; that is, $p\le \frac{d}{d-1+\beta}$. By choosing $s<\beta<1$ we can choose $p > 1$ and we are done.

Using a similar method, in this case via Lemma~\ref{lem.GreenFct} (v), we can bound $II_n$ as
\begin{align*}
II_n & \le C\delta_n^{-1} \left(\int_0^1 t^{-sp-1}\sup_{|h|\le t} |h|^{p \alpha } \int_{B_{2\delta_n}} |y|^{(2-d-\alpha)p}\,\,dy\, dt\right)^{\frac{1}{p}}\\
& \le C\delta_n^{1-d-\alpha+d/p} \left(\int_0^1 t^{(\alpha-s)p-1}\, dt\right)^{\frac{1}{p}},
 \end{align*}
 which is going to be bounded if $\alpha >s $ and $p <\frac{d}{d-1+\alpha}$. As before, it is enough for us to choose $s < \alpha < 1$ to get the desired result. This completes the proof of the second stability condition, \eqref{eq.2stab}.

\end{document}